\documentclass[preprint,11pt,a4paper,leqno]{elsarticle}

\usepackage{amssymb}
\usepackage{amsmath}
\usepackage{mathrsfs}
\usepackage{amsthm}
\usepackage{amsfonts}
\usepackage{enumerate}
\usepackage{tikz}
\usetikzlibrary{arrows.meta}
\usepackage{microtype}

\allowdisplaybreaks
\journal{Advances in Mathematics}
\biboptions{sort&compress}

\def\dfrac{\displaystyle\frac}

\newtheorem{theorem}{Theorem}[section]
\newtheorem{lemma}[theorem]{Lemma}
\newtheorem{corollary}[theorem]{Corollary}
\newtheorem{proposition}[theorem]{Proposition}
\theoremstyle{definition}
\newtheorem{remark}[theorem]{Remark}

\numberwithin{equation}{section}

\begin{document}

\arraycolsep=1pt

\begin{frontmatter}

\title{Optimal multipolar $L^p$-Hardy inequalities:\\
ground states, criticality and variational thresholds\tnoteref{funding}}
\tnotetext[funding]{Funding information: NNSF of China (12071431).}

\author{Yongyang Jin}
\author{Shoufeng Shen}
\author{Li Tang\corref{corresponding}}
\cortext[corresponding]{Corresponding author.}

\begin{abstract}
We study ground states, criticality, and variational compactness for genuinely multipolar $L^p$-Hardy weights, $1<p<n$. A ground-state representation across the poles gives a complete criticality classification of the full potential and characterizes finite-energy attainment. Removing a geometric interaction term produces a variational problem governed by concentration at the poles and infinity. A strict concentration gap yields compactness, positive extremizers, and criticality of the optimal operator; for $p\ge2$, it is also necessary for attainment. In the subquadratic range, the optimal truncated coefficient is below one and attained throughout the full-potential critical set. For admissible parameters with $1<p<2$ and $(n-p)/p<s\le m(n-p)/p$, arbitrary fixed pole configurations exhibit negative directions, attained critical states, and coercive subcritical regimes outside this set along positive-weight paths towards the simplex vertices. A single-pole degeneration formula and continuity in the weights establish this coexistence. In the superquadratic range, an interaction threshold determines a non-attained infinity-concentration interval. Under strict separation from pole concentration, its endpoint is critical with infinite energy and a compact phase follows. We prove separation for every bipolar configuration and in several multipolar regimes, including a vertex neighborhood uniform in the exponent. Near parameters strictly inside the quadratic attained region, the optimal coefficient crosses one continuously through $p=2$. The proofs combine ground-state remainders, capacity, weighted compactness, and critical Hardy barriers.
\end{abstract}


\end{frontmatter}

\section{Introduction and main results}

A sharp Hardy coefficient, criticality of the associated operator, and attainment in the homogeneous energy space are distinct properties. In multipolar problems their relation is governed by concentration at the poles, loss of mass at infinity, and interaction between singularities. We study these mechanisms for an explicit family of weights throughout $1<p<n$.

The classical $L^p$-Hardy inequality provides the basic model:
\[
\int_{\mathbb R^n}|\nabla u|^p\,dx
\ge \left(\frac{n-p}{p}\right)^p
\int_{\mathbb R^n}\frac{|u|^p}{|x|^p}\,dx,
\qquad u\in C_c^\infty(\mathbb R^n).
\]
Its coefficient is sharp but is not attained in $D^{1,p}(\mathbb R^n)$.
For its role in singular elliptic equations and spectral theory, see
\cite{D,GP,DH,GM,HY}; quadratic multipolar inequalities are studied in
\cite{FMT,FT,BDE,CZ,CH}, and the optimal-weight and criticality viewpoint
in \cite{DFP}.

For the nonlinear problem, Cazacu and Rugin\u{a} \cite[Theorem~2.2]{CR24} considered two distinct poles $a_1,a_2\in\mathbb R^n$, $n\ge3$. Writing $\bar a=(a_1+a_2)/2$, define
\[
\begin{aligned}
\mathcal V_{1,p}(x)
&=\frac{|a_1-a_2|^2|x-\bar a|^{p-2}}
{|x-a_1|^p|x-a_2|^p},\\
\mathcal V_{2,p}(x)
&=\frac{|x-\bar a|^{p-4}
\bigl(|x-a_1|^2|x-a_2|^2-((x-a_1)\cdot(x-a_2))^2\bigr)}
{|x-a_1|^p|x-a_2|^p}.
\end{aligned}
\]
They proved that, for $1<p<n$,
\[
\begin{aligned}
\int_{\mathbb R^n}|\nabla u|^p\,dx
&\ge \frac{p-1}{4}\left(\frac{n-p}{p-1}\right)^p
\int_{\mathbb R^n}\mathcal V_{1,p}|u|^p\,dx\\
&\quad+\frac{p-2}{2}\left(\frac{n-p}{p-1}\right)^{p-1}
\int_{\mathbb R^n}\mathcal V_{2,p}|u|^p\,dx
\end{aligned}
\]
for every $u\in C_c^\infty(\mathbb R^n)$. The overall coefficient $1$ multiplying the right-hand side is optimal for $2\le p<n$. For $2<p<n$, the extremizers in $D^{1,p}(\mathbb R^n)$ are the scalar multiples of
\[
\phi(x)=|x-a_1|^{-\frac{n-p}{2(p-1)}}
|x-a_2|^{-\frac{n-p}{2(p-1)}}.
\]
The quadratic endpoint is not attained. In the parametrization below, this is the symmetric choice $m=2$, $\mu_1=\mu_2=1/2$, and $s=(n-p)/(p-1)$. For $2<p<n$, deleting the second, nonnegative term also yields a subcritical operator \cite[Corollary~2.5]{CR24}. Thus the distinction between full and truncated potentials already arises in the bipolar setting; here we study it with variable pole weights and throughout $1<p<n$.

Multipolar inequalities on Riemannian manifolds were developed in \cite{FFK,CiR25}; related weighted and non-Euclidean constructions appear in
\cite{CP,CPT,GKY,J,JS,K1,NZW,Y1,Y2,YKG,JTHL,JST}.
Our contribution is a criticality and variational framework, rather
than only an extension of a bipolar inequality to more poles.
For the full potential, a representation across the singularities
gives the exact criticality and finite-energy regions. For the truncated
potential, a concentration threshold governs compactness and, for
$p\ge2$, characterizes attainment. Degeneration towards a single pole
then gives three subquadratic regimes outside the full critical set.
In the superquadratic case, a second threshold separates concentration
at infinity from an attained phase; we establish endpoint criticality
under strict separation and prove that separation in several geometric
regimes. The transition through $p=2$ reflects the sign change of the
removed interaction.

\subsection{Ground-state structure and criticality geometry}

Fix $n\ge3$ and distinct points $a_1,\ldots,a_m\in\mathbb R^n$, with $m\ge2$. Let
\[
\overline\Psi=\left\{\mu\in[0,1]^m:
\sum_{\alpha=1}^m\mu_\alpha=1\right\},
\qquad
\Psi=\overline\Psi\setminus\{\mathbf{e}_1,\ldots,\mathbf{e}_m\}.
\]
The excluded vertices give the classical one-pole problem. We call
$a_\alpha$ an \emph{active pole} if $\mu_\alpha>0$, and an
\emph{inactive pole} if $\mu_\alpha=0$.
Zero components are allowed, so $\Psi$ consists of all genuinely
multipolar configurations, with at least two active poles.
Throughout, the admissible parameters satisfy
\[
1<p<n,\qquad \frac{n-p}{p}\le s\le\frac{n-p}{p-1},
\qquad \mu\in\Psi.
\]
Put
\[
z_\alpha(x)=\frac{x-a_\alpha}{|x-a_\alpha|^2},
\qquad
v=\sum_{\alpha=1}^m\mu_\alpha z_\alpha,
\qquad
S=\sum_{\alpha=1}^m\mu_\alpha|z_\alpha|^2.
\]
The full multipolar potential is
\begin{equation}\label{def:full-weight}
\begin{aligned}
V_{m,s,p}^{\mu}(x)
&=s^{p-1}\bigl(n-p-(p-1)s\bigr)|v|^{p-2}
\sum_{\alpha=1}^m\frac{\mu_\alpha}{|x-a_\alpha|^2}\\
&\quad +(p-1)s^p|v|^{p-2}
\sum_{1\le\alpha<\beta\le m}
\frac{\mu_\alpha\mu_\beta|a_\alpha-a_\beta|^2}
{|x-a_\alpha|^2|x-a_\beta|^2}\\
&\quad +2(p-2)s^{p-1}|v|^{p-4}
\sum_{\alpha=1}^m\mu_\alpha
\left(\frac{|v|^2}{|x-a_\alpha|^2}
-\frac{(v\cdot(x-a_\alpha))^2}{|x-a_\alpha|^4}\right).
\end{aligned}
\end{equation}
The first two terms define the \emph{truncated potential}
$W=W_{m,s,p}^{\mu}$, while the last term is the remainder $R$;
thus $V_{m,s,p}^{\mu}=W+R$. Using
$S-|v|^2=\sum_{\alpha<\beta}\mu_\alpha\mu_\beta
|z_\alpha-z_\beta|^2$, we may equivalently write
\begin{equation}\label{def:truncated-intro}
W=s^{p-1}|v|^{p-2}
\bigl((n-p)S-(p-1)s|v|^2\bigr).
\end{equation}
These expressions are initially understood away from the poles and
the zero set of $v$; their locally integrable representatives are
used in all integral statements.
The remainder $R$ is nonpositive for $p<2$, vanishes for $p=2$, and is nonnegative for $p>2$.

Let $D^{1,p}(\mathbb R^n)$ denote the completion of $C_c^\infty(\mathbb R^n)$ for the norm $\|\nabla u\|_{L^p}$, and set
\[
H_{m,s,p}^{\mu}=-\Delta_p-V_{m,s,p}^{\mu}|\cdot|^{p-2}\cdot,
\qquad
\mathcal Q[u]=\int|\nabla u|^p\,dx-\int V_{m,s,p}^{\mu}|u|^p\,dx.
\]
A nonnegative form $\mathcal Q$ is called \emph{subcritical} if
$\mathcal Q[u]\ge\int Q|u|^p$ for some nonzero nonnegative continuous function $Q$ and every $u\in C_c^\infty(\mathbb R^n)$.
Otherwise it is called \emph{critical}; we use the same terminology for its operator.
For $\theta>0$, write
\[
\begin{aligned}
\Psi_\theta&=\{\mu\in\Psi:\mu_\alpha\le\theta\ \text{for every }\alpha\},\\
(\Psi_\theta)^\circ&=\{\mu\in\Psi:\mu_\alpha<\theta\ \text{for every }\alpha\},
\end{aligned}
\]
and $\partial\Psi_\theta=\Psi_\theta\setminus(\Psi_\theta)^\circ$.
Here the superscript $\circ$ denotes the strict part, not the relative interior; inactive poles remain allowed.

\begin{theorem}[Ground-state structure, criticality, and attainment]
\label{thmopt}\label{thm:groundstate}
For the admissible parameters above, define
\[
\phi_{m,s,\mu}(x)=\prod_{\alpha=1}^m|x-a_\alpha|^{-s\mu_\alpha}.
\]
Then the following statements hold.
\begin{enumerate}
\item[(i)] The potential $V_{m,s,p}^{\mu}$ has a nonnegative locally integrable representative, and
\[
-\Delta_p\phi_{m,s,\mu}
=V_{m,s,p}^{\mu}\phi_{m,s,\mu}^{p-1}
\quad\text{in }\mathcal D'(\mathbb R^n).
\]
For $u\in C_c^\infty(\mathbb R^n)$, set $w=u/\phi_{m,s,\mu}$ away from the active poles. Then
\[
\mathcal Q[u]
=\int_{\mathbb R^n}C_p(\nabla u,\phi_{m,s,\mu}\nabla w)\,dx\ge0,
\]
where $C_p$ is the convexity remainder defined in Lemma~\ref{lemCp}.
In particular,
\begin{equation}\label{MHp1}
\int_{\mathbb R^n}|\nabla u|^p\,dx
\ge\int_{\mathbb R^n}V_{m,s,p}^{\mu}|u|^p\,dx,
\qquad u\in D^{1,p}(\mathbb R^n).
\end{equation}

\item[(ii)] The operator $H_{m,s,p}^{\mu}$ is critical if and only if
$\mu\in\Psi_{(n-p)/(ps)}$, and is subcritical on its complement in $\Psi$.
On the critical set, $\phi_{m,s,\mu}$ is its ground state in the null-sequence sense and coefficient $1$ in \eqref{MHp1} is optimal.

\item[(iii)] The profile $\phi_{m,s,\mu}$ belongs to $D^{1,p}(\mathbb R^n)$ if and only if
\[
s>\frac{n-p}{p},
\qquad
s\mu_\alpha<\frac{n-p}{p}\quad\text{for every }\alpha.
\]
Exactly in this region, equality in \eqref{MHp1} is attained by nonzero functions, and these functions are the nonzero multiples of $\phi_{m,s,\mu}$.
On the remaining critical parameters, coefficient $1$ is sharp but is not attained.
\end{enumerate}
\end{theorem}

The structural assertion and the critical and attained cases are proved in Section~\ref{sec:groundstate}. Complementary subcriticality is proved in Section~\ref{sec:subcritical} by a strict-pole capacity argument; the quadratic case also admits the independent parameter comparison in Proposition~\ref{thmNCquad}.
The result separates two questions: criticality is determined by membership in $\Psi_{(n-p)/(ps)}$, whereas finite energy requires strict inequalities at every pole and at infinity.

The set $\Psi_\theta$ is nonempty precisely when $\theta\ge1/m$.
Consequently, the critical region is nonempty if and only if
$s\le m(n-p)/p$; equality leaves only the symmetric weight vector.
At $s=(n-p)/p$, every $\mu\in\Psi$ is critical and none is attained.
For $s>(n-p)/p$, set $\theta=(n-p)/(ps)$. Figure~\ref{fig:parameter-regimes} illustrates the attained strict region, its non-attained threshold set, and the subcritical complement for three poles.

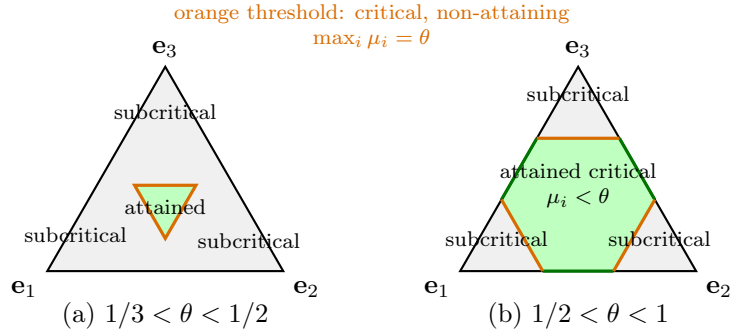
\begin{figure}[htbp]
\centering
\begin{tikzpicture}[scale=0.78, every node/.style={font=\small}]
  \begin{scope}
    \coordinate (e1a) at (0,0);
    \coordinate (e2a) at (4,0);
    \coordinate (e3a) at (2,3.464);
    \fill[gray!12] (e1a)--(e2a)--(e3a)--cycle;
    \fill[green!25] (2,0.554)--(1.48,1.455)--(2.52,1.455)--cycle;
    \draw[thick] (e1a)--(e2a)--(e3a)--cycle;
    \draw[very thick,orange!85!black] (2,0.554)--(1.48,1.455)--(2.52,1.455)--cycle;
    \node[below left] at (e1a) {$\mathbf{e}_1$};
    \node[below right] at (e2a) {$\mathbf{e}_2$};
    \node[above] at (e3a) {$\mathbf{e}_3$};
    \node[align=center,font=\scriptsize] at (2,1.10) {attained};
    \node[align=center,font=\scriptsize] at (0.48,0.62) {subcritical};
    \node[align=center,font=\scriptsize] at (3.42,0.52) {subcritical};
    \node[align=center,font=\scriptsize] at (2,2.70) {subcritical};
    \node[below=8pt,font=\small] at (2,0) {(a) $1/3<\theta<1/2$};
  \end{scope}
  \begin{scope}[xshift=7cm]
    \coordinate (e1b) at (0,0);
    \coordinate (e2b) at (4,0);
    \coordinate (e3b) at (2,3.464);
    \fill[gray!12] (e1b)--(e2b)--(e3b)--cycle;
    \fill[green!25] (1.40,0)--(2.60,0)--(3.30,1.212)--(2.70,2.252)--(1.30,2.252)--(0.70,1.212)--cycle;
    \draw[thick] (e1b)--(e2b)--(e3b)--cycle;
    \draw[very thick,green!45!black] (1.40,0)--(2.60,0);
    \draw[very thick,green!45!black] (3.30,1.212)--(2.70,2.252);
    \draw[very thick,green!45!black] (1.30,2.252)--(0.70,1.212);
    \draw[very thick,orange!85!black] (1.40,0)--(0.70,1.212);
    \draw[very thick,orange!85!black] (2.60,0)--(3.30,1.212);
    \draw[very thick,orange!85!black] (2.70,2.252)--(1.30,2.252);
    \node[below left] at (e1b) {$\mathbf{e}_1$};
    \node[below right] at (e2b) {$\mathbf{e}_2$};
    \node[above] at (e3b) {$\mathbf{e}_3$};
    \node[align=center,font=\scriptsize] at (2,1.45) {attained critical\\$\mu_i<\theta$};
    \node[align=center,font=\scriptsize] at (0.62,0.55) {subcritical};
    \node[align=center,font=\scriptsize] at (3.38,0.55) {subcritical};
    \node[align=center,font=\scriptsize] at (2,3.00) {subcritical};
    \node[below=8pt,font=\small] at (2,0) {(b) $1/2<\theta<1$};
  \end{scope}
  \node[align=center,text=orange!85!black,font=\scriptsize] at (5.5,4.15) {orange threshold: critical, non-attaining\\$\max_i\mu_i=\theta$};
\end{tikzpicture}
\caption{Parameter geometry for $m=3$ and $s>(n-p)/p$:
green denotes attained criticality, orange non-attained criticality,
and gray subcriticality. Green simplex edges allow inactive poles.
At $\theta=1$ all admissible weights are critical and non-attaining;
$\theta=1/3,1/2$ are limiting configurations.}
\label{fig:parameter-regimes}
\end{figure}

\subsection{Variational thresholds for the truncated potential}

Deleting $R$ leaves the nonnegative weight $W$ in
\eqref{def:truncated-intro}. This operation need not preserve either the optimal coefficient or the existence of a finite-energy ground state.
For fixed poles, whose dependence is suppressed in the notation, define
\begin{equation}\label{def:Cstar}
C_{*}(p,s,\mu):=
\inf_{\substack{u\in D^{1,p}(\mathbb{R}^{n})\\
\int_{\mathbb R^n}W|u|^p\,dx>0}}
\frac{\int_{\mathbb{R}^{n}}|\nabla u|^{p}\,dx}
{\int_{\mathbb{R}^{n}}W|u|^{p}\,dx}.
\end{equation}
The full Hardy inequality and the comparison in
Theorem~\ref{tmHSW2} imply that $D^{1,p}(\mathbb R^n)$ embeds continuously into $L^p(W\,dx)$ and that $0<C_*<\infty$ for every admissible parameter.

The local Hardy costs at the poles and infinity determine the
concentration threshold. Set
\begin{equation}\label{eq:hardy-profile}
\begin{gathered}
h:=\frac{n-p}{p},\qquad H_p:=h^p,\qquad
s_{\max}:=\frac{n-p}{p-1},\\
g(t):=t^{p-1}\bigl(n-p-(p-1)t\bigr),\qquad 0\le t\le s_{\max},
\end{gathered}
\end{equation}
and
\begin{equation}\label{def:concentration-coefficients}
A_\alpha(p,s,\mu):=g(s\mu_\alpha),\qquad A_\infty(p,s):=g(s).
\end{equation}
Then
\begin{equation}\label{def:concentration-threshold}
\Lambda(p,s,\mu):=
\min\left\{\frac{H_p}{A_\infty(p,s)},
\min_{\alpha:\mu_\alpha>0}\frac{H_p}{A_\alpha(p,s,\mu)}\right\},
\end{equation}
where $H_p/0=+\infty$. Since $g$ has its unique maximum $H_p$
at $h$ and all active pole coefficients are positive,
\begin{equation}\label{eq:threshold-at-least-one}
1\le\Lambda(p,s,\mu)<\infty.
\end{equation}
The first inequality is strict if $s>h$ and $s\mu_\alpha<h$
for every $\alpha$.

\begin{theorem}[Compactness, attainment, and criticality at the variational threshold]
\label{prop:compactness-threshold}
For admissible parameters, the strict gap
$C_*(p,s,\mu)<\Lambda(p,s,\mu)$ implies that the infimum in
\eqref{def:Cstar} is attained.
Moreover, every sequence $(u_k)\subset D^{1,p}(\mathbb R^n)$ satisfying
\[
\int_{\mathbb R^n}W|u_k|^p\,dx=1,\qquad
\int_{\mathbb R^n}|\nabla u_k|^p\,dx\longrightarrow C_*(p,s,\mu)
\]
admits a subsequence converging strongly in $D^{1,p}(\mathbb R^n)$
to a minimizer.
Every nonnegative normalized minimizer $u_*$ is positive a.e. and
satisfies
\begin{equation}\label{eq:subquad-extremizer-equation}
-\Delta_pu_*=C_*(p,s,\mu)Wu_*^{p-1}
\quad\text{in }\mathcal D'(\mathbb R^n).
\end{equation}
The optimally rescaled operator
\[
u\longmapsto-\Delta_pu-C_*(p,s,\mu)W|u|^{p-2}u
\]
is critical on $\mathbb R^n$.
The positivity, equation, and criticality conclusions hold whenever
the infimum is attained, even without assuming $C_*<\Lambda$.
If $2\le p<n$, the strict-gap condition is also necessary:
\begin{equation}\label{eq:attainment-gap-equivalence}
C_*(p,s,\mu)\text{ is attained in }D^{1,p}(\mathbb R^n)
\quad\Longleftrightarrow\quad
C_*(p,s,\mu)<\Lambda(p,s,\mu).
\end{equation}
When $p\ge2$ and $C_*=\Lambda$, every normalized minimizing
sequence converges weakly to zero in $D^{1,p}(\mathbb R^n)$.
Its weighted mass can concentrate or escape only through
singular ends whose concentration cost equals $\Lambda$.
\end{theorem}

The proof in Section~\ref{sec:variational} excludes concentration
on the zero set $Z(v)$ even when $p<2$; only the active
poles and infinity can carry a defect. For $p\ge2$, critical Hardy
barriers give the converse. For $p<2$ no converse is asserted, and
equality $C_*=\Lambda$ alone does not decide operator criticality.

The coefficient is also the variational spectral threshold:
$\mathcal Q_\lambda[u]:=\int|\nabla u|^p-\lambda\int W|u|^p$
is coercive for $0\le\lambda<C_*$ and has a negative direction
for $\lambda>C_*$; at $\lambda=C_*$ it is critical whenever
the infimum is attained
(Remark~\ref{cor:variational-spectral-threshold}).

\subsection{Coexistence outside the full critical set}

The next result separates full and truncated criticality for
arbitrary fixed pole configurations, varying only positive weights.

\begin{theorem}[Three truncated regimes outside the full critical set]
\label{thm:multipolar-three-regimes}
Fix $n\ge3$, $1<p<2$, and $m\ge2$ distinct poles
$a_1,\ldots,a_m$. Assume
\[
h<s\le\min\{mh,s_{\max}\},\qquad t_0:=\frac hs\in[1/m,1).
\]
For each $i\in\{1,\ldots,m\}$, fix positive numbers
$\rho_\alpha^{(i)}$, $\alpha\ne i$, such that
\[
\sum_{\alpha\ne i}\rho_\alpha^{(i)}=1,\qquad
(s-h)\max_{\alpha\ne i}\rho_\alpha^{(i)}\le h,
\]
and consider the positive-weight path
\[
\mu^{(i)}_\alpha(t):=
\begin{cases}
t,&\alpha=i,\\[1mm]
(1-t)\rho_\alpha^{(i)},&\alpha\ne i,
\end{cases}
\qquad t_0\le t<1.
\]
Write $C_i(t):=C_*(p,s,\mu^{(i)}(t))$,
$W_{i,t}:=W_{m,s,p}^{\mu^{(i)}(t)}$, and
$Q_{i,t}[u]:=\int|\nabla u|^p-\int W_{i,t}|u|^p$.
For every $t\in I:=(t_0,1)$, the full operator is subcritical and
$\Lambda(p,s,\mu^{(i)}(t))>1$. On each of these paths:
\begin{enumerate}[(i)]
\item There exists $\eta_->0$, with $\eta_-<1-t_0$, such that
      $C_i(t)<1$ for $t\in(t_0,t_0+\eta_-)$.
      At each such parameter, $Q_{i,t}$ has a negative direction
      in $C_c^\infty(\mathbb R^n)$.
\item There exists $\eta_+>0$, with $\eta_+<1-t_0$, such that
      $C_i(t)>1$ for $t\in(1-\eta_+,1)$.
      At each such parameter, $Q_{i,t}$ is coercive in
      $D^{1,p}(\mathbb R^n)$ and subcritical.
\item There exists at least one $t_{c,i}\in I$ with
      $C_i(t_{c,i})=1$. Every such parameter satisfies
      $C_i(t_{c,i})<\Lambda(p,s,\mu^{(i)}(t_{c,i}))$.
      The infimum is attained by a positive finite-energy function,
      normalized minimizing sequences are relatively compact in
      $D^{1,p}(\mathbb R^n)$, and $Q_{i,t_{c,i}}$ is critical.
\end{enumerate}
By carrying out the construction separately for each component
$\mu_i$, $i=1,\ldots,m$, we obtain at least $m$ distinct
positive-weight parameters outside the full critical set for which
the coefficient-one truncated operator is critical and possesses
a finite-energy ground state.
\end{theorem}

The proof uses continuity in the positive weights
(Lemma~\ref{lem:weight-continuity}) and the single-pole limit
(Proposition~\ref{prop:multipolar-one-pole})
\[
\lim_{t\uparrow1}C_i(t)=
\begin{cases}
H_p/g(s)>1,&h<s<s_{\max},\\
+\infty,&s=s_{\max}.
\end{cases}
\]
The ratio bound places $\mu^{(i)}(t_0)$ in the full critical
set, so $C_i(t_0)<1$. Continuity gives a crossing of $1$, and the concentration gap
makes every crossing attained and critical. No monotonicity,
uniqueness of the crossing, or geometric symmetry is assumed.

\subsection{The transition at the quadratic exponent}

The sign of $R$ leads to $C_*<1$ on the full critical set for
$1<p<2$, but to $C_*>1$ on its attained part for $p>2$
(Theorems~\ref{thm:subquad-sharp} and~\ref{tmHSW3}).
Near the quadratic exponent these comparisons yield continuity
of the optimum and persistence of attainment.

\begin{theorem}[Transition through the quadratic exponent]
\label{thm:attainment-near-two}
Fix $n\ge3$, distinct poles $a_1,\ldots,a_m$, and $\mu\in\Psi$.
Assume
\begin{equation}\label{eq:quadratic-strict-assumptions}
\frac{n-2}{2}<s<n-2,\qquad
s\mu_\alpha<\frac{n-2}{2}\quad\text{for every }\alpha.
\end{equation}
Then
\begin{equation}\label{eq:Cstar-limit-two}
\lim_{p\to2}C_*(p,s,\mu)=1.
\end{equation}
There exists $\varepsilon_0>0$ such that, for $|p-2|<\varepsilon_0$,
the parameters remain admissible and in the full-potential
attainment region, and
\begin{equation}\label{eq:near-two-strict-gap}
C_*(p,s,\mu)<\Lambda(p,s,\mu).
\end{equation}
Thus the optimum is attained by a positive finite-energy function
satisfying \eqref{eq:subquad-extremizer-equation}, and the optimally
rescaled operator is critical.
The position of the coefficient relative to $1$ is
\[
C_*(p,s,\mu)<1\quad(p<2),\qquad
C_*(2,s,\mu)=1,\qquad
C_*(p,s,\mu)>1\quad(p>2).
\]
\end{theorem}

Continuity is asserted for the optimal value, not for extremizers in the exponent-dependent energy spaces.

Table~\ref{tab:truncated-phases} summarizes this distinction.
For each fixed $p$, let $\mathcal C$ denote the admissible pairs
$(s,\mu)$ with $\mu\in\Psi_{(n-p)/(ps)}$, and let
$\mathcal A\subset\mathcal C$ be the full-potential attainment
region $s>(n-p)/p$, $\max_\alpha s\mu_\alpha<(n-p)/p$.
The \emph{optimal operator} means
$-\Delta_p-C_*W|\cdot|^{p-2}\cdot$.

\begin{table}[htbp]
\centering
\small
\renewcommand{\arraystretch}{1.3}
\begin{tabular}{@{}llll@{}}
\hline
Exponent and region & Best coefficient & Attainment & Optimal operator\\
\hline
$1<p<2$, $\mathcal C$ & $0<C_*<1$ & Yes & Critical\\
$p=2$, $\mathcal C$ & $C_*=1$ & Exactly on $\mathcal A$ & Critical\\
$2<p<n$, $\mathcal A$ & $C_*>1$ & Iff $C_*<\Lambda$ & Critical if attained\\
\hline
\end{tabular}
\caption{Truncated-potential regimes within the full critical set.
The third row is restricted to the full-potential attainment region
$\mathcal A$.}
\label{tab:truncated-phases}
\end{table}

On $\mathcal C\setminus\mathcal A$, coefficient $1$ remains sharp
for $p\ge2$ (Theorem~\ref{tmHSW1}).

\subsection{Superquadratic concentration and compactness phases}

Strict improvement over coefficient one does not imply attainment.
For $p>2$, the truncated problem has a non-attained phase inside the
full-potential attainment region. Its endpoint and the compact phase
beyond it are governed by an additional interaction threshold.

Fix $n\ge3$, $2<p<n$, $\mu\in\Psi$, and distinct poles
$a_1,\ldots,a_m$. Write $W_s:=W_{m,s,p}^{\mu}$ and, with $v$ and $S$ as above, define
\begin{equation}\label{eq:interaction-kernel}
G:=|v|^{p-2}(S-|v|^2).
\end{equation}
The auxiliary variational constant at $s=h$ is
\begin{equation}\label{eq:auxiliary-kappa}
\begin{aligned}
Q_h[u]&:=\int_{\mathbb R^n}|\nabla u|^p\,dx
-\int_{\mathbb R^n}W_h|u|^p\,dx,\\
\kappa_*&:=\inf_{0\ne u\in D^{1,p}(\mathbb R^n)}
\frac{Q_h[u]}{\int_{\mathbb R^n}G|u|^p\,dx}.
\end{aligned}
\end{equation}
Here $G>0$ a.e. and $pH_pG\le W_h$, so the denominator is
finite and positive for every nonzero $u\in D^{1,p}$.
Proposition~\ref{prop:positive-kappa} proves that $\kappa_*>0$.

\begin{theorem}[The superquadratic concentration interval and its endpoint]
\label{thm:superquadratic-phases}
For the fixed data above, let $q=\max_\alpha\mu_\alpha\in(0,1)$
and define
\begin{equation}\label{eq:superquad-transition-parameters}
\begin{aligned}
s_c&:=s_{\max}\left(1-\frac1{p+\kappa_*/H_p}\right),\\
s_\infty&:=s_{\max}\frac{1-q^{p-1}}{1-q^p},\qquad
\Lambda_\infty(s):=\frac{H_p}{g(s)}\quad(h\le s<s_{\max}).
\end{aligned}
\end{equation}
Then
\begin{equation}\label{eq:transition-parameter-order}
h<s_c\le s_\infty<\min\{s_{\max},h/q\}.
\end{equation}
The following assertions hold.
\begin{enumerate}[(i)]
\item For $h\le s\le s_c$,
\[
C_*(p,s,\mu)=\Lambda(p,s,\mu)=\Lambda_\infty(s).
\]
This optimum is not attained in $D^{1,p}$, including at $s=s_c$.
For $h<s\le s_c$ its value is strictly greater than one.
The auxiliary infimum $\kappa_*$ is also not attained.

\item For $h\le s<s_c$, the optimal operator is subcritical.
Every sequence normalized by $\int W_s|u_k|^p\,dx=1$ and
satisfying $\int|\nabla u_k|^p\,dx\to C_*(p,s,\mu)$ obeys
\begin{equation}\label{eq:superquad-mass-escape}
\int_{B_R(0)}W_s|u_k|^p\,dx\longrightarrow0
\qquad\text{for every fixed }R>0.
\end{equation}

\item If $s_c<s_\infty$, there exists $\delta>0$, with
$s_\infty+\delta<\min\{s_{\max},h/q\}$, such that
\begin{equation}\label{eq:superquad-compact-phase}
1<C_*(p,s,\mu)<\Lambda(p,s,\mu)
\qquad(s_c<s<s_\infty+\delta).
\end{equation}
On this interval the optimum has a positive finite-energy
extremizer, normalized minimizing sequences are relatively
compact in $D^{1,p}$, and the optimal operator is critical.
A sufficient condition for $s_c<s_\infty$ is
\begin{equation}\label{eq:explicit-two-phase-condition}
\frac{2(p-2)}h<
\frac{1-pq^{p-1}+(p-1)q^p}{q^{p-1}(1-q)}.
\end{equation}
For each fixed $n$, pole configuration, and $\mu\in\Psi$,
this condition holds for all $p>2$ sufficiently close to $2$.
\end{enumerate}
When $s_c<s_\infty$, the optimal operator at $s=s_c$ is also
critical, although its ground state has infinite energy; see
Proposition~\ref{prop:separated-endpoint-criticality}.
\end{theorem}

Thus, under strict separation, increasing $s$ produces a subcritical
non-attained interval, a critical endpoint with infinite energy,
and an attained critical phase. These are different behaviors of
the optimally rescaled operator, even though $C_*>1$ throughout
the part with $s>h$. Criticality at a coincident endpoint
$s_c=s_\infty$ is not settled by this theorem.

Strict separation holds for every bipolar configuration
(Proposition~\ref{prop:bipolar-integral-positive}). For fixed
multipolar weights and poles it holds near both $p=2$ and $p=n$.
It also holds in a simplex-vertex neighborhood uniform in
$2<p<n$, and, for fixed weights, when all poles except a
maximal-weight pole cluster near a second point
(Proposition~\ref{prop:multipolar-strictness}).
The latter geometric neighborhood can be chosen uniformly for
$p$ in compact subintervals of $(2,n)$; no universal separation
claim is made for arbitrary multipolar data.

A distinct non-attained phase can be dominated by a pole.
If $q\ge(p-1)/p$, Proposition~\ref{prop:pole-concentration-layer}
gives a relative neighborhood $J$ of $s_0=h/q$ in
$[h,s_{\max}]$ on which
\[
C_*(p,s,\mu)=\Lambda(p,s,\mu)=\frac{H_p}{g(sq)}.
\]
The optimum is not attained and the optimal operator is subcritical.
This neighborhood is two-sided when $q>(p-1)/p$ and one-sided
when $q=(p-1)/p$, since then $s_0=s_{\max}$.
The intervening pole-dominated phase boundaries are not classified.

\paragraph{Scope and organization}
The full-potential classification is complete. For the truncated
problem, we do not give a general closed formula for $C_*$ or
a classification of all weights. In particular, the subquadratic
range $s>mh$, universal superquadratic separation
$s_c<s_\infty$, criticality at a coincident endpoint, and the
remaining pole-dominated phase boundaries are unresolved.

Sections~\ref{sec:groundstate} and~\ref{sec:subcritical} prove
the full-potential theorem. Section~\ref{sec:variational}
establishes the concentration principle and its coefficient
consequences. Section~\ref{sec:parameter-phases} proves
Theorems~\ref{thm:multipolar-three-regimes},
\ref{thm:attainment-near-two}, and~\ref{thm:superquadratic-phases},
including the endpoint, separation, and pole-concentration results.

\section{Ground-state representation and criticality}\label{sec:groundstate}
\subsection{The $C_p$ inequality}
We begin with a standard inequality for $C_p$, which underlies the ground-state representation and the proof of the multipolar Hardy inequality below.
\begin{lemma}[$C_{p}$ inequality]\label{lemCp}
Let $1<p<\infty$ and $\xi,\eta\in\mathbb{R}^{n}$. Define
$$
C_p(\xi,\eta)=|\xi|^p-|\xi-\eta|^p-p|\xi-\eta|^{p-2}(\xi-\eta)\cdot\eta.
$$
Then the following hold:
\begin{itemize}
\item[(i)] If $p\ge 2$, then one may take
$$
c_1(p)=\frac{2^{2-p}}{p-1},\qquad c_2(p)=p(p-1),
$$
and
\begin{equation}\label{Cpbig2}
c_1(p)|\eta|^p \leq C_p(\xi,\eta) \leq c_2(p)\bigl(|\eta|+|\xi-\eta|\bigr)^{p-2}|\eta|^2.
\end{equation}
\item[(ii)] If $1<p<2$, then one may take
$$
\pi_1(p)=\frac{p(p-1)}{2},\qquad \pi_2(p)=1+(p-1)^{p-1},
$$
and
\begin{equation}\label{Cpless2}
\pi_1(p)\bigl(|\eta|+|\xi-\eta|\bigr)^{p-2}|\eta|^2 \leq C_p(\xi,\eta) \leq \pi_2(p)|\eta|^p.
\end{equation}
\end{itemize}
\end{lemma}

\begin{proof}
Put $a=\xi-\eta$, $b=\eta$, and $\Phi(t)=|a+tb|^p$. The integral
Taylor formula gives
\[
\begin{aligned}
C_p(\xi,\eta)&=\int_0^1(1-t)\Phi''(t)\,dt,\\
\Phi''(t)&=p|a+tb|^{p-2}|b|^2
+p(p-2)|a+tb|^{p-4}((a+tb)\cdot b)^2,
\end{aligned}
\]
for a.e. $t$. Thus $\Phi''\le p(p-1)(|a|+|b|)^{p-2}|b|^2$
when $p\ge2$, whereas the reverse inequality holds when $1<p<2$.
Integration proves the upper bound in \eqref{Cpbig2} and the lower
bound in \eqref{Cpless2}, with the stated constants.

For the other bound when $p\ge2$, let $A(z)=|z|^{p-2}z$.
Projecting $a$ onto the direction of $b\ne0$ and integrating $DA$
along the segment from $a$ to $a+tb$ gives
\[
(A(a+tb)-A(a))\cdot(tb)
\ge (t|b|)^p\inf_{\gamma\in\mathbb R}
\int_\gamma^{\gamma+1}|r|^{p-2}\,dr
=\frac{2^{2-p}}{p-1}t^p|b|^p.
\]
The infimum occurs on $[-1/2,1/2]$. Using
$C_p=p\int_0^1(A(a+tb)-A(a))\cdot b\,dt$ proves the lower bound.

Finally, suppose $1<p<2$ and $a,b\ne0$, and put
$u=|a|^2+2a\cdot b$. If $u\ge0$, concavity and subadditivity of
$t^{p/2}$ imply
\[
C_p\le (u+|b|^2)^{p/2}-u^{p/2}\le |b|^p.
\]
If $u<0$, then $|a+b|\le|b|$, and, with $x=|a|/|b|$,
\[
C_p\le(1-x^p+px^{p-1})|b|^p
\le\bigl(1+(p-1)^{p-1}\bigr)|b|^p.
\]
The last scalar maximum occurs at $x=p-1$; the cases $a=0$ or
$b=0$ are immediate.
\end{proof}

\subsection{Ground-state representation}
We next record the ground-state representation associated with a positive weight.

\begin{lemma}[Ground-state representation and remainder bounds]\label{RTE}
Let $1<p<\infty$ and let $\phi\in C^{2}(\mathbb{R}^{n})$ be positive. Then, for every $u=\phi w\in C_{0}^{\infty}(\mathbb{R}^{n})$,
$$
\int_{\mathbb{R}^{n}} |\nabla u|^{p} \,dx + \int_{\mathbb{R}^{n}} \frac{\Delta_{p}\phi}{\phi^{p-1}} |u|^{p} \,dx
= \int_{\mathbb{R}^{n}} C_{p}(\nabla u,\phi\nabla w) \,dx,
$$
where
$$
C_{p}(\xi,\eta)=|\xi|^{p}-|\xi-\eta|^{p}-p|\xi-\eta|^{p-2}(\xi-\eta)\cdot\eta.
$$
Moreover, the following remainder bounds hold:
\begin{enumerate}
\item If $p\ge2$, then
\begin{equation}\label{REgeq2}
\begin{aligned}
c_{1}(p) \int_{\mathbb{R}^{n}} \phi^{p}|\nabla w|^{p} \,dx
&\le \int_{\mathbb{R}^{n}} C_{p}(\nabla u,\phi\nabla w) \,dx\\
&\le c_{2}(p) \int_{\mathbb{R}^{n}} \phi^{p}
\Bigl(|\nabla w|+|w|\frac{|\nabla\phi|}{\phi}\Bigr)^{p-2}|\nabla w|^{2} \,dx.
\end{aligned}
\end{equation}
\item If $1<p<2$, then
\begin{equation}\label{REleq2}
\begin{aligned}
\pi_{1}(p) \int_{\mathbb{R}^{n}} \phi^{p}
\Bigl(|\nabla w|+|w|\frac{|\nabla\phi|}{\phi}\Bigr)^{p-2}|\nabla w|^{2} \,dx
&\le \int_{\mathbb{R}^{n}} C_{p}(\nabla u,\phi\nabla w) \,dx\\
&\le \pi_{2}(p) \int_{\mathbb{R}^{n}} \phi^{p}|\nabla w|^{p} \,dx.
\end{aligned}
\end{equation}
\end{enumerate}
\end{lemma}

\begin{proof}
Since $\nabla u-\phi\nabla w=w\nabla\phi$, the definition of $C_p$ gives
\[
C_p(\nabla u,\phi\nabla w)
=|\nabla u|^p-|\nabla\phi|^{p-2}\nabla\phi\cdot
\nabla(\phi|w|^p).
\]
Integration by parts proves the identity. Applying
Lemma~\ref{lemCp} with $\eta=\phi\nabla w$ and
$\xi-\eta=w\nabla\phi$ gives \eqref{REgeq2}--\eqref{REleq2}.
\end{proof}

We now apply Lemma~\ref{RTE} with the multipolar ansatz to
identify the potential \eqref{def:full-weight} and prove the
Hardy inequality asserted in Theorem~\ref{thmopt}.

\begin{theorem}[Multipolar Hardy inequality]\label{thmMHC}
For the admissible parameters and the potential
$V_{m,s,p}^{\mu}$ defined in \eqref{def:full-weight},
\begin{equation}\label{MHp}
\int_{\mathbb R^n}|\nabla u|^p\,dx
\ge\int_{\mathbb R^n}V_{m,s,p}^{\mu}|u|^p\,dx,
\qquad u\in C_c^\infty(\mathbb R^n).
\end{equation}
\end{theorem}

\begin{proof}[Proof of Theorem~\ref{thmMHC}]
Set $\phi=\prod_\alpha|x-a_\alpha|^{-s\mu_\alpha}$ and recall
$v=\sum_\alpha\mu_\alpha z_\alpha$.
We compute away from the poles and $Z(v)$. Lemmas~\ref{lem:Vloc}
and \ref{lem:GSextension} below justify the resulting weak identity
and representation throughout $\mathbb R^n$.

Write
\[
\begin{aligned}
S&=\sum_\alpha\mu_\alpha|z_\alpha|^2,\qquad
P=S-|v|^2=\sum_{\alpha<\beta}\mu_\alpha\mu_\beta|z_\alpha-z_\beta|^2,\\
T&=\sum_\alpha\mu_\alpha\bigl(|v|^2|z_\alpha|^2-(v\cdot z_\alpha)^2\bigr).
\end{aligned}
\]
The identities $\nabla\phi=-s\phi v$,
$\operatorname{div}v=(n-2)S$, and
\[
\nabla(|v|^{p-2})\cdot v
=(2-p)|v|^{p-2}S+2(p-2)|v|^{p-4}T
\]
yield
\begin{equation}\label{pL}
\begin{aligned}
\frac{-\Delta_p\phi}{\phi^{p-1}}
&=s^{p-1}\bigl(n-p-(p-1)s\bigr)|v|^{p-2}S\\
&\quad +(p-1)s^p|v|^{p-2}P
+2(p-2)s^{p-1}|v|^{p-4}T.
\end{aligned}
\end{equation}
Expanding $S,P,T$ recovers precisely \eqref{def:full-weight}.
The ground-state representation gives
\begin{equation}\label{GHI}
\int_{\mathbb R^n}|\nabla u|^p\,dx
\ge\int_{\mathbb R^n}\frac{-\Delta_p\phi}{\phi^{p-1}}|u|^p\,dx,
\end{equation}
and hence \eqref{MHp}.
\end{proof}

The preceding computation also yields a short proof of the pointwise nonnegativity of $V_{m,s,p}^{\mu}$.

\begin{proposition}[Positivity and local integrability of the full potential]\label{propVpos}
Let $n\geq 3$, $1<p<n$, $\frac{n-p}{p}\leq s\leq\frac{n-p}{p-1}$, and $\mu\in\Psi$. Then $V_{m,s,p}^{\mu}$ has a nonnegative locally integrable representative; in particular,
\[
V_{m,s,p}^{\mu}\geq0\qquad\hbox{a.e. in }\mathbb R^n.
\]
\end{proposition}

\begin{proof}[Proof of Proposition~\ref{propVpos}]
Write
$$
z_\alpha=\frac{x-a_\alpha}{|x-a_\alpha|^2},\qquad v=\sum_{\alpha=1}^m \mu_\alpha z_\alpha,
$$
and set
$$
S=\sum_{\alpha=1}^m \mu_\alpha |z_\alpha|^2,\qquad
P=\sum_{1\leq \alpha<\beta\leq m}\mu_\alpha\mu_\beta |z_\alpha-z_\beta|^2=S-|v|^2,
$$
$$
T=\sum_{\alpha=1}^m \mu_\alpha\bigl(|v|^2|z_\alpha|^2-(v\cdot z_\alpha)^2\bigr)\geq 0.
$$
Then
$$
V_{m,s,p}^{\mu}=s^{p-1}\bigl(n-p-(p-1)s\bigr)|v|^{p-2}S+(p-1)s^p|v|^{p-2}P+2(p-2)s^{p-1}|v|^{p-4}T.
$$
If $p\geq 2$, then each term is nonnegative, since $n-p-(p-1)s\geq 0$ in the admissible range of $s$.

Assume next that $1<p<2$ and work on $\{v\neq0\}$. By Jensen's inequality,
$$
\sum_{\alpha=1}^m \mu_\alpha(v\cdot z_\alpha)^2\geq \Bigl(\sum_{\alpha=1}^m \mu_\alpha v\cdot z_\alpha\Bigr)^2=|v|^4,
$$
so
$$
0\leq T\leq |v|^2\bigl(S-|v|^2\bigr)=|v|^2P.
$$
Therefore
$$
V_{m,s,p}^{\mu}\geq s^{p-1}|v|^{p-2}\Bigl(\bigl(n-p-(p-1)s\bigr)S+\bigl((p-1)s+2p-4\bigr)P\Bigr).
$$
If $(p-1)s+2p-4\geq 0$, the right-hand side is nonnegative. If $(p-1)s+2p-4<0$, then $P\leq S$, and hence
$$
V_{m,s,p}^{\mu}\geq s^{p-1}|v|^{p-2}(n+p-4)S\geq 0.
$$
Since $n\geq 3$ and $p>1$, we have $n+p-4>0$. This proves the
pointwise nonnegativity on $\{v\neq0\}$. The argument in
Lemma~\ref{lem:Vloc} below shows that $Z(v)$ has measure zero and
establishes local integrability. We choose the representative
$V_{m,s,p}^{\mu}=0$ on $Z(v)$; all assertions are understood almost
everywhere.
\end{proof}

\begin{lemma}[Local integrability and the weak ground-state equation]\label{lem:Vloc}
Let $n\ge3$, $1<p<n$, $\frac{n-p}{p}\leq s\leq\frac{n-p}{p-1}$, and $\mu\in\Psi$. Let $V=V_{m,s,p}^{\mu}$ be defined by \eqref{def:full-weight} on $\{v\ne0\}$. Then $V$ has an extension in $L^1_{\rm loc}(\mathbb R^n)$, and
\[
-\Delta_p\phi=V\phi^{p-1}\quad\hbox{in }\mathcal D'(\mathbb R^n),
\qquad \phi=\prod_{\alpha=1}^m|x-a_\alpha|^{-s\mu_\alpha}.
\]
\end{lemma}
\begin{proof}
Put $\Omega=\mathbb R^n\setminus\{a_1,\ldots,a_m\}$ and
$Z=\{x\in\Omega:v(x)=0\}$. At $x_0\in Z$,
\[
\operatorname{tr}Dv(x_0)
=(n-2)\sum_\alpha\frac{\mu_\alpha}{|x_0-a_\alpha|^2}>0.
\]
Some component $f=v_j$ therefore has nonzero gradient at $x_0$.
After a rotation, the implicit function theorem and the fundamental
theorem of calculus give, in a neighborhood $U$,
\[
f(x',x_n)=(x_n-\psi(x'))h(x',x_n),\qquad |h|\ge c_0>0.
\]
Thus $Z\cap U$ lies on a smooth graph and
$|v(x)|\ge c_0|x_n-\psi(x')|$. For $1<p<2$,
\[
|v(x)|^{p-2}\le C|x_n-\psi(x')|^{p-2}\in L^1(U),
\]
because $p-2>-1$. A finite covering gives $|Z\cap K|=0$ and
$|v|^{p-2}\in L^1(K)$ for every $K\Subset\Omega$.
Moreover,
\[
0\le |v|^2|z_\alpha|^2-(v\cdot z_\alpha)^2
\le C_K|v|^2,
\]
so every term in $V$ is locally integrable on $\Omega$. Near an
active pole, $|v|\asymp|x-a_\alpha|^{-1}$ gives
$|V|\le C|x-a_\alpha|^{-p}\in L^1_{\rm loc}$.
At an inactive pole, $\phi,v$ are smooth; either $V$ is bounded
or the preceding graph argument applies. Hence
$V\in L^1_{\rm loc}(\mathbb R^n)$.

To justify weak differentiation across $Z$ when $1<p<2$, regularize
\[
A_\varepsilon(v)=(|v|^2+\varepsilon^2)^{(p-2)/2}v.
\]
On each $K\Subset\Omega$,
\[
|D A_\varepsilon(v)|
\le C_p|v|^{p-2}|Dv|\in L^1(K)
\quad\hbox{a.e.}
\]
Dominated convergence for $A_\varepsilon$ and its derivatives yields
$|v|^{p-2}v\in W^{1,1}_{\rm loc}(\Omega)$.
For $p\ge2$ this follows directly from the $C^1$ regularity of
$z\mapsto|z|^{p-2}z$. Multiplying by
$-s^{p-1}\phi^{p-1}$ proves the distributional equation on $\Omega$;
the same argument also extends it across inactive poles.

Finally, near an active pole put $r=|x-a_\alpha|$ and
$\beta_\alpha=s\mu_\alpha$. Then
\[
\phi\asymp r^{-\beta_\alpha},\qquad
|\nabla\phi|^{p-1}\asymp r^{-(\beta_\alpha+1)(p-1)}.
\]
For a cut-off $\chi_\varepsilon$ vanishing on $B_\varepsilon(a_\alpha)$,
equal to one outside $B_{2\varepsilon}(a_\alpha)$, and satisfying
$|\nabla\chi_\varepsilon|\le C/\varepsilon$, the flux error obeys
\[
\begin{aligned}
\left|\int\eta|\nabla\phi|^{p-2}\nabla\phi\cdot
\nabla\chi_\varepsilon\,dx\right|
&\le C\varepsilon^{-1}\int_\varepsilon^{2\varepsilon}
r^{n-1-(\beta_\alpha+1)(p-1)}\,dr\\
&\le C\varepsilon^{n-p-(p-1)\beta_\alpha}.
\end{aligned}
\]
Since $\mu_\alpha<1$ and $s\le(n-p)/(p-1)$, this exponent is positive.
The same condition gives
$V\phi^{p-1}=O(r^{-p-(p-1)\beta_\alpha})\in L^1_{\rm loc}$.
The fluxes therefore vanish at every pole, extending the equation to
$\mathbb R^n$.
\end{proof}

\begin{lemma}[Ground-state representation across the poles]\label{lem:GSextension}
For every $u\in C_c^\infty(\mathbb R^n)$, put $w=u/\phi$ on $\Omega=\mathbb R^n\setminus\{a_\alpha:\mu_\alpha>0\}$. Then
\begin{equation}\label{GHIexact}
\int_{\mathbb R^n}\bigl(|\nabla u|^p-V|u|^p\bigr)dx
=\int_{\mathbb R^n}C_p(\nabla u,\phi\nabla w)dx\ge0.
\end{equation}
This is the representation asserted in Theorem~\ref{thm:groundstate}(i), with its domain stated explicitly.
\end{lemma}
\begin{proof}
For every active pole choose a cut-off $\chi_{\alpha,\epsilon}$ that
vanishes on $B_\epsilon(a_\alpha)$, equals one outside
$B_{2\epsilon}(a_\alpha)$, and satisfies
$|\nabla\chi_{\alpha,\epsilon}|\le C\epsilon^{-1}$. Put
$\chi_\epsilon=\prod_{\mu_\alpha>0}\chi_{\alpha,\epsilon}$ and
$u_\epsilon=\chi_\epsilon u$. Lemma~\ref{RTE} applies to $u_\epsilon$
on $\Omega$. By Lemma~\ref{lem:Vloc}, $V|u|^p\in L^1_{\rm loc}$.
Moreover, near each active pole,
\[
\begin{aligned}
\phi\nabla(u/\phi)&=O(|x-a_\alpha|^{-1}),\\
C_p(\nabla u,\phi\nabla(u/\phi))&=O(|x-a_\alpha|^{-p}),
\qquad V|u|^p=O(|x-a_\alpha|^{-p}).
\end{aligned}
\]
Hence the original remainder and potential integrands over the removed
balls $B_{2\epsilon}(a_\alpha)$ have integrals
$O(\epsilon^{n-p})$. On each transition annulus
$B_{2\epsilon}(a_\alpha)\setminus B_\epsilon(a_\alpha)$, both
$\nabla u_\epsilon$ and
$\phi\nabla(u_\epsilon/\phi)$ are $O(\epsilon^{-1})$, so the cut-off
remainder integrand is $O(\epsilon^{-p})$ and its integral is again
$O(\epsilon^{n-p})$. The cut-off and original integrands agree outside
the union of the balls $B_{2\epsilon}(a_\alpha)$.

Finally, $u_\epsilon\to u$ in $W^{1,p}(\mathbb R^n)$ because $p<n$,
and the potential terms converge by the preceding $O(\epsilon^{n-p})$
estimate. Passing to the limit in the identity for $u_\epsilon$ proves
\eqref{GHIexact}.
\end{proof}

\subsection{Null sequences and criticality}
We next turn from the Hardy inequality itself to the criticality criterion used in the proof of Theorem~\ref{thmopt}.

\begin{lemma}[A null-sequence criterion for criticality]\label{lem:crit}
Let $H$ be a nonnegative operator on $L^{p}(\mathbb{R}^{n})$. If there
exist admissible test functions
$u_{\epsilon}\in C_c^\infty(\mathbb R^n)$ and a measurable function
$\psi$, positive a.e. in $\mathbb{R}^{n}$, such that
\begin{equation*}
\lim_{\epsilon\to 0}\langle Hu_{\epsilon},u_{\epsilon}\rangle=0,
\qquad
u_{\epsilon}\to\psi \quad\text{a.e. in }\mathbb{R}^{n},
\end{equation*}
then $H$ is critical.
\end{lemma}
\begin{proof}
Suppose otherwise. Then there exists a nontrivial nonnegative potential $Q$ such that
\begin{equation*}
\langle Hu,u\rangle\geq \int_{\mathbb{R}^{n}}Q|u|^{p}dx.
\end{equation*}
Fatou's lemma yields
\begin{equation*}
0\leq\int_{\mathbb{R}^{n}}Q|\psi|^{p}dx\leq \liminf_{\epsilon\to 0}\int_{\mathbb{R}^{n}}Q|u_{\epsilon}|^{p}dx\leq\lim_{\epsilon\to 0}\langle Hu_{\epsilon},u_{\epsilon}\rangle=0,
\end{equation*}
so $Q=0$ a.e. in $\mathbb{R}^{n}$, a contradiction.
\end{proof}

We now construct the logarithmic cut-off family used to prove criticality and optimality in Theorem~\ref{thmopt}.

\begin{proof}[Proof of criticality and optimality in Theorem~\ref{thmopt}]
Assume $\mu\in\Psi_{\frac{n-p}{ps}}$, and put
$d=\min_{\alpha\ne\beta}|a_\alpha-a_\beta|$.
For sufficiently small $0<\epsilon<d/4$, define
\[
w_\epsilon(x)=
\begin{cases}
0,&|x-a_\alpha|\le\epsilon^2\quad\hbox{for some }\alpha,\\
\dfrac{\log(|x-a_\alpha|/\epsilon^2)}{\log(1/\epsilon)},
 &\epsilon^2<|x-a_\alpha|<\epsilon,\\
1,&x\in B_{\epsilon^{-1}}(0)\setminus\bigcup_\alpha B_\epsilon(a_\alpha),\\
\dfrac{\log(\epsilon^2|x|)}{\log\epsilon},
 &\epsilon^{-1}<|x|<\epsilon^{-2},\\
0,&|x|\ge\epsilon^{-2}.
\end{cases}
\]
Smooth this cut-off without changing its zero and unit regions or
the estimates below. Then $0\le w_\epsilon\le1$ and
$u_\epsilon=\phi w_\epsilon\in C_c^\infty(\mathbb R^n)$, with
$u_\epsilon\to\phi>0$ a.e.

Let $q=2$ for $p\ge2$ and $q=p$ for $1<p<2$. The upper remainder
bounds in Lemma~\ref{RTE} give
$0\le\mathcal Q[u_\epsilon]\le C\int\phi^p E_{q,\epsilon}\,dx$, where
\[
E_{q,\epsilon}=
\begin{cases}
\bigl(|\nabla w_\epsilon|+
 w_\epsilon|\nabla\phi|/\phi\bigr)^{p-2}|\nabla w_\epsilon|^2,
 &p\ge2,\\
|\nabla w_\epsilon|^p,&1<p<2.
\end{cases}
\]
On each inner transition annulus,
$|\nabla w_\epsilon|\le C/(r\log(1/\epsilon))$,
$\phi\asymp r^{-s\mu_\alpha}$, and $|\nabla\phi|/\phi\le C/r$,
where $r=|x-a_\alpha|$. On the outer annulus the same estimates
hold with $r=|x|$ and $\phi\asymp r^{-s}$. Consequently,
\[
\mathcal Q[u_\epsilon]\le
\frac{C}{(\log(1/\epsilon))^q}
\left(
\sum_\alpha\int_{\epsilon^2}^{\epsilon}
r^{n-p-ps\mu_\alpha-1}\,dr+
\int_{\epsilon^{-1}}^{\epsilon^{-2}}r^{n-p-ps-1}\,dr
\right)\longrightarrow0.
\]
Indeed, the inner exponents satisfy $n-p-ps\mu_\alpha\ge0$,
the outer exponent satisfies $n-p-ps\le0$, and $q>1$.
Lemma~\ref{lem:crit} now proves criticality.

For completeness, these tests also establish coefficient optimality
directly. Choose a fixed annulus $A$ sufficiently far from the poles
that $|\nabla\phi|>0$ there; this is possible since
$v(x)=x/|x|^2+O(|x|^{-2})$ at infinity.
For small $\epsilon$, $w_\epsilon=1$ on $A$, so
\[
\begin{aligned}
\int|\nabla u_\epsilon|^p\,dx&\ge
C_A:=\int_A|\nabla\phi|^p\,dx>0,\\
\int V|u_\epsilon|^p\,dx
&=\int|\nabla u_\epsilon|^p\,dx-o(1)\ge C_A/2.
\end{aligned}
\]
Thus the Hardy quotient tends to one, as asserted.
\end{proof}

\subsection{Attainability}
\begin{proof}[Proof of the attainability of the constant in Theorem~\ref{thmopt}]
\leavevmode\par\noindent
By \eqref{REgeq2} and \eqref{REleq2}, equality in \eqref{MHp} holds exactly for $u=c\phi$. Since \eqref{MHp} extends from $C_{0}^{\infty}(\mathbb{R}^{n})$ to $D^{1,p}(\mathbb{R}^{n})$ by density, the optimal constant $1$ is attained in $D^{1,p}(\mathbb{R}^{n})$ if and only if $\phi\in D^{1,p}(\mathbb{R}^{n})$.

Now $\nabla\phi=-s\phi v$, so $|\nabla\phi|^{p}=|s|^{p}\phi^{p}|v|^{p}$. Fix $0<\epsilon<1$. At every active pole $a_\alpha$, meaning $\mu_\alpha>0$, the factors $\prod_{\beta\neq\alpha}|x-a_{\beta}|^{-s\mu_{\beta}}$ and $|x-a_{\alpha}|\,|v(x)|$ are bounded above and below by positive constants on $B_{\epsilon}(a_{\alpha})\setminus\{a_{\alpha}\}$; hence $|\nabla\phi(x)|^{p}\asymp |x-a_{\alpha}|^{-(s\mu_{\alpha}+1)p}$ as $x\to a_{\alpha}$. At an inactive pole, where $\mu_\alpha=0$, both $\phi$ and $\nabla\phi$ are locally bounded. Moreover, $|\nabla\phi|^{p}$ is bounded on $B_{\epsilon^{-1}}(0)\setminus\bigcup_{\alpha=1}^{m}B_{\epsilon}(a_{\alpha})$, while $\phi(x)\asymp |x|^{-s}$ and $|v(x)|\asymp |x|^{-1}$ as $|x|\to\infty$, so $|\nabla\phi(x)|^{p}\asymp |x|^{-(s+1)p}$ at infinity. Therefore $\phi\in D^{1,p}(\mathbb{R}^{n})$ if and only if $s>\frac{n-p}{p}$ and $n-p-ps\mu_{\alpha}>0$ for every $\alpha$, equivalently, if and only if $s\in\bigl(\frac{n-p}{p},\frac{n-p}{p-1}\bigr]$ and $\mu\in(\Psi_{\frac{n-p}{ps}})^{\circ}$. In that case the minimizers are exactly $u=c\phi$.

If $s=\frac{n-p}{p}$, then $|\nabla\phi(x)|^{p}\asymp |x|^{-n}$ at infinity. If $s\in\bigl(\frac{n-p}{p},\frac{n-p}{p-1}\bigr]$ and $n-p-ps\mu_{\alpha}=0$ for some $\alpha$, then $|\nabla\phi(x)|^{p}\asymp |x-a_{\alpha}|^{-n}$ near $a_{\alpha}$. In either case the integral of $|\nabla\phi|^{p}$ diverges logarithmically, so $\phi\notin D^{1,p}(\mathbb{R}^{n})$ and the constant $1$ is not attained.
\end{proof}

\subsection{The bipolar specialization}
For two poles, the vector identities reduce the full potential to an
explicit expression. Its sharpness and attainment follow directly
from Theorem~\ref{thmopt}.

\begin{corollary}[The explicit bipolar potential]\label{MH2}\label{corMH2}
Let $n\ge3$, $1<p<n$, $(n-p)/p\le s\le(n-p)/(p-1)$, and
$0<\kappa<1$. Put $\bar a=\kappa a_2+(1-\kappa)a_1$ and
$r_i=|x-a_i|$, $i=1,2$. Then
\begin{equation}\label{2HI}
\int_{\mathbb R^n}|\nabla u|^p\,dx
\ge\int_{\mathbb R^n}V_{2,s,p}^{\kappa}|u|^p\,dx,
\qquad u\in C_c^\infty(\mathbb R^n),
\end{equation}
where
\begin{equation}\label{2HP}
\begin{aligned}
V_{2,s,p}^{\kappa}
&=s^{p-1}\bigl(n-p-(p-1)s\bigr)|x-\bar a|^{p-2}
\left(\frac{\kappa}{r_1^pr_2^{p-2}}
+\frac{1-\kappa}{r_1^{p-2}r_2^p}\right)\\
&\quad +(p-1)s^p\kappa(1-\kappa)
\frac{|a_1-a_2|^2|x-\bar a|^{p-2}}{r_1^pr_2^p}\\
&\quad +2(p-2)s^{p-1}\kappa(1-\kappa)|x-\bar a|^{p-4}
\frac{r_1^2r_2^2-((x-a_1)\cdot(x-a_2))^2}{r_1^pr_2^p}.
\end{aligned}
\end{equation}
Coefficient one is optimal when
$1-(n-p)/(ps)\le\kappa\le(n-p)/(ps)$.
It is attained by a nonzero function in $D^{1,p}(\mathbb R^n)$
if and only if
\[
s>\frac{n-p}{p},\qquad
1-\frac{n-p}{ps}<\kappa<\frac{n-p}{ps}.
\]
In that case all extremizers are
$u(x)=c|x-a_1|^{-s\kappa}|x-a_2|^{-s(1-\kappa)}$, $c\ne0$.
\end{corollary}

\begin{proof}
For $\mu=(\kappa,1-\kappa)$,
\[
|v|^2=\frac{|x-\bar a|^2}{|x-a_1|^2|x-a_2|^2}.
\]
The remaining angular term reduces to $\kappa(1-\kappa)$ times
the Gram determinant of $x-a_1,x-a_2$, divided by
$|x-a_1|^4|x-a_2|^4$. Substitution into \eqref{def:full-weight} gives \eqref{2HP},
and \eqref{MHp} yields \eqref{2HI}. Theorem~\ref{thmopt}
supplies sharpness and attainment.
\end{proof}

Within the closed critical interval, equality in either weight threshold
or $s=(n-p)/p$ gives logarithmic non-attainment. The interval collapses
to $\{1/2\}$ at $p=2$, $s=n-2$, or at $1<p<2$,
$s=2(n-p)/p$. For $2<p<n$ and $s=(n-p)/(p-1)$, the strict interval is
$(1/p,(p-1)/p)$; in particular $\kappa=1/2$ recovers the symmetric
full-potential inequality of Cazacu--Rugin\u{a} \cite{CR24}.

\section{Subcriticality outside the critical region}\label{sec:subcritical}
The critical and subcritical assertions are proved by different mechanisms. In this section we assume criticality, extract a normalized null sequence, and show that a strict pole forces a fixed positive energy cost.

\subsection{Localized null sequences}
The next lemma supplies the converse variational direction used below: criticality yields locally normalized null sequences. Unlike Lemma~\ref{lem:crit}, it does not assert pointwise convergence to a positive ground state.

\begin{lemma}[Localized null sequences]\label{lem:localnull}
Let $1<p<\infty$, let $V\in L^1_{\mathrm{loc}}(\mathbb R^n)$, and suppose that
\[
\mathcal Q[u]:=\int_{\mathbb R^n}
\bigl(|\nabla u|^p-V|u|^p\bigr)\,dx\ge0,
\qquad u\in C_c^\infty(\mathbb R^n).
\]
If the associated operator is critical in the sense of the definition above, then, for every nonempty ball $B\Subset\Omega$, there exist $u_k\in C_c^\infty(\mathbb R^n)$ such that
\[
u_k\ge0,\qquad \int_Bu_k^p\,dx=1,\qquad \mathcal Q[u_k]\longrightarrow0.
\]
\end{lemma}

\begin{proof}
Fix $B\Subset\Omega$ and set
\[
\lambda_B:=\inf\left\{\mathcal Q[u]:u\in C_c^\infty(\mathbb R^n),\ \int_B|u|^p\,dx=1\right\}.
\]
If $\lambda_B>0$, positive $p$-homogeneity and nonnegativity give
\[
\mathcal Q[u]\ge\lambda_B\int_B|u|^p\,dx
\]
for $u\in C_c^\infty(\mathbb R^n)$. Choosing a nonzero $P\in C_c^\infty(B)$ with $0\le P\le\lambda_B$ contradicts criticality. Hence $\lambda_B=0$. Choose $v_k$ such that
\[
\int_B|v_k|^p\,dx=1,\qquad \mathcal Q[v_k]<\frac1{2k}.
\]
For a fixed $v\in C_c^\infty(\mathbb R^n)$, define $g_\varepsilon(v):=\sqrt{v^2+\varepsilon^2}-\varepsilon$. Then $g_\varepsilon(v)$ is smooth and nonnegative, its support is contained in $\operatorname{supp}v$, and
\[
0\le g_\varepsilon(v)\le |v|,\qquad
\nabla g_\varepsilon(v)=\frac{v}{\sqrt{v^2+\varepsilon^2}}\nabla v.
\]
Consequently $g_\varepsilon(v)\to |v|$ in $W^{1,p}(\mathbb R^n)$, while dominated convergence gives $\mathcal Q[g_\varepsilon(v)]\to\mathcal Q[v]$ and $\int_Bg_\varepsilon(v)^p\,dx\to\int_B|v|^p\,dx$. For each $k$, choose $\varepsilon_k>0$ sufficiently small and normalize $g_{\varepsilon_k}(v_k)$ on $B$. The resulting nonnegative smooth functions $u_k$ satisfy $\int_Bu_k^p\,dx=1$ and $\mathcal Q[u_k]<1/k$.
\end{proof}

\subsection{Rigidity of nonlinear null sequences}
\begin{lemma}[Rigidity of nonlinear null sequences]\label{lem:rigidity}
Let $1<p<2$, let $U$ be a bounded connected Lipschitz domain, and let
$B\Subset U$ be a nonempty ball. Suppose that $w_k\in W^{1,p}(U)$,
$w_k\ge0$, and
\[
0<b_0\le\int_Bw_k^p\,dx\le b_1<\infty,
\qquad
J_k:=\int_U\frac{|\nabla w_k|^2}
{(|\nabla w_k|+w_k)^{2-p}}\,dx\longrightarrow0.
\]
The integrand is defined to be zero where both $w_k$ and $\nabla w_k$
vanish. Then, after passing to a subsequence, there exists $c>0$ such that
\[
w_k\longrightarrow c\qquad\text{strongly in }W^{1,p}(U).
\]
\end{lemma}

\begin{proof}
Set
\[
X_k:=\int_U |\nabla w_k|^p\,dx,
\qquad
Y_k:=\int_U w_k^p\,dx,
\qquad
Z_k:=\int_B w_k^p\,dx,
\]
and denote by
\[
m_k:=(w_k)_U:=\frac1{|U|}\int_U w_k\,dx
\]
the average of $w_k$ over $U$.
The anchored Poincar\'e inequality gives
\begin{equation}\label{eq:rigidity-anchor}
Y_k\le C(X_k+Z_k).
\end{equation}
Indeed, Poincar\'e's inequality gives
\[
\|w_k-m_k\|_{L^p(U)}^p\le C X_k.
\]
Moreover,
\[
|B|^{1/p}m_k
\le \|w_k\|_{L^p(B)}+\|w_k-m_k\|_{L^p(B)}
\le Z_k^{1/p}+C X_k^{1/p}.
\]
Thus $|U|m_k^p\le C(X_k+Z_k)$, and combining this with
$\|w_k\|_{L^p(U)}^p\le C\bigl(\|w_k-m_k\|_{L^p(U)}^p+|U|m_k^p\bigr)$
proves \eqref{eq:rigidity-anchor}.
H\"older's inequality, with exponents $2/p$ and $2/(2-p)$, yields
\[
\begin{aligned}
X_k
&=\int_U
\left(\frac{|\nabla w_k|^2}
{(|\nabla w_k|+w_k)^{2-p}}\right)^{p/2}
(|\nabla w_k|+w_k)^{p(2-p)/2}\,dx\\
&\le J_k^{p/2}
\left(\int_U(|\nabla w_k|+w_k)^p\,dx\right)^{(2-p)/2}\\
&\le C J_k^{p/2}(X_k+Y_k)^{(2-p)/2}\\
&\le C J_k^{p/2}(X_k+Z_k)^{(2-p)/2}.
\end{aligned}
\]
Set $\vartheta=(2-p)/2\in(0,1)$. Since $p/2=1-\vartheta$, Young's
inequality and $(X_k+Z_k)^\vartheta\le X_k^\vartheta+Z_k^\vartheta$ give
\[
X_k
\le \frac12X_k+CJ_k+CJ_k^{p/2}Z_k^{(2-p)/2}.
\]
Therefore
\begin{equation}\label{eq:rigidity-gradient}
X_k\le C\bigl(J_k+J_k^{p/2}Z_k^{(2-p)/2}\bigr)\longrightarrow0.
\end{equation}
Equation~\eqref{eq:rigidity-anchor} also bounds $Y_k$ uniformly.
Since $m_k\ge0$, Poincar\'e now gives
\[
\|w_k-m_k\|_{L^p(U)}\longrightarrow0.
\]
Restricting this convergence to $B$ and using
$b_0\le Z_k\le b_1$, we obtain
\[
b_0^{1/p}-o(1)\le |B|^{1/p}m_k\le b_1^{1/p}+o(1).
\]
After taking a subsequence, $m_k\to c>0$. Together with
\eqref{eq:rigidity-gradient}, this proves the asserted strong
$W^{1,p}(U)$ convergence.
\end{proof}

\subsection{Spherical reduction}
\begin{lemma}[Spherical Jensen inequality]\label{lem:sphericalJensen}
Let $1<p<2$ and define
\[
F(v,t):=t^2(t+v)^{p-2},\qquad v,t\ge0,
\]
with $F(0,0)=0$. Then $F$ is jointly convex on $[0,\infty)^2$ and
nondecreasing in its second variable. In particular, let $w\ge0$ be
$C^1$ on an annulus centered at $a_\alpha$, and set
\[
\overline w(r):=\frac1{|\mathbb S^{n-1}|}
\int_{\mathbb S^{n-1}}w(a_\alpha+r\theta)\,d\theta.
\]
For every radius in this annulus,
\[
\frac1{|\mathbb S^{n-1}|}\int_{\mathbb S^{n-1}}
F\left(\frac{w}{r},|\partial_rw|\right)\,d\theta
\ge F\left(\frac{\overline w(r)}r,|\overline w'(r)|\right).
\]
\end{lemma}

\begin{proof}
For $v,t>0$, direct differentiation gives
\[
F_t=t(t+v)^{p-3}(pt+2v)\ge0,
\]
and
\[
\begin{aligned}
F_{vv}&=(p-2)(p-3)t^2(t+v)^{p-4},\\
F_{vt}&=(p-2)t(t+v)^{p-4}\bigl((p-1)t+2v\bigr),\\
F_{tt}&=(t+v)^{p-4}
\bigl(p(p-1)t^2+4(p-1)tv+2v^2\bigr).
\end{aligned}
\]
Moreover,
\[
\det D^2F
=2(2-p)(p-1)t^2(t+v)^{2p-6}>0.
\]
Thus $F$ is convex on the open quadrant. Since $0\le F(v,t)\le t^p$,
it extends continuously to the closed quadrant, where it remains convex
and nondecreasing in $t$. Applying Jensen's inequality to normalized
surface measure gives
\[
\frac1{|\mathbb S^{n-1}|}\int_{\mathbb S^{n-1}}
F\left(\frac{w}{r},|\partial_rw|\right)\,d\theta
\ge F\left(\frac{\overline w(r)}r,
\frac1{|\mathbb S^{n-1}|}\int_{\mathbb S^{n-1}}|\partial_rw|\,d\theta\right).
\]
The conclusion follows from monotonicity in the second variable and
\[
|\overline w'(r)|\le
\frac1{|\mathbb S^{n-1}|}\int_{\mathbb S^{n-1}}|\partial_rw|\,d\theta.
\]
\end{proof}

\subsection{Strict-pole capacity}
\begin{lemma}[Strict-pole condenser]\label{lem:condenser}
Let $n\ge3$, $1<p<2$, and $\beta>(n-p)/p$, and put
\[
a:=p\beta-(n-p)>0.
\]
Fix $r_0>0$. There exists $\kappa_{p,a}>0$ such that, for every
$0<\rho<r_0$, $\delta>0$, and nonnegative absolutely continuous function
$f:[\rho,r_0]\to[0,\infty)$ satisfying
\[
f(\rho)\le\delta,\qquad f(r_0)\ge2\delta,
\]
one has
\[
\int_\rho^{r_0}r^{n-1-p\beta}
\frac{|f'(r)|^2}{(|f'(r)|+f(r)/r)^{2-p}}\,dr
\ge \kappa_{p,a}r_0^{-a}\delta^p.
\]
In particular, the lower bound is uniform in $\rho$.
\end{lemma}

\begin{proof}
This is a one-dimensional capacity estimate. After passing to logarithmic
radial variables, the prescribed drop from $2\delta$ to $\delta$ forces a
fixed amount of logarithmic variation. We estimate that variation separately
on the regions where the relative slope is at most one and greater than one.
Set
\[
t=\log\frac{r_0}{r},\qquad
y(t)=f(r_0e^{-t}),\qquad
T=\log\frac{r_0}{\rho}.
\]
Since $|f'(r)|=|y'(t)|/r$ and $dr=-r\,dt$, the integral in the statement
is
\[
I=r_0^{-a}\int_0^T e^{at}
\frac{|y'(t)|^2}{(|y'(t)|+y(t))^{2-p}}\,dt.
\]
We may assume that $I<\infty$. We have $y(0)\ge2\delta$ and
$y(T)\le\delta$. Let $\tau\in(0,T]$ be the first time at which
$y(\tau)=\delta$. Then $y\ge\delta$ on $[0,\tau]$. Define
\[
q(t):=\frac{|y'(t)|}{y(t)},\qquad 0<t<\tau.
\]
Absolute continuity of $\log y$ on this interval gives
\[
\int_0^\tau q(t)\,dt
\ge |\log y(\tau)-\log y(0)|\ge\log2.
\]
Also,
\[
\frac{|y'|^2}{(|y'|+y)^{2-p}}
=y^p\frac{q^2}{(1+q)^{2-p}}
\ge\delta^p\frac{q^2}{(1+q)^{2-p}}.
\]
Set $E_1=\{t\in(0,\tau):q(t)\le1\}$ and
$E_2=(0,\tau)\setminus E_1$. On these sets,
\[
\frac{q^2}{(1+q)^{2-p}}\ge 2^{p-2}
\begin{cases}
q^2,&t\in E_1,\\
q^p,&t\in E_2.
\end{cases}
\]
At least one of $\int_{E_1}q$ and $\int_{E_2}q$ is at least
$\ell:=(\log2)/2$.
If $\int_{E_1}q\ge\ell$, Cauchy--Schwarz yields
\[
\ell^2\le
\left(\int_{E_1}e^{at}q^2\,dt\right)
\left(\int_{E_1}e^{-at}\,dt\right)
\le\frac1a\int_{E_1}e^{at}q^2\,dt.
\]
If $\int_{E_2}q\ge\ell$, H\"older's inequality gives
\[
\ell^p\le
\left(\int_{E_2}e^{at}q^p\,dt\right)
\left(\int_{E_2}e^{-at/(p-1)}\,dt\right)^{p-1}
\le\left(\frac{p-1}{a}\right)^{p-1}
\int_{E_2}e^{at}q^p\,dt.
\]
Consequently,
\[
I\ge 2^{p-2}r_0^{-a}\delta^p
\min\left\{a\ell^2,
\left(\frac{a}{p-1}\right)^{p-1}\ell^p\right\}.
\]
The expression multiplying $r_0^{-a}\delta^p$ is strictly positive
and depends only on $p$ and $a$.
\end{proof}

\subsection{Complementary subcriticality}
\begin{proposition}[Uniform complementary subcriticality]\label{thmNCnonlin}
Let $n\ge3$, $1<p<n$, and
\[
\frac{n-p}{p}<s\le\frac{n-p}{p-1}.
\]
If $\mu\in\Psi\setminus\Psi_{\frac{n-p}{ps}}$, then
$H_{m,s,p}^{\mu}$ is subcritical in $\mathbb R^n$.
\end{proposition}

\begin{proof}[Proof of Proposition~\ref{thmNCnonlin}]
Assume, to the contrary, that $H_{m,s,p}^{\mu}$ is critical.
Since $\mu\notin\Psi_{\frac{n-p}{ps}}$, there is an index $\alpha$ such
that
\[
\beta:=s\mu_\alpha>\frac{n-p}{p},\qquad
a:=p\beta-(n-p)>0.
\]
Near $a_\alpha$,
\[
\phi(x)=|x-a_\alpha|^{-\beta}g_\alpha(x),
\]
where $g_\alpha$ is smooth and strictly positive. In particular, for
some $R>0$ and all $0<r<R$,
\[
C^{-1}r^{-\beta}\le\phi(a_\alpha+r\theta)\le Cr^{-\beta},
\qquad
\frac{|\nabla\phi(a_\alpha+r\theta)|}{\phi(a_\alpha+r\theta)}
\le\frac Cr,
\]
uniformly in $\theta\in\mathbb S^{n-1}$.

Fix a nonempty ball $B\Subset\Omega$. By Lemma~\ref{lem:localnull},
there exist nonnegative $u_k\in C_c^\infty(\mathbb R^n)$ such that
\[
\int_Bu_k^p\,dx=1,\qquad\mathcal Q[u_k]\longrightarrow0.
\]
Set $w_k=u_k/\phi$ on $\Omega$. For a measurable $D\subset\Omega$, put
\[
\mathcal R_D[u_k]:=\int_D C_p(\nabla u_k,\phi\nabla w_k)\,dx.
\]
We emphasize that $\mathcal R_D[u]$ is the localization of the
nonnegative ground-state remainder. In general,
\[
\mathcal R_D[u]\neq
\int_D\bigl(|\nabla u|^p-V|u|^p\bigr)\,dx,
\]
since the latter differs from $\mathcal R_D[u]$ by a boundary term. The
advantage of $\mathcal R_D$ is that the ground-state representation and
$C_p\ge0$ give
\[
0\le\mathcal R_D[u]\le\mathcal Q[u].
\]

Let $U$ be a bounded connected Lipschitz domain with
$B\Subset U\Subset\Omega$. On $U$ there are constants $m_U>0$ and
$L_U\ge1$ such that
\[
\phi\ge m_U,\qquad\frac{|\nabla\phi|}{\phi}\le L_U,
\qquad 0<b_U\le\int_Bw_k^p\,dx\le B_U<\infty.
\]
If $p\ge2$, the lower $C_p$ estimate gives
\[
\int_U|\nabla w_k|^p\,dx\le C_U\mathcal Q[u_k]\longrightarrow0.
\]
Anchored Poincar\'e and normalization then give convergence, after
extraction, to a positive constant in $W^{1,p}(U)$.
If $1<p<2$, the lower $C_p$ estimate and $p-2<0$ give
\[
\begin{aligned}
C_p(\nabla u_k,\phi\nabla w_k)
&\ge c_p\phi^p
\frac{|\nabla w_k|^2}
{(|\nabla w_k|+L_Uw_k)^{2-p}}\\
&\ge c_pm_U^pL_U^{p-2}
\frac{|\nabla w_k|^2}
{(|\nabla w_k|+w_k)^{2-p}}.
\end{aligned}
\]
Thus the hypothesis of Lemma~\ref{lem:rigidity} holds on $U$.
Let
\[
U_1\Subset U_2\Subset\cdots\Subset\Omega,
\qquad \bigcup_{j=1}^{\infty}U_j=\Omega,
\]
be a connected Lipschitz exhaustion with $B\Subset U_1$. Applying the preceding argument
successively on $U_j$ and using a diagonal extraction, we obtain a single
subsequence such that
\[
w_k\longrightarrow c_j>0\qquad\text{in }L^p(U_j)
\]
for every $j$. Since $U_j\subset U_{j+1}$, uniqueness of the $L^p$-limit
gives $c_j=c_{j+1}$; hence all $c_j$ coincide with a single constant
$c>0$. The normalization also identifies this constant through
\[
c^p\int_B\phi^p\,dx=1.
\]
Therefore
\[
w_k\longrightarrow c\qquad\text{in }L^p_{\rm loc}(\Omega).
\]

Choose $0<r_1<r_2<R$ so that $B_{r_2}(a_\alpha)$ contains no other
pole, and set
\[
A:=B_{r_2}(a_\alpha)\setminus\overline{B_{r_1}(a_\alpha)}.
\]
Then $A\Subset\Omega$. Since $w_k\to c$ in $L^p(A)$, after passing to a
further subsequence we may assume
\[
\|w_k-c\|_{L^p(A)}^p\le2^{-k},
\]
and hence
\[
\sum_{k=1}^{\infty}\|w_k-c\|_{L^p(A)}^p<\infty.
\]
By the polar-coordinate formula and Tonelli's theorem,
\[
\int_{r_1}^{r_2}r^{n-1}
\sum_{k=1}^{\infty}
\|w_k(a_\alpha+r\,\cdot)-c\|_{L^p(\mathbb S^{n-1})}^p\,dr<\infty.
\]
Therefore, for almost every $r\in(r_1,r_2)$, the series in $k$ is finite.
Fix one such radius $r_0$. Then
\[
\|w_k(a_\alpha+r_0\,\cdot)-c\|_{L^p(\mathbb S^{n-1})}\longrightarrow0.
\]
Define
\[
f_k(r):=\frac1{|\mathbb S^{n-1}|}
\int_{\mathbb S^{n-1}}w_k(a_\alpha+r\theta)\,d\theta.
\]
H\"older's inequality on $\mathbb S^{n-1}$ therefore gives
$f_k(r_0)\to c$. Fix $\delta\in(0,c/4)$, so that
$f_k(r_0)\ge2\delta$ for all sufficiently large $k$.
For each fixed $k$, smoothness of $u_k$ at $a_\alpha$ and
$\phi(x)\asymp|x-a_\alpha|^{-\beta}$ give
\[
w_k(x)=O_k(|x-a_\alpha|^\beta),
\]
and hence $f_k(r)\to0$ as $r\downarrow0$. We may therefore choose
\[
0<\rho_k<\min\{r_0,1/k\},\qquad f_k(\rho_k)\le\delta.
\]
In particular $\rho_k\to0$ and
$f_k(r_0)-f_k(\rho_k)\ge\delta$.

If $p\ge2$, the lower $C_p$ estimate, the local bounds for $\phi$, and
Jensen's inequality imply
\[
\mathcal Q[u_k]\ge C_1\int_{\rho_k}^{r_0}
r^{n-1-p\beta}|f_k'(r)|^p\,dr.
\]
Weighted H\"older gives
\[
\begin{aligned}
\delta
&\le\int_{\rho_k}^{r_0}|f_k'(r)|\,dr\\
&\le\left(\int_{\rho_k}^{r_0}
r^{n-1-p\beta}|f_k'(r)|^p\,dr\right)^{1/p}
\left(\int_0^{r_0}r^{-\frac{n-1-p\beta}{p-1}}\,dr
\right)^{(p-1)/p}.
\end{aligned}
\]
Since $a=p\beta-(n-p)>0$,
\[
\int_0^{r_0}r^{-\frac{n-1-p\beta}{p-1}}\,dr
=\frac{p-1}{a}\,r_0^{a/(p-1)}<\infty.
\]
Therefore
\[
\mathcal Q[u_k]\ge
C_1\left(\frac{a}{p-1}\right)^{p-1}r_0^{-a}\delta^p>0.
\]

If $1<p<2$, the local $C_p$ estimate gives
\[
\begin{aligned}
\mathcal Q[u_k]
&\ge C_2\int_{\rho_k}^{r_0}r^{n-1-p\beta}
\frac1{|\mathbb S^{n-1}|}\int_{\mathbb S^{n-1}}
F\left(\frac{w_k}{r},|\nabla w_k|\right)\,d\theta\,dr\\
&\ge C_2\int_{\rho_k}^{r_0}r^{n-1-p\beta}
F\left(\frac{f_k(r)}r,|f_k'(r)|\right)\,dr.
\end{aligned}
\]
Here the second inequality follows first from monotonicity in the
second argument of $F$, and then from
Lemma~\ref{lem:sphericalJensen}. Lemma~\ref{lem:condenser} now implies
\[
\mathcal Q[u_k]\ge C_2\kappa_{p,a}r_0^{-a}\delta^p>0.
\]
All lower bounds are independent of $k$. Both cases contradict
$\mathcal Q[u_k]\to0$. Consequently $H_{m,s,p}^{\mu}$ is subcritical.
\end{proof}

\subsection{The quadratic parameter comparison}
The quadratic decomposition permits a second proof of complementary subcriticality, independent of the strict-pole capacity argument.

\begin{proposition}[Quadratic parameter comparison]\label{thmNCquad}
Let $n\geq 3$, $s\in\bigl(\frac{n-2}{2},n-2\bigr]$, and $\mu\in\Psi\setminus\Psi_{\frac{n-2}{2s}}$. Then the operator $H_{m,s,2}^{\mu}$ is subcritical in $\mathbb{R}^{n}$.
\end{proposition}

\begin{proof}[Proof of Proposition~\ref{thmNCquad}]
Write $d=n-2$, $M=\mu_1=\max_\alpha\mu_\alpha>d/(2s)$, and
\[
\begin{aligned}
V_{m,s,2}^{\mu}&=s(d-s)W_m^\mu+s^2U_m^\mu,\\
W_m^\mu&=\sum_\alpha\frac{\mu_\alpha}{|x-a_\alpha|^2},\qquad
U_m^\mu=\sum_{\alpha<\beta}\mu_\alpha\mu_\beta|z_\alpha-z_\beta|^2,
\end{aligned}
\]
where $z_\alpha=(x-a_\alpha)/|x-a_\alpha|^2$.
We construct admissible parameters $t,\mu'$ for which
$V_{m,t,2}^{\mu'}>V_{m,s,2}^{\mu}$ away from the poles.
Theorem~\ref{thmMHC} then supplies a nontrivial nonnegative gap.

Choose $b<M$ and redistribute the remaining weights proportionally:
\[
\mu'_1=b,\qquad
\mu'_\alpha=\frac{1-b}{1-M}\mu_\alpha\quad(\alpha\ge2).
\]
If $s=d$, take $b=1/2$ and $t=d$. All active pairwise products
increase strictly, since $M>1/2$ and
$b(1-b)=1/4>M(1-M)$. The linear terms vanish, proving the claim.

If $d/2<s<d$, let $b<M$ be sufficiently close to $M$ and set
\[
t(b)=\frac{d+\sqrt{d^2-4(M/b)s(d-s)}}2,\qquad
F(b)=t(b)-s\sqrt{\frac{M(1-M)}{b(1-b)}}.
\]
These functions are defined near $M$, with $t(M)=s$, $F(M)=0$, and
\[
F'(M)=\frac{s(d-2sM)}{2M(1-M)(2s-d)}<0.
\]
Hence we may choose $d/(2s)<b<M$ such that
\[
s\sqrt{\frac{M(1-M)}{b(1-b)}}<t<s,\qquad
t(d-t)=\frac Mb\,s(d-s).
\]
The last identity and $\mu'_\alpha/\mu_\alpha=(1-b)/(1-M)>1$
for active $\alpha\ge2$ imply
$s(d-s)W_m^\mu\le t(d-t)W_m^{\mu'}$.
For every active pair,
\[
\frac{\mu_\alpha\mu_\beta}{\mu'_\alpha\mu'_\beta}
\le\frac{M(1-M)}{b(1-b)}<\frac{t^2}{s^2},
\]
so $s^2U_m^\mu<t^2U_m^{\mu'}$.
Thus the comparison holds in both cases. A nonzero nonnegative
smooth function supported away from the poles and bounded above
by the difference of the potentials proves subcriticality.
\end{proof}

\section{Variational thresholds and compactness for the truncated potential}\label{sec:variational}

Truncation replaces the prescribed full-potential ground state by the
variational problem \eqref{def:Cstar}. We first identify the costs of
concentration at the poles and infinity, then prove the compactness
principle in Theorem~\ref{prop:compactness-threshold}. The sign of the
deleted remainder subsequently determines the position of $C_*$ relative
to one. Throughout this section, $n\ge3$, $1<p<n$, $h\le s\le s_{\max}$,
and $\mu\in\Psi$, unless otherwise stated.

With $V_s:=V_{m,s,p}^{\mu}$ and the notation of the introduction,
recall the compact form
\begin{equation}\label{eq:W-compact-form}
W_s=s^{p-1}|v|^{p-2}\bigl((n-p)S-(p-1)s|v|^2\bigr),
\end{equation}
and $R_s=V_s-W_s$, the last term in \eqref{def:full-weight}.
We suppress $s$ when it is fixed. As in Section~\ref{sec:groundstate},
these are locally integrable representatives; pointwise calculations
are made away from the poles and $Z(v)$.

\subsection{Concentration costs at the singular ends}

\begin{proposition}[Asymptotics at an active pole]\label{prop:polecompare}
If $\mu_\alpha>0$, then, with $y=x-a_\alpha$,
\begin{equation}\label{poleexp-W}
W(x)=A_\alpha(p,s,\mu)|y|^{-p}+O(|y|^{1-p}),
\qquad A_\alpha(p,s,\mu)=g(s\mu_\alpha),
\end{equation}
and
\begin{equation}\label{poleexp-R}
R(x)=O(|y|^{2-p}).
\end{equation}
In particular,
\begin{equation}\label{poleexp-RW}
R(x)/W(x)\longrightarrow0\qquad(x\to a_\alpha),
\end{equation}
and for every $\sigma>0$ there is $r_{\sigma,\alpha}>0$ such that
\begin{equation}\label{polecompare}
R(x)\le\sigma W(x)\qquad(0<|x-a_\alpha|<r_{\sigma,\alpha}).
\end{equation}
\end{proposition}
\begin{proof}
Separating the singular term gives
\begin{equation}\label{poleexp-v}
v(x)=\mu_\alpha\frac{y}{|y|^2}+\widetilde w_\alpha+O(|y|),
\qquad
\widetilde w_\alpha=\sum_{\beta\ne\alpha}
\mu_\beta\frac{a_\alpha-a_\beta}{|a_\alpha-a_\beta|^2}.
\end{equation}
Thus $S=\mu_\alpha|y|^{-2}+O(1)$ and
$|v|^2=\mu_\alpha^2|y|^{-2}+O(|y|^{-1})$; substitution in
\eqref{eq:W-compact-form} proves \eqref{poleexp-W}.
For the $\alpha$-summand of $R$, the component of $v$ orthogonal
to $y$ is bounded, so its bracket is $O(|y|^{-2})$.
The other brackets have the same bound because $z_\beta$ is
bounded for $\beta\ne\alpha$. Since
$|v|^{p-4}=O(|y|^{4-p})$, this proves \eqref{poleexp-R}.
Finally, $0<s\mu_\alpha<s_{\max}$ for a genuinely multipolar
weight, hence $A_\alpha>0$, and the comparisons follow.
\end{proof}

\begin{proposition}[Asymptotics at infinity]\label{prop:inftycompare}
Translate the configuration so that
\begin{equation}\label{barycenter-new}
\sum_\alpha\mu_\alpha a_\alpha=0.
\end{equation}
Then
\begin{equation}\label{inftyexp-W}
W(x)=A_\infty(p,s)|x|^{-p}+O(|x|^{-p-2}),
\qquad A_\infty(p,s)=g(s),
\end{equation}
and
\begin{equation}\label{inftyexp-R}
R(x)=O(|x|^{-p-2}).
\end{equation}
If $g(s)>0$, then
\begin{equation}\label{inftyexp-RW}
R(x)/W(x)\longrightarrow0\qquad(|x|\to\infty),
\end{equation}
and, for every $\sigma>0$, some $R_\sigma>0$ satisfies
\begin{equation}\label{inftycompare}
R(x)\le\sigma W(x)\qquad(|x|\ge R_\sigma).
\end{equation}
\end{proposition}
\begin{proof}
For $x=r\omega$, $|\omega|=1$,
\[
z_\alpha=\frac{\omega}{r}
+\frac{2(\omega\cdot a_\alpha)\omega-a_\alpha}{r^2}
+O(r^{-3}).
\]
The barycentric normalization therefore gives
\begin{equation}\label{inftyexp-v}
v(x)=\frac{\omega}{r}+O(r^{-3}),
\end{equation}
as well as $S=r^{-2}+O(r^{-4})$. Equation~\eqref{eq:W-compact-form}
now yields \eqref{inftyexp-W}. Moreover,
$v\wedge(x-a_\alpha)=-r^{-1}\omega\wedge a_\alpha+O(r^{-2})$.
Each bracket in $R$ is consequently $O(r^{-6})$, proving
\eqref{inftyexp-R}. The remaining conclusions are immediate.
\end{proof}

\begin{lemma}[Upper bound by the concentration threshold]
\label{lem:concentration-upper}
For every admissible genuinely multipolar parameter,
$C_*(p,s,\mu)\le\Lambda(p,s,\mu)$.
\end{lemma}
\begin{proof}
For nonzero nonnegative $\chi\in C_c^\infty(0,1)$, the annular tests
$f_L(x)=|x|^{-h}\chi((\log|x|)/L)$ satisfy
\[
\frac{\int|\nabla f_L|^p}{\int|x|^{-p}|f_L|^p}
=\frac{\int_0^1|-h\chi+L^{-1}\chi'|^p\,dt}{\int_0^1\chi^p\,dt}
\xrightarrow{L\to\infty} H_p.
\]
For a fixed annular test $f$, the functions
$u_\rho(x)=\rho^{-h}f((x-a_\alpha)/\rho)$ have invariant energy, and
Proposition~\ref{prop:polecompare} gives
\[
\int W|u_\rho|^p\,dx
\xrightarrow{\rho\downarrow0}
g(s\mu_\alpha)\int|y|^{-p}|f(y)|^p\,dy.
\]
Taking first $\rho\downarrow0$ and then Hardy near-minimizers yields
$C_*\le H_p/g(s\mu_\alpha)$. The tests $R^{-h}f(x/R)$ similarly give
$C_*\le H_p/g(s)$ as $R\to\infty$ if $g(s)>0$. Taking the minimum proves the claim,
with the convention $H_p/0=+\infty$.
\end{proof}

\subsection{Weighted embedding and local compactness}

For $p\ge2$, the full Hardy inequality and $0\le W\le V$ give a
continuous embedding $D^{1,p}\hookrightarrow L^p(W\,dx)$.
The following comparison supplies its subquadratic counterpart.

\begin{theorem}[An explicit subquadratic parameter comparison]\label{tmHSW2}
Let $1<p<2$, and define
\begin{equation}\label{eq:explicit-subquad-bound}
c_0:=\frac{n+p-4}{n-p},\qquad
\tau_s:=c_0s+(1-c_0)s_{\max},\qquad
\underline C_p(s):=c_0\left(\frac{\tau_s}{s}\right)^{p-1}.
\end{equation}
Then
\begin{equation}\label{mHSW2}
\int_{\mathbb R^n}|\nabla u|^p\,dx
\ge\underline C_p(s)\int_{\mathbb R^n}W|u|^p\,dx
\qquad(u\in D^{1,p}(\mathbb R^n)).
\end{equation}
Moreover, $c_0\le\underline C_p(s)<1$, with equality in the first
inequality exactly when $s=s_{\max}$.
\end{theorem}
\begin{proof}
Put $d=2(2-p)$ and $M=\sum_\alpha\mu_\alpha(v\cdot z_\alpha)^2$.
The full potential at parameter $t$ is
\[
V_t=t^{p-1}|v|^{p-2}
\left((n-p-d)S-(p-1)t|v|^2+d\frac{M}{|v|^2}\right).
\]
Using $n-p-d=c_0(n-p)$ and
$(p-1)\tau_s=c_0(p-1)s+d$, we obtain
\begin{equation}\label{eq:subquad-comparison}
V_{\tau_s}-\underline C_p(s)W_s
=d\tau_s^{p-1}|v|^{p-4}(M-|v|^4)\ge0.
\end{equation}
Indeed, weighted Cauchy--Schwarz gives
$M\ge(\sum_\alpha\mu_\alpha v\cdot z_\alpha)^2=|v|^4$.
Since $\tau_s\in[s,s_{\max}]$, the full Hardy inequality proves
\eqref{mHSW2}, first for smooth tests and then by density and Fatou's
lemma. This also proves the continuous weighted embedding.

Let $b=1-c_0\in(0,1)$. The ratio $\tau_s/s$ decreases for $s\ge h$,
and therefore
\[
\underline C_p(s)\le
(1-b)\left(1+\frac{b}{p-1}\right)^{p-1}<1,
\]
where the strict inequality follows by taking logarithms and using
$\log(1-b)<-b$ and $\log(1+b/(p-1))<b/(p-1)$.
The comparison with $c_0$ follows from $\tau_s\ge s$.
\end{proof}

Thus $0<C_*<\infty$ for all admissible parameters: positivity follows
from the weighted embedding, and finiteness from a smooth test where
$W>0$. When $p<2$, the possible singularities on $Z(v)$ require
the following local estimate.

\begin{lemma}[Local compactness across the zero set]
\label{lem:weighted-local-compactness}
Let $\Omega_\mu=\mathbb R^n\setminus\{a_\alpha:\mu_\alpha>0\}$.
For every compact $K\Subset O\Subset\Omega_\mu$, with $O$ open,
and every $\varepsilon>0$,
\begin{equation}\label{eq:local-weight-form-smallness}
\int_KW|z|^p\,dx
\le\varepsilon\int_O|\nabla z|^p\,dx
+C_{K,O,\varepsilon}\int_O|z|^p\,dx
\qquad(z\in W^{1,p}(O)).
\end{equation}
Consequently, boundedness in $W^{1,p}(O)$ and convergence to zero
in $L^p(O)$ imply convergence to zero in $L^p(K,W\,dx)$.
\end{lemma}
\begin{proof}
For $p\ge2$, $W$ is locally bounded on $\Omega_\mu$. If $1<p<2$,
then, at every $x_0\in Z(v)$,
\[
\operatorname{tr}Dv(x_0)
=(n-2)\sum_{\mu_\alpha>0}\frac{\mu_\alpha}{|x_0-a_\alpha|^2}>0.
\]
As in Lemma~\ref{lem:Vloc}, one component $f$ of $v$ has nonzero
gradient there. Smooth coordinates $(y,t)$ with $t=f(x)$ therefore
give $|v|\ge|t|$ and $W\le C|t|^{p-2}$ almost everywhere in a
small cylinder $Q=Q'\times(-b,b)$. The coordinate maps and their
Jacobians are uniformly controlled on a smaller cylinder.
The one-dimensional Sobolev inequality on the slices gives
$\int_{Q'}\sup_{|t|<b}|\zeta(y,t)|^p\,dy
\le C\|\zeta\|_{W^{1,p}(Q)}^p$, hence
\begin{equation}\label{eq:zero-set-layer}
\int_{Q'}\int_{-\delta}^{\delta}|t|^{p-2}|\zeta|^p\,dt\,dy
\le C\delta^{p-1}\|\zeta\|_{W^{1,p}(Q)}^p.
\end{equation}
On $|t|\ge\delta$ the weight is bounded. Cover $K\cap Z(v)$ by
finitely many such charts, apply \eqref{eq:zero-set-layer} with
cut-offs supported in $O$, and choose $\delta$ so that the total
gradient coefficient is at most $\varepsilon$. Cut-off derivatives
contribute only to the $L^p$ term; the remaining compact set avoids
$Z(v)$ and has bounded weight. This proves
\eqref{eq:local-weight-form-smallness}. Its final consequence follows
by first taking the sequence limit and then $\varepsilon\downarrow0$.
\end{proof}

\subsection{The concentration-threshold principle}

Local weighted compactness leaves only the singular ends as possible
concentration sites. For $p\ge2$, a critical Hardy end also obstructs
finite-energy attainment. We establish this obstruction before proving
Theorem~\ref{prop:compactness-threshold}.

\begin{lemma}[Critical Hardy ends exclude finite energy]
\label{lem:critical-ends}\label{lem:critical-tail}
Let $2\le p<n$, $h=(n-p)/p$, and $H=h^p$.
\begin{enumerate}[(i)]
\item If $Z\in L^\infty_{\mathrm{loc}}(B_{R_*}(a)\setminus\{a\})$
and $Z(x)\ge H|x-a|^{-p}-C|x-a|^{1-p}$ there, every positive
continuous weak solution of $-\Delta_pu=Zu^{p-1}$ satisfies
$u(x)\ge c|x-a|^{-h}$ near $a$.
\item If $Z\in L^\infty_{\mathrm{loc}}(\mathbb R^n\setminus
\overline{B_{R_*}(0)})$ and $Z(x)\ge H|x|^{-p}-C|x|^{-p-2}$
there, every positive continuous weak solution satisfies
$u(x)\ge c|x|^{-h}$ for all sufficiently large $|x|$.
\end{enumerate}
In either case, the solution cannot extend to a function in
$D^{1,p}(\mathbb R^n)$.
\end{lemma}
\begin{proof}
We first recall positive-function Picone comparison on a bounded
annulus. If a positive subsolution $v$ lies strictly below a positive
solution $u$ on the boundary, use the tests
$(v^p-u^p)_+/v^{p-1}$ and $(v^p-u^p)_+/u^{p-1}$, respectively.
They are supported away from the boundary, where both denominators
are bounded below. Subtraction cancels the potential terms and
bounds the sum of two nonnegative Picone remainders by zero.
Equality would make $v/u$ constant on each component of $\{v>u\}$,
contrary to its boundary value one. Thus $v\le u$.

For (i), take $a=0$ and, with $A>0$ and $R<R_*$ to be chosen, set
\[
f_\rho(r)=r^{-h}(1+Ar)
\left(\frac{\log(r/\rho)}{\log(R/\rho)}\right)^2,\qquad \rho<r<R.
\]
Put $d=Ar/(1+Ar)$, $\xi=1/\log(r/\rho)$, and
$y=-rf_\rho'/f_\rho=h-d-2\xi$.
Differentiation with respect to $\log r$ gives
$\dot y=-d+d^2+2\xi^2$. The radial $p$-Laplacian formula yields
\begin{equation}\label{eq:inner-critical-barrier}
\begin{split}
r^p\frac{-\Delta_pf_\rho}{f_\rho^{p-1}}
&=|y|^{p-2}\left[B_0(y)+(p-1)d\left(y-h-1+\frac32d\right)\right],\\
B_0(y)&=\frac{p-1}{2}(h^2-y^2)+hy.
\end{split}
\end{equation}
At $y=0$ the identity is understood by continuity; the derivative
of $|y|^{p-2}y$ is zero there if $p>2$ and one if $p=2$.
For $F_0(y)=|y|^{p-2}B_0(y)$ and $0\le t\le1$,
\[
\frac{F_0(ht)}H
=t^{p-2}\left(\frac{p-1}{2}(1-t^2)+t\right),
\]
whose derivative is
$\frac{p-1}{2}t^{p-3}(1-t)(pt+p-2)\ge0$ on $(0,1)$.
For $-1\le t<0$, its value is smaller than at $-t$, and for
$t<-1$ it is negative. Thus $F_0(y)\le H$ for $y\le h$,
with equality only at $h$, and
$m_p:=\sup_{y\le h/2}F_0(y)/H<1$.
If $d\le1/2$, then $y\le h-d$ and
$y-h-1+3d/2\le-3/4$. Splitting into $y\ge h/2$ and $y\le h/2$
in \eqref{eq:inner-critical-barrier} gives
\[
r^p\frac{-\Delta_pf_\rho}{f_\rho^{p-1}}\le H-c_1d,\qquad
c_1=\min\left\{\frac34(p-1)(h/2)^{p-2},\,2H(1-m_p)\right\}>0.
\]
Choose $A\ge\max\{1,2C/c_1\}$ and $R<\min\{R_*,1/A\}$.
Then $d\ge Ar/2$, so $f_\rho$ is a subsolution for $Z$.
Since $f_\rho(R)=R^{-h}(1+AR)$ is independent of $\rho$,
a fixed multiple $cf_\rho$ lies strictly below $u$ on the outer
boundary and vanishes on the inner boundary. Comparison and
$\rho\downarrow0$ yield $u(x)\ge c|x|^{-h}(1+A|x|)$.

For (ii), choose $R$ large and set
\[
f_L(r)=r^{-h}(1+r^{-1})
\left(1-\frac{\log(r/R)}L\right),\qquad R<r<Re^L.
\]
Now $d=1/(r+1)$, $\xi=1/(L-\log(r/R))$, and
$y=-rf_L'/f_L=h+d+\xi$ satisfy $\dot y=-d+d^2+\xi^2$. Thus
\[
r^p\frac{-\Delta_pf_L}{f_L^{p-1}}
=\mathfrak f(y)-(p-1)y^{p-2}d(1+2\xi),
\qquad
\mathfrak f(y)=y^{p-2}\bigl((2-p)hy+(p-1)h^2\bigr).
\]
For $y\ge h$,
$\mathfrak f'(y)=(p-1)(p-2)hy^{p-3}(h-y)\le0$ and
$\mathfrak f(h)=H$. Consequently,
\[
r^p\frac{-\Delta_pf_L}{f_L^{p-1}}
\le H-\frac{(p-1)h^{p-2}}{r+1}
\le H-\frac C{r^2}\le r^pZ(x)
\]
for $R$ sufficiently large, independently of $L$.
A fixed multiple $cf_L$ lies below $u$ at $r=R$ and vanishes at
$r=Re^L$. Comparison followed by $L\to\infty$ gives
$u(x)\ge c|x|^{-h}(1+|x|^{-1})$.
Finally, $hp^*=n$ for $p^*=np/(n-p)$, so either lower bound
forces logarithmic divergence of the $L^{p^*}$ integral,
contradicting the Sobolev embedding of $D^{1,p}$.
\end{proof}

\begin{proof}[Proof of Theorem~\ref{prop:compactness-threshold}]
Write $c_*=C_*(p,s,\mu)$ and take a normalized minimizing sequence
$(u_k)$. After extraction,
$u_k\rightharpoonup u$ in $D^{1,p}$ and $u_k\to u$ in $L^p_{\rm loc}$,
and, as finite Radon measures tested against $C_0(\mathbb R^n)$,
\[
\nu_k:=W|u_k|^p\,dx\stackrel{*}{\rightharpoonup}\nu,\qquad
\lambda_k:=|\nabla u_k|^p\,dx\stackrel{*}{\rightharpoonup}\lambda.
\]
Lemma~\ref{lem:weighted-local-compactness} gives weighted $L^p$
convergence on compact sets away from active poles, including across
$Z(v)$ and inactive poles. Hence
\[
\nu=W|u|^p\,dx+\sum_{\mu_\alpha>0}\nu_\alpha\delta_{a_\alpha},
\qquad \nu_\alpha\ge0.
\]
Set $t=\int W|u|^p$, $\nu_\infty=1-\nu(\mathbb R^n)$, and
$\lambda_\infty=c_*-\lambda(\mathbb R^n)$; all are nonnegative.
Localized weak lower semicontinuity gives
$\lambda\ge|\nabla u|^p\,dx$. Since the latter measure has no atoms,
\[
\lambda=|\nabla u|^p\,dx
+\sum_{\mu_\alpha>0}\lambda_\alpha\delta_{a_\alpha}
+\lambda_{\rm rem},\qquad
\lambda_\alpha=\lambda(\{a_\alpha\}),\quad\lambda_{\rm rem}\ge0.
\]
Thus the mass and energy defects obey
\begin{equation}\label{eq:threshold-mass-balance}
\begin{aligned}
1&=t+\sum_{\mu_\alpha>0}\nu_\alpha+\nu_\infty,\\
c_*&=\int|\nabla u|^p\,dx+\sum_{\mu_\alpha>0}\lambda_\alpha
+\lambda_{\rm rem}(\mathbb R^n)+\lambda_\infty.
\end{aligned}
\end{equation}
In particular, discarding $\lambda_{\rm rem}$ gives the corresponding
energy inequality; no equality between mass and energy defects has
yet been assumed.

Suppose first that $c_*<\Lambda$, and fix $q$ with $c_*<q<\Lambda$.
The end asymptotics give $W\le(H_p/q)|x-a_\alpha|^{-p}$ in small
active-pole balls and $W\le(H_p/q)|x|^{-p}$ outside a large ball.
The classical Hardy inequality therefore implies
\begin{equation}\label{eq:threshold-local-Hardy}
q\int W|z|^p\,dx\le\int|\nabla z|^p\,dx
\end{equation}
for every $z\in D^{1,p}$ supported in one of these regions.
For $0<\rho<1$, let $\phi_{\alpha,\rho}$ be one on
$B_\rho(a_\alpha)$, zero outside $B_{\sqrt\rho}(a_\alpha)$,
and affine in $\log|x-a_\alpha|$ on the intervening annulus.
Likewise, let $\chi_L$ be zero on $B_L(0)$, one outside $B_{L^2}(0)$,
and affine in $\log|x|$ between these radii. These cut-offs satisfy
\[
\|\nabla\phi_{\alpha,\rho}\|_{L^n}^n
\le C|\log\rho|^{1-n},\qquad
\|\nabla\chi_L\|_{L^n}^n\le C(\log L)^{1-n}.
\]
Applying \eqref{eq:threshold-local-Hardy} to $\psi u_k$, with either
cut-off and its scale sufficiently small or large, gives
\begin{equation}\label{eq:threshold-cutoff}
q\int\psi^p\,d\nu_k
\le(1+\varepsilon)\int\psi^p\,d\lambda_k
+C_\varepsilon\|\nabla\psi\|_{L^n}^p.
\end{equation}
Here the constant is uniform in $k$ and the scale, since
$\int|u_k|^p|\nabla\psi|^p
\le\|u_k\|_{L^{p^*}}^p\|\nabla\psi\|_{L^n}^p$.
For pole cut-offs, take $k\to\infty$ and then $\rho\downarrow0$.
For exterior cut-offs, first write
\[
\int\chi_L^p\,d\nu_k=1-\int(1-\chi_L^p)\,d\nu_k,\qquad
\int\chi_L^p\,d\lambda_k
=\int|\nabla u_k|^p-\int(1-\chi_L^p)\,d\lambda_k.
\]
The subtracted functions are continuous and compactly supported,
so $k\to\infty$ followed by $L\to\infty$ is legitimate.
Finally $\varepsilon\downarrow0$ yields
$\lambda_\alpha\ge q\nu_\alpha$ and
$\lambda_\infty\ge q\nu_\infty$.
Since $\int|\nabla u|^p\ge c_*t$, \eqref{eq:threshold-mass-balance}
implies $c_*\ge c_*t+q(1-t)$, and therefore $t=1$.
It follows that $\int|\nabla u|^p=c_*$ and $u$ is a normalized
minimizer. Weak convergence of the gradients and convergence of
their $L^p$ norms give strong convergence by uniform convexity.

The positivity, equation, and criticality conclusions require only
attainment. Replace a normalized minimizer by its absolute value
$u_*\ge0$. The continuous weighted embedding permits differentiation
of the constraint, giving \eqref{eq:subquad-extremizer-equation}.
Its right-hand side is locally integrable, since
\[
\int_KWu_*^{p-1}
\le\left(\int_KWu_*^p\right)^{(p-1)/p}
\left(\int_KW\right)^{1/p}<\infty
\]
for compact $K$. Thus $u_*$ is a nontrivial nonnegative weak
$p$-superharmonic function, and the weak Harnack inequality
\cite[Corollary~3.18]{Lindqvist17} gives $u_*>0$ a.e.
Choose smooth compactly supported $v_j\to u_*$ in $D^{1,p}$ and,
after extraction, a.e. Weighted convergence gives
$\int|\nabla v_j|^p-c_*\int W|v_j|^p\to0$.
The optimal form is nonnegative, so Lemma~\ref{lem:crit}
proves its criticality.

For necessity when $p\ge2$, suppose $c_*=\Lambda$ is attained.
The preceding equation and local regularity for the locally bounded
potential $W$ give a positive continuous solution away from the
active poles. The minimum defining $\Lambda$ has a finite-cost
channel. At a minimizing pole,
$c_*W=H_p|x-a_\alpha|^{-p}+O(|x-a_\alpha|^{1-p})$;
at minimizing infinity, after barycentric translation,
$c_*W=H_p|x|^{-p}+O(|x|^{-p-2})$.
Either alternative contradicts Lemma~\ref{lem:critical-ends}.

Finally, let $p\ge2$ and $c_*=\Lambda$, and return to the measure
decomposition for an arbitrary normalized minimizing sequence.
Write $\Lambda_\alpha=H_p/g(s\mu_\alpha)$ and
$\Lambda_\infty=H_p/g(s)$.
Applying the cut-off argument at each end with any number below its
individual cost gives
$\lambda_\alpha\ge\Lambda_\alpha\nu_\alpha$ and
$\lambda_\infty\ge\Lambda_\infty\nu_\infty$ if the latter is finite.
If $g(s)=0$, arbitrary finite exterior constants give $\nu_\infty=0$.
Consequently,
\[
\Lambda\ge\int|\nabla u|^p+\sum_{\mu_\alpha>0}\lambda_\alpha
+\lambda_\infty
\ge c_*t+\Lambda\left(\sum_{\mu_\alpha>0}\nu_\alpha+\nu_\infty\right)
=\Lambda.
\]
Equality throughout gives $\int|\nabla u|^p=c_*t$.
If $t>0$, normalization would produce an extremizer, already
excluded. Thus $t=0$, $u=0$, and
$\lambda_{\rm rem}=0$, $\lambda_\alpha=\Lambda\nu_\alpha$,
$\lambda_\infty=\Lambda\nu_\infty$. Every channel of cost strictly
larger than $\Lambda$ has zero mass, and
\begin{equation}\label{eq:threshold-equality-concentration}
\begin{aligned}
W|u_k|^p\,dx&\stackrel{*}{\rightharpoonup}
\sum_{\mu_\alpha>0}\nu_\alpha\delta_{a_\alpha},\\
|\nabla u_k|^p\,dx&\stackrel{*}{\rightharpoonup}
\Lambda\sum_{\mu_\alpha>0}\nu_\alpha\delta_{a_\alpha},\\
\sum_{\mu_\alpha>0}\nu_\alpha+\nu_\infty&=1.
\end{aligned}
\end{equation}
The weighted mass and energy lost at infinity are respectively
$\nu_\infty$ and $\Lambda\nu_\infty$. Since this applies to every
weakly convergent subsequence, the whole sequence converges weakly
to zero in $D^{1,p}$.
\end{proof}

\begin{remark}[Variational spectral threshold]
\label{cor:variational-spectral-threshold}
For $Q_\lambda[u]=\int|\nabla u|^p-\lambda\int W|u|^p$,
the definition of $c_*=C_*$ gives
\[
Q_\lambda[u]\ge(1-\lambda/c_*)\int|\nabla u|^p
\qquad(0\le\lambda<c_*).
\]
At $\lambda=c_*$ the form is nonnegative but not coercive, as a
normalized minimizing sequence shows. For $\lambda>c_*$, a quotient
below $\lambda$, followed by smooth approximation, gives a negative
direction in $C_c^\infty$. Whenever the infimum is attained, in
particular if $c_*<\Lambda$, Theorem~\ref{prop:compactness-threshold}
provides a positive eigenfunction and a critical optimal operator.
Moreover, testing any nonzero finite-energy solution of
$-\Delta_pu=\lambda W|u|^{p-2}u$ with $u$, justified by density and
the weighted embedding, gives $\lambda\ge c_*$. Thus an attained
$c_*$ is the lowest variational eigenvalue.
\end{remark}

\subsection{The sign of the remainder and coefficient one}

The concentration principle reduces the coefficient-one statements
to the sign of $R$. For later use, if $p\ne2$, $|R|$ is positive
on some ball avoiding the poles and $Z(v)$: choose two active poles.
Off their joining line, their two vectors $x-a_\alpha$ cannot both
be parallel to $v\ne0$. Since $v$ is a nontrivial analytic field,
one can choose such a point outside $Z(v)$, and continuity gives
the asserted ball.

\begin{theorem}[Critical sharpness for the truncated potential]\label{tmHSW1}
Let $2\le p<n$ and $\mu\in\Psi_{h/s}$. Then
\begin{equation}\label{mHSW1}
\int_{\mathbb R^n}|\nabla u|^p\,dx
\ge\int_{\mathbb R^n}W|u|^p\,dx
\qquad(u\in D^{1,p}(\mathbb R^n)).
\end{equation}
The coefficient one is optimal in the non-attainment cases
$s=h$ or $\max_\alpha s\mu_\alpha=h$.
\end{theorem}
\begin{proof}
For $p\ge2$, $R\ge0$, so the full Hardy inequality gives
$C_*\ge1$. In the stated non-attainment cases one concentration
coefficient equals $H_p$, hence $\Lambda=1$.
Lemma~\ref{lem:concentration-upper} gives $1\le C_*\le\Lambda=1$.
\end{proof}

\begin{theorem}[Strict multiplicative improvement in the strict part]
\label{tmHSW3}
Let $2<p<n$, $h<s\le s_{\max}$, and
$\mu\in(\Psi_{h/s})^\circ$. Then
\begin{equation}\label{strict-Cstar}
C_*(p,s,\mu)>1.
\end{equation}
\end{theorem}
\begin{proof}
All concentration coefficients are strictly below $H_p$, so
$\Lambda>1$. If $C_*=1$, Theorem~\ref{prop:compactness-threshold}
gives a positive normalized minimizer $u_*$. But $R\ge0$ and is
positive on a ball, while the full form is nonnegative. Hence
\[
0=\int|\nabla u_*|^p-\int Wu_*^p
=\left(\int|\nabla u_*|^p-\int Vu_*^p\right)+\int Ru_*^p>0,
\]
a contradiction. Together with $C_*\ge1$, this proves the claim.
\end{proof}

\begin{corollary}[Strictness of the explicit subquadratic comparison]
\label{cor:strict-subquad-comparison}
Under the assumptions of Theorem~\ref{tmHSW2},
\begin{equation}\label{eq:global-strict-subquad-bound}
c_0\le\underline C_p(s)<C_*(p,s,\mu)\le\Lambda(p,s,\mu).
\end{equation}
\end{corollary}
\begin{proof}
Only strictness requires proof. If $C_*=\underline C_p(s)<1\le\Lambda$,
Theorem~\ref{prop:compactness-threshold} supplies a positive normalized
minimizer $u_*$. At any active pole, the notation of
\eqref{eq:subquad-comparison} gives, for $r=|x-a_\alpha|\downarrow0$,
\[
M-|v|^4=\mu_\alpha^3(1-\mu_\alpha)r^{-4}+O(r^{-3})>0.
\]
Thus \eqref{eq:subquad-comparison} is strict on a set of positive
measure, and the full Hardy inequality implies
\[
C_*=\int|\nabla u_*|^p
\ge\int V_{\tau_s}u_*^p
>\underline C_p(s)\int W_su_*^p=\underline C_p(s),
\]
a contradiction. All integrals are finite by the full and truncated
Hardy inequalities.
\end{proof}
In particular, the explicit comparison form is coercive with constant
$1-\underline C_p(s)/C_*>0$, including at $s=s_{\max}$, where
$\underline C_p(s)=c_0$. By the Sobolev inequality it therefore admits
a nonzero compactly supported nonnegative improvement.

\begin{theorem}[Subquadratic attainment on the full critical set]
\label{thm:subquad-sharp}
Let $1<p<2$ and $\mu\in\Psi_{h/s}$. Then
\begin{equation}\label{eq:subquad-sharp-interval}
\underline C_p(s)<C_*(p,s,\mu)<1\le\Lambda(p,s,\mu).
\end{equation}
The infimum has a positive normalized minimizer satisfying
\eqref{eq:subquad-extremizer-equation}, normalized minimizing sequences
are relatively compact in $D^{1,p}$, and the optimally rescaled
operator is critical.
\end{theorem}
\begin{proof}
Here $W-V=-R\ge0$ and is positive on a ball away from the singular
sets. Choose a nonzero $\eta\in C_c^\infty$ supported there with
$0\le\eta\le W-V$. Criticality of the full form
(Theorem~\ref{thmopt}) gives a smooth compactly supported test $u$
such that
\[
\int|\nabla u|^p-\int V|u|^p
<\int\eta|u|^p\le\int(W-V)|u|^p.
\]
Thus $C_*<1\le\Lambda$. Theorem~\ref{prop:compactness-threshold}
gives attainment, compactness, and criticality; the strict lower
bound is Corollary~\ref{cor:strict-subquad-comparison}.
\end{proof}

This conclusion includes the boundary parameters at which the full
ground state has infinite energy. It characterizes $C_*$ variationally,
without asserting a closed formula for the coefficient or its extremizer.

\section{Parameter transitions and concentration phases}
\label{sec:parameter-phases}

We apply Section~\ref{sec:variational} to weight degeneration,
the transition through $p=2$, and the superquadratic concentration
and compactness phases. Pole configurations are arbitrary unless
otherwise specified.

\subsection{Multipolar degeneration and coexistence of variational regimes}
\label{sec:multipolar-regimes}
To prove Theorem~\ref{thm:multipolar-three-regimes}, we connect
$C_*<1$ on the full critical boundary to a coercive region near a
simplex vertex, using a one-pole limit and continuity in the weights.
The poles remain fixed: only the weights degenerate, staying
positive before the vertex is reached.

In this subsection, let $n\ge3$, $1<p<2$, and recall
\[
h=\frac{n-p}{p},\qquad H_p=h^p,\qquad
s_{\max}=\frac{n-p}{p-1},\qquad h\le s\le s_{\max}.
\]
We use the Hardy profile $g$ in \eqref{eq:hardy-profile}, whose
unique maximum on $[0,s_{\max}]$ is $g(h)=H_p$.
The next result allows arbitrary fixed positive proportions:
the ratio bound in Theorem~\ref{thm:multipolar-three-regimes}
is needed only to place its starting weights in the full critical set.

\begin{proposition}[Degeneration to a single pole]
\label{prop:multipolar-one-pole}
Fix distinct poles $a_1,\ldots,a_m$, with $m\ge2$, and any
index $i\in\{1,\ldots,m\}$. Let $\rho_\alpha^{(i)}>0$,
$\alpha\ne i$, be fixed proportions satisfying
$\sum_{\alpha\ne i}\rho_\alpha^{(i)}=1$. For $0<\varepsilon<1$, set
\[
\mu^\varepsilon_\alpha:=
\begin{cases}
1-\varepsilon,&\alpha=i,\\[1mm]
\varepsilon\rho_\alpha^{(i)},&\alpha\ne i,
\end{cases}
\qquad
W_\varepsilon:=W_{m,s,p}^{\mu^\varepsilon},\quad
C_\varepsilon:=C_*(p,s,\mu^\varepsilon).
\]
There exist constants $B>0$ and $\varepsilon_0\in(0,1)$,
depending on $n,p,s$, the poles, the chosen index $i$,
and the fixed proportions, such that
for $0<\varepsilon<\varepsilon_0$ and
$u\in D^{1,p}(\mathbb R^n)$,
\begin{equation}\label{eq:multipolar-vertex-form}
\int_{\mathbb R^n}
\left|W_\varepsilon-g(s)|x-a_i|^{-p}\right||u|^p\,dx
\le \frac{B}{H_p}\varepsilon^{p-1}
\int_{\mathbb R^n}|\nabla u|^p\,dx.
\end{equation}
In particular,
\begin{equation}\label{eq:multipolar-vertex-bounds}
\frac{H_p}{g(s)+B\varepsilon^{p-1}}
\le C_\varepsilon
\le\frac{H_p}{g(s(1-\varepsilon))}.
\end{equation}
If $s<s_{\max}$, then also $C_\varepsilon\le H_p/g(s)$.
Consequently,
\begin{equation}\label{eq:multipolar-vertex-limit}
\lim_{\varepsilon\downarrow0}C_\varepsilon
=
\begin{cases}
\displaystyle\frac{H_p}{s^{p-1}(n-p-(p-1)s)},
&h\le s<s_{\max},\\[2mm]
+\infty,&s=s_{\max}.
\end{cases}
\end{equation}
The finite limit is strictly larger than $1$ whenever $s>h$.
\end{proposition}

\begin{proof}
Relabeling the poles and the corresponding proportions, we may
assume that the distinguished index is $1$. Write
$\rho_\alpha:=\rho_\alpha^{(1)}$ for $\alpha\ge2$.
Set $r_i=|x-a_i|$, $z_i=(x-a_i)/r_i^2$, and
\[
v_\varepsilon=(1-\varepsilon)z_1+
\varepsilon\sum_{i=2}^m\rho_i z_i,\qquad
S_\varepsilon=(1-\varepsilon)r_1^{-2}+
\varepsilon\sum_{i=2}^m\rho_i r_i^{-2}.
\]
Then
\[
W_\varepsilon=s^{p-1}|v_\varepsilon|^{p-2}
\bigl((n-p)S_\varepsilon-(p-1)s|v_\varepsilon|^2\bigr).
\]
All estimates below hold a.e., uniformly for small $\varepsilon$
with the proportions fixed.

Choose $\delta>0$ so that the balls $B_\delta(a_i)$,
$i=2,\ldots,m$, are pairwise disjoint and avoid $a_1$.
In each such ball set
\[
q_{i,\varepsilon}=(1-\varepsilon)z_1+
\varepsilon\sum_{j\ne1,i}\rho_jz_j.
\]
After reducing $\delta$ and the upper bound on $\varepsilon$,
these vector fields satisfy
\[
0<c\le |q_{i,\varepsilon}|\le C,\qquad
|Dq_{i,\varepsilon}|\le C.
\]
Since $q_{i,\varepsilon}$ stays away from zero, the map
$q_{i,\varepsilon}/|q_{i,\varepsilon}|^2$ has a uniformly bounded
$C^1$ norm. Hence
\[
F_{i,\varepsilon}(x)
=x-a_i+\varepsilon\rho_i
\frac{q_{i,\varepsilon}(x)}{|q_{i,\varepsilon}(x)|^2}
\]
satisfies $\|DF_{i,\varepsilon}-I\|_{L^\infty}\le C\varepsilon$.
Integrating along segments in the convex ball $B_\delta(a_i)$ gives
\[
(1-C\varepsilon)|x-y|
\le |F_{i,\varepsilon}(x)-F_{i,\varepsilon}(y)|
\le(1+C\varepsilon)|x-y|.
\]
In particular, $F_{i,\varepsilon}$ is injective there for small
$\varepsilon$.

To locate its zero, set
$T_{i,\varepsilon}(x):=x-F_{i,\varepsilon}(x)$.
Choose $M>\max_{i\ge2}\rho_i/c$. For sufficiently small
$\varepsilon$, we have $M\varepsilon<\delta$, and
$T_{i,\varepsilon}$ maps
$\overline B_{M\varepsilon}(a_i)\subset B_\delta(a_i)$ into itself
and has Lipschitz constant at most $C\varepsilon<1$.
The contraction theorem gives a fixed point $b_{i,\varepsilon}$
with $|b_{i,\varepsilon}-a_i|\le M\varepsilon$; injectivity of
$F_{i,\varepsilon}$ makes it the unique zero in the whole ball.
Moreover,
$q_{i,\varepsilon}(b_{i,\varepsilon})=z_1(a_i)+O(\varepsilon)$,
so the fixed-point equation yields
\[
b_{i,\varepsilon}
=a_i-\varepsilon\rho_i
\frac{z_1(a_i)}{|z_1(a_i)|^2}+O(\varepsilon^2),\qquad
|b_{i,\varepsilon}-a_i|\asymp\varepsilon,
\]
the last assertion using $\rho_i>0$ and $z_1(a_i)\ne0$.

Taking $y=b_{i,\varepsilon}$ in the two-sided estimate and using
the exact identity
\[
|v_\varepsilon(x)|
=\frac{|q_{i,\varepsilon}(x)|}{r_i}
|F_{i,\varepsilon}(x)|
\]
gives
\begin{equation}\label{eq:multipolar-vertex-zero}
|v_\varepsilon(x)|\asymp
\frac{|x-b_{i,\varepsilon}|}{r_i}
\qquad (x\in B_\delta(a_i)\setminus\{a_i\}).
\end{equation}
In particular, almost everywhere in this ball,
\[
|v_\varepsilon(x)|^{p-2}
\asymp r_i^{\,2-p}|x-b_{i,\varepsilon}|^{p-2}.
\]
Fix a constant $K>2M$, independent of
$\varepsilon$. In the inner region $r_i\le K\varepsilon$,
we have $S_\varepsilon\le C\varepsilon r_i^{-2}$.
Since $W_\varepsilon\ge0$ and $p<2$,
\eqref{eq:multipolar-vertex-zero} gives
\begin{equation}\label{eq:multipolar-vertex-inner}
W_\varepsilon
\le C\varepsilon r_i^{-p}|x-b_{i,\varepsilon}|^{p-2}
\le C\varepsilon^{p-1}
\bigl(r_i^{-p}+|x-b_{i,\varepsilon}|^{-p}\bigr).
\end{equation}
For the second inequality, put
$d_i=|b_{i,\varepsilon}-a_i|\asymp\varepsilon$.
If $r_i\le d_i/2$, then $|x-b_{i,\varepsilon}|\ge d_i/2$,
which gives the first term on the right. Otherwise
$r_i\ge c\varepsilon$ and
$|x-b_{i,\varepsilon}|\le C\varepsilon$; using $2p-2>0$
gives the second term. The classical Hardy inequality with
centers $a_i$ and $b_{i,\varepsilon}$ thus bounds the integral
of $W_\varepsilon|u|^p$ over the inner region by
$C\varepsilon^{p-1}\int|\nabla u|^p$.
The limiting weight $W_0:=g(s)r_1^{-p}$ is bounded on
$B_\delta(a_i)$, and the Sobolev inequality bounds its
contribution on $B_{K\varepsilon}(a_i)$ by
$C\varepsilon^p\int|\nabla u|^p$.

On the outer annulus
$A_{i,\varepsilon}:=\{x:K\varepsilon\le r_i<\delta\}$,
the bound $|b_{i,\varepsilon}-a_i|\le M\varepsilon$ gives
$\frac12r_i\le|x-b_{i,\varepsilon}|\le\frac32r_i$.
Thus \eqref{eq:multipolar-vertex-zero} yields
$0<c\le|v_\varepsilon|\le C$; the same bounds hold for $|z_1|$
since $B_\delta(a_i)$ avoids $a_1$. Both $t^{p-2}$ and $t^p$
are therefore Lipschitz on the relevant fixed interval in
$(0,\infty)$, including the negative exponent $p-2$.
Using
\[
|v_\varepsilon-z_1|\le C\varepsilon(1+r_i^{-1}),\qquad
|S_\varepsilon-|z_1|^2|
\le C\varepsilon(1+r_i^{-2}),
\]
the formula for $W_\varepsilon$ gives
\[
\begin{aligned}
|W_\varepsilon-W_0|
&\le C\bigl(|S_\varepsilon-|z_1|^2|
             +|v_\varepsilon-z_1|\bigr)\\
&\le C\varepsilon(1+r_i^{-2})
\le C\varepsilon r_i^{-2},
\end{aligned}
\]
where the last step uses $r_i<\delta$. Since $2-p>0$,
\[
\varepsilon r_i^{-2}
=\varepsilon^{p-1}r_i^{-p}
 \left(\frac{\varepsilon}{r_i}\right)^{2-p}
\le K^{p-2}\varepsilon^{p-1}r_i^{-p}.
\]
Hardy's inequality centered at $a_i$ consequently yields
\[
\int_{A_{i,\varepsilon}}|W_\varepsilon-W_0||u|^p\,dx
\le C\varepsilon^{p-1}
\int_{\mathbb R^n}|\nabla u|^p\,dx,
\]
uniformly for $u\in D^{1,p}(\mathbb R^n)$.

On the complement of the fixed balls $B_\delta(a_i)$,
$i=2,\ldots,m$, we have, uniformly up to $a_1$ and infinity,
\[
r_1|v_\varepsilon-z_1|\le C\varepsilon,
\qquad r_1^2S_\varepsilon=1+O(\varepsilon).
\]
Substitution into the expression for $W_\varepsilon$ yields
\[
|W_\varepsilon-W_0|\le C\varepsilon r_1^{-p}.
\]
Hardy's inequality at $a_1$, summation, and density prove
\eqref{eq:multipolar-vertex-form}. Combining it with Hardy gives
the lower bound in \eqref{eq:multipolar-vertex-bounds};
Lemma~\ref{lem:concentration-upper} gives the upper bounds
at $a_1$ and infinity. These bounds imply
\eqref{eq:multipolar-vertex-limit}, including divergence when
$g(s_{\max})=0$. The finite limit exceeds one for $s>h$
because $g$ has its unique maximum at $h$.
\end{proof}

\paragraph{Quantitative consequences.}
For $s<s_{\max}$, the estimates give
\begin{equation}\label{eq:multipolar-vertex-error}
0\le\frac{H_p}{g(s)}-C_\varepsilon
\le
\frac{H_pB\varepsilon^{p-1}}
{g(s)\bigl(g(s)+B\varepsilon^{p-1}\bigr)}.
\end{equation}
Thus the error is $O(\varepsilon^{p-1})$ along each fixed choice
of proportions.
At $s=s_{\max}$, the bounds instead have orders
$\varepsilon^{-(p-1)}$ and $\varepsilon^{-1}$, since
$g(s_{\max}(1-\varepsilon))
=(n-p)s_{\max}^{p-1}(1-\varepsilon)^{p-1}\varepsilon$;
no precise divergence rate is asserted.
For $h<s\le s_{\max}$, the condition
\begin{equation}\label{eq:multipolar-coercive-neighborhood}
0<\varepsilon<
\min\left\{\varepsilon_0,\ 1-\frac hs,\
\left(\frac{H_p-g(s)}{B}\right)^{1/(p-1)}\right\}
\end{equation}
implies $s(1-\varepsilon)>h$ and $C_\varepsilon>1$.
The constants depend on the configuration and approach direction.
The full potential is subcritical by
Theorem~\ref{thmopt}, while the coefficient-one truncated form
satisfies
\begin{equation}\label{eq:multipolar-coercivity}
\int_{\mathbb R^n}|\nabla u|^p\,dx
-\int_{\mathbb R^n}W_\varepsilon|u|^p\,dx
\ge\left(1-\frac1{C_\varepsilon}\right)
\int_{\mathbb R^n}|\nabla u|^p\,dx.
\end{equation}
The Sobolev inequality turns this positive energy bound into a
nonzero nonnegative compactly supported improvement, proving
subcriticality of the truncated operator as well.

\begin{lemma}[Continuity in the multipolar weights]
\label{lem:weight-continuity}
Fix $n\ge3$, $1<p<2$, an admissible parameter $s$, and
$m\ge2$ distinct poles $a_1,\ldots,a_m$. Write
$W_\mu:=W_{m,s,p}^\mu$. The map
$\mu\mapsto C_*(p,s,\mu)$ is continuous on the open weight
simplex $\{\mu\in(0,1)^m:\sum_\alpha\mu_\alpha=1\}$.
More precisely, for every $\mu^0$ in this simplex,
\begin{equation}\label{eq:weight-form-continuity}
\sup_{\substack{u\in D^{1,p}(\mathbb R^n)\\
\int_{\mathbb R^n}|\nabla u|^p\,dx=1}}
\int_{\mathbb R^n}|W_\mu-W_{\mu^0}|\,|u|^p\,dx
\longrightarrow0
\qquad(\mu\to\mu^0).
\end{equation}
\end{lemma}

\begin{proof}
Restrict $\mu$ to a compact neighborhood of $\mu^0$ in the
open simplex. The pole coefficients
$A_\alpha(\mu)=g(s\mu_\alpha)$ are continuous.
The expansions underlying
Propositions~\ref{prop:polecompare} and~\ref{prop:inftycompare}
give, uniformly in this parameter neighborhood,
\[
\begin{aligned}
W_\mu(x)&=A_\alpha(\mu)|x-a_\alpha|^{-p}
+O(|x-a_\alpha|^{1-p})&& (x\to a_\alpha),\\
W_\mu(x)&=g(s)|x|^{-p}+O(|x|^{-p-1})&& (|x|\to\infty).
\end{aligned}
\]
These expansions are uniform because the weights stay positive
and the poles are fixed; no barycentric normalization is needed.
For sufficiently small $\delta>0$, the classical Hardy inequality
therefore gives, at each pole,
\begin{align*}
&\int_{B_\delta(a_\alpha)}|W_\mu-W_{\mu^0}|\,|u|^p\,dx\\
&\quad\le
\frac{|A_\alpha(\mu)-A_\alpha(\mu^0)|+C\delta}{H_p}
\int_{\mathbb R^n}|\nabla u|^p\,dx.
\end{align*}
The leading far-field coefficient is independent of $\mu$.
Hence, for large $R$,
\[
\int_{|x|>R}|W_\mu-W_{\mu^0}|\,|u|^p\,dx
\le\frac{C}{H_pR}
\int_{\mathbb R^n}|\nabla u|^p\,dx.
\]
It remains to control the compact set
$K=\overline{B_R(0)}\setminus\bigcup_{\alpha=1}^m
B_\delta(a_\alpha)$, including the moving zero set of the vector
field. To display its dependence on the weights, write
\[
v_\mu(x)=\sum_\alpha\mu_\alpha
\frac{x-a_\alpha}{|x-a_\alpha|^2},
\qquad
S_\mu(x)=\sum_\alpha\frac{\mu_\alpha}{|x-a_\alpha|^2}.
\]
At every $x_0\in K$ with $v_{\mu^0}(x_0)=0$,
\[
\operatorname{tr}Dv_{\mu^0}(x_0)
=(n-2)S_{\mu^0}(x_0)>0.
\]
Thus some component $f_{\mu^0}=(v_{\mu^0})_i$ satisfies
$\partial_i f_{\mu^0}(x_0)>0$. Choose a fixed rectangular
box $Q=Q'\times I$ around $x_0$, with closure avoiding the
poles, and use $t$ for its $i$th coordinate. After shrinking
the box and the parameter neighborhood, we have
\[
\partial_t f_\mu\ge c>0\quad\text{on }Q,
\qquad f_\mu:=(v_\mu)_i,
\]
uniformly in $\mu$. Moreover, \eqref{eq:W-compact-form}
and $p<2$ give, almost everywhere on $Q$,
\[
0\le W_\mu\le C|v_\mu|^{p-2}\le C|f_\mu|^{p-2}.
\]
For almost every $y\in Q'$, the change of variable
$q=f_\mu(y,t)$ is monotone. Consequently,
\begin{align*}
&\int_{I\cap\{|f_\mu(y,t)|<\eta\}}
|f_\mu(y,t)|^{p-2}|u(y,t)|^p\,dt\\
&\quad\le c^{-1}\sup_{t\in I}|u(y,t)|^p
\int_{-\eta}^{\eta}|q|^{p-2}\,dq
\le C\eta^{p-1}\sup_{t\in I}|u(y,t)|^p.
\end{align*}
The one-dimensional Sobolev inequality on $I$, followed by
integration over $Q'$, yields the uniform thin-layer bound
\begin{equation}\label{eq:weight-uniform-layer}
\int_{Q\cap\{|f_\mu|<\eta\}}W_\mu|u|^p\,dx
\le C\eta^{p-1}\|u\|_{W^{1,p}(Q)}^p.
\end{equation}

To compare both weights on the same set, fix
$E_\eta=Q\cap\{|f_{\mu^0}|<\eta\}$.
For $\mu$ sufficiently close to $\mu^0$,
$\|f_\mu-f_{\mu^0}\|_{L^\infty(Q)}<\eta/2$, so
$E_\eta\subset\{|f_\mu|<3\eta/2\}$ and
\[
\int_{E_\eta}(W_\mu+W_{\mu^0})|u|^p\,dx
\le C\eta^{p-1}\|u\|_{W^{1,p}(Q)}^p.
\]
On $Q\setminus E_\eta$, the reverse triangle inequality gives
\[
|v_\mu|\ge|f_\mu|
\ge|f_{\mu^0}|-|f_\mu-f_{\mu^0}|\ge\eta/2.
\]
For each fixed $\eta>0$, this lower bound keeps the factor
$|v_\mu|^{p-2}$ away from its singularity. Since $v_\mu$ and
$S_\mu$ converge uniformly on $Q$, \eqref{eq:W-compact-form}
therefore yields $W_\mu\to W_{\mu^0}$ uniformly on
$Q\setminus E_\eta$.
Cover $K\cap Z(v_{\mu^0})$ by finitely many smaller boxes
whose closures lie in boxes of the preceding type. The remaining
compact subset of $K$ avoids $Z(v_{\mu^0})$, and the weights
converge uniformly on that set as well.
Global Sobolev controls each $\|u\|_{W^{1,p}(Q)}^p$ by the
Dirichlet energy. The finite covering and
\eqref{eq:weight-uniform-layer}, first with $\mu\to\mu^0$
and then $\eta\downarrow0$, therefore give uniform form
convergence on $K$, without any nondegeneracy assumption on
the zeros. Combining this with the pole and tail estimates,
then sending $\delta\downarrow0$ and $R\to\infty$, proves
\eqref{eq:weight-form-continuity}.
Finally,
\[
\frac1{C_*(p,s,\mu)}=
\sup_{\substack{u\in D^{1,p}(\mathbb R^n)\\
\int_{\mathbb R^n}|\nabla u|^p\,dx=1}}
\int_{\mathbb R^n}W_\mu|u|^p\,dx.
\]
The form estimate proves continuity of this positive finite
quantity, and hence of $C_*(p,s,\mu)$.
\end{proof}

\begin{proof}[Proof of Theorem~\ref{thm:multipolar-three-regimes}]
Fix an index $i$. Relabeling the poles, we may suppose $i=1$;
this changes no geometric assumption. Write
$\rho_j:=\rho_j^{(1)}$ for $j\ge2$, and abbreviate
$\mu(t)=\mu^{(1)}(t)$, $C(t)=C_1(t)$, and $Q_t=Q_{1,t}$.
At $t_0=h/s$, all pole exponents are at most $h$:
\[
s\mu_1(t_0)=h,\qquad
s\mu_j(t_0)=(s-h)\rho_j\le h\quad(j\ge2).
\]
The last inequality is the assumed ratio bound. Thus $\mu(t_0)$
belongs to the full critical set, and
Theorem~\ref{thm:subquad-sharp} gives $C(t_0)<1$.
All components of $\mu(t_0)$ are positive, so
Lemma~\ref{lem:weight-continuity} applies at this parameter
and along the path. It follows that $C(t)<1$ for
$t\in(t_0,t_0+\eta_-)$ with some $\eta_-\in(0,1-t_0)$.
Remark~\ref{cor:variational-spectral-threshold} then provides
a test function in $C_c^\infty(\mathbb R^n)$ with $Q_t[u]<0$
at every such parameter.

At the other end, apply Proposition~\ref{prop:multipolar-one-pole}
with $\varepsilon=1-t$, $\mu^\varepsilon=\mu(t)$, and the fixed
proportions $\rho_j$. This yields
\[
\lim_{t\uparrow1}C(t)=
\begin{cases}
H_p/g(s)>1,&s<s_{\max},\\
+\infty,&s=s_{\max}.
\end{cases}
\]
Hence $C(t)>1$ for $t\in(1-\eta_+,1)$ with some
$\eta_+>0$. Estimate \eqref{eq:multipolar-coercivity} and its
Sobolev consequence give coercivity and subcriticality.

For every $t\in(t_0,1)$,
\[
s\mu_1(t)=st>h,\qquad
0<s\mu_j(t)=s(1-t)\rho_j
<(s-h)\rho_j\le h\quad(j\ge2).
\]
Thus the full operator is subcritical by Theorem~\ref{thmopt}.
Every pole exponent lies in $(0,s_{\max})$ and differs from
$h$, while $s\in(h,s_{\max}]$. Consequently,
\[
g(s\mu_\alpha(t))<H_p\quad(1\le\alpha\le m),
\qquad g(s)<H_p,
\]
and hence $\Lambda(p,s,\mu(t))>1$. This also covers
$s=s_{\max}$, when the far-field quotient is $+\infty$.

By continuity, there exists $t_c\in(t_0,1)$ with $C(t_c)=1$.
Every such parameter satisfies
\[
C(t_c)=1<\Lambda(p,s,\mu(t_c)).
\]
Theorem~\ref{prop:compactness-threshold} supplies a positive
normalized minimizer, its weak equation, relative compactness
of normalized minimizing sequences in $D^{1,p}(\mathbb R^n)$,
and criticality of the coefficient-one operator.

Since $s\le mh$, the equal-ratio choice
$\rho_\alpha^{(i)}=1/(m-1)$ is admissible for every $i$,
including at $s=mh$. The argument applies separately to each
index and any admissible ratios; neither the transition parameters
nor the endpoint neighborhoods need agree for different choices.
Along the $i$th path with $t>t_0$, the
unique exponent exceeding $h$ is $s\mu_i^{(i)}(t)$. Thus critical
weights chosen from distinct paths cannot coincide, proving the
last assertion.
\end{proof}

The theorem gives coexistence, not monotonicity or uniqueness of
the crossing, nor distinct connected components for the selected
paths. The assumption $s\le mh$ supplies the nonempty full critical
set; it is not claimed necessary for coexistence. It covers every
admissible $s>h$ when $m\ge p/(p-1)$, but the case $mh<s\le s_{\max}$
and general weights with several strict poles remain open.

\subsection{Transition through the quadratic exponent}
Write $W_p:=W_{m,s,p}^{\mu}$, keeping $s$, $\mu$, and the poles fixed.

\begin{proof}[Proof of Theorem~\ref{thm:attainment-near-two}]
All inequalities in \eqref{eq:quadratic-strict-assumptions} are
strict, so the asserted admissibility conditions persist for
$p$ sufficiently close to $2$ on either side.
At $p=2$, the remainder vanishes and $W_2=V_{m,s,2}^{\mu}$.
By Theorem~\ref{thmopt}, the ground state
\[
\phi(x)=\prod_{\alpha=1}^m|x-a_\alpha|^{-s\mu_\alpha}
\]
belongs to $D^{1,2}(\mathbb R^n)$ and attains the full-potential
inequality. Hence $C_*(2,s,\mu)=1$.

Let $\delta>0$. We may choose a fixed
$f\in C_c^\infty(\Omega)$, with
$\Omega=\mathbb R^n\setminus\{a_1,\ldots,a_m\}$, such that
\[
F_f(2):=
\frac{\int_{\mathbb R^n}|\nabla f|^2\,dx}
{\int_{\mathbb R^n}W_2|f|^2\,dx}<1+\delta.
\]
To justify this choice, finite sets have zero $2$-capacity in
$\mathbb R^n$ for $n\ge3$, so $C_c^\infty(\Omega)$ is dense in
$D^{1,2}(\mathbb R^n)$. The Hardy inequality with weight $W_2$
also gives convergence in $L^2(W_2\,dx)$ for such approximations
of $\phi$.

Set $K:=\operatorname{supp}f\Subset\Omega$. On $K$, all pole
factors in the formula for $W_p$ are bounded. The vector field
$v$ is real analytic on the connected set $\Omega$ and
is not identically zero by its far-field expansion
\eqref{inftyexp-v}; hence $Z(v)$ has measure zero.
Consequently,
\[
|v(x)|^{p-2}\longrightarrow1
\quad\text{for a.e. }x\in K\quad\text{as }p\downarrow2.
\]
For any fixed sufficiently small $\eta>0$ and $2<p\le2+\eta$,
\[
|v(x)|^{p-2}
\le\max\{1,\|v\|_{L^\infty(K)}^\eta\}
\quad (x\in K).
\]
The remaining coefficients in $W_p$ depend continuously on $p$.
Dominated convergence therefore yields
\[
F_f(p):=
\frac{\int_{\mathbb R^n}|\nabla f|^p\,dx}
{\int_{\mathbb R^n}W_p|f|^p\,dx}
\longrightarrow F_f(2)\qquad(p\downarrow2),
\]
with a positive denominator for $p$ sufficiently close to $2$.
Since $R\ge0$ and $V=W_p+R$ for admissible $p\ge2$, the full
Hardy inequality gives $C_*(p,s,\mu)\ge1$. Thus
\[
1\le\liminf_{p\downarrow2}C_*(p,s,\mu)
\le\limsup_{p\downarrow2}C_*(p,s,\mu)
\le F_f(2)<1+\delta.
\]
Letting $\delta\downarrow0$ proves the right-hand limit.
On the other hand, for $p<2$ sufficiently close to $2$,
Theorem~\ref{thm:subquad-sharp} gives
\[
\frac{n+p-4}{n-p}<C_*(p,s,\mu)<1.
\]
The explicit lower bound tends to $1$ as $p\uparrow2$.
This proves the left-hand limit and hence
\eqref{eq:Cstar-limit-two}.

By \eqref{eq:quadratic-strict-assumptions},
$0<A_\infty(2,s)<H_2$ and
$0<A_\alpha(2,s,\mu)<H_2$ at every active pole.
All denominators in \eqref{def:concentration-threshold} are
therefore positive near $p=2$, and continuity gives
\[
\Lambda(p,s,\mu)\longrightarrow\Lambda(2,s,\mu)>1.
\]
Together with \eqref{eq:Cstar-limit-two}, this implies
$C_*(p,s,\mu)<\Lambda(p,s,\mu)$ for $p$ sufficiently close
to $2$ on either side. The comparison with $1$ follows from
Theorems~\ref{thm:subquad-sharp} and~\ref{tmHSW3}.
The positive eigenfunction and criticality assertions follow
directly from Theorem~\ref{prop:compactness-threshold}.
\end{proof}

For example, any distinct pair of poles with
$\mu_1=\mu_2=1/2$ and $s=3(n-2)/4$ satisfies
\eqref{eq:quadratic-strict-assumptions}, with
$\Lambda(2,s,\mu)=\min\{4/3,16/15\}=16/15$.
We only prove that the result holds for $p$ sufficiently close to $2$;
we do not give an explicit size of this neighborhood, nor do we prove
that the corresponding minimizers vary continuously with $p$.

\subsection{Superquadratic concentration and compactness phases}
\label{subsec:superquadratic-phases}

The improvement $C_*>1$ need not give a concentration gap.
An auxiliary variational constant locates an infinity-concentration
interval inside the full-potential attainment region and,
under strict separation, the compact phase following it.

Unless otherwise specified, throughout the remainder of this
section we fix $2<p<n$, $\mu\in\Psi$, and the pole configuration.
We abbreviate
\[
\begin{gathered}
H:=h^p,\qquad W_s:=W_{m,s,p}^{\mu},\qquad
R_s:=V_{m,s,p}^{\mu}-W_s,\\
\phi_s:=\prod_\alpha|x-a_\alpha|^{-s\mu_\alpha}.
\end{gathered}
\]
Recall the interaction weight
$G=|v|^{p-2}(S-|v|^2)$ from \eqref{eq:interaction-kernel}.
The variance identity gives
\[
S-|v|^2=\sum_{\alpha<\beta}
\frac{\mu_\alpha\mu_\beta|a_\alpha-a_\beta|^2}
{|x-a_\alpha|^2|x-a_\beta|^2}>0
\]
away from the active poles. Since $v$ is real analytic and not
identically zero, $G>0$ a.e. The truncated potential decomposes as
\begin{equation}\label{eq:superquad-two-kernels}
W_s=g(s)|v|^p+(n-p)s^{p-1}G,\qquad
W_h=H|v|^p+pHG.
\end{equation}
All unqualified integrals below are over $\mathbb R^n$.
We use the auxiliary form $Q_h$ and constant $\kappa_*$
defined in \eqref{eq:auxiliary-kappa}.

\begin{proposition}[An anchored interaction improvement]
\label{prop:positive-kappa}
The auxiliary constant satisfies
\begin{equation}\label{eq:kappa-gs-bounds}
0<\kappa_*<\kappa_{\mathrm{gs}}
:=\frac{\int R_h\phi_h^p\,dx}{\int G\phi_h^p\,dx}
<2(p-2)h^{p-1}.
\end{equation}
Both integrals defining $\kappa_{\mathrm{gs}}$ are finite and
strictly positive.
\end{proposition}

\begin{proof}
Write $\phi=\phi_h$. For $u\in C_c^\infty(\Omega)$, set
$w=u/\phi$. The ground-state representation and the
superquadratic $C_p$ estimate imply
\begin{equation}\label{eq:kappa-remainder-control}
Q_h[u]\ge c_p\int\phi^p|\nabla w|^p\,dx
+\int R_h\phi^p|w|^p\,dx,
\qquad c_p>0.
\end{equation}
We first prove the anchored weighted estimate
\begin{equation}\label{eq:interaction-anchored}
\int G\phi^p|w|^p\,dx
\le C\left(\int\phi^p|\nabla w|^p\,dx
+\int R_h\phi^p|w|^p\,dx\right).
\end{equation}

Near an active pole, let $r=|x-a_\alpha|$ and
$\beta_\alpha=h\mu_\alpha<h$. Then
$\phi^p\asymp r^{-p\beta_\alpha}$ and $G\le Cr^{-p}$.
For a function $f$ compactly supported in a sufficiently small
punctured ball, the weighted Hardy inequality gives
\[
\int r^{-p\beta_\alpha-p}|f|^p\,dx
\le\left(\frac{p}{n-p-p\beta_\alpha}\right)^p
\int r^{-p\beta_\alpha}|\nabla f|^p\,dx.
\]
This follows by divergence and H\"older, since
$n-p-p\beta_\alpha=(n-p)(1-\mu_\alpha)>0$.
Apply this inequality to $f=\chi w$, where $\chi$ is supported
in that ball and equals one near the pole. The term
$w\nabla\chi$ is supported on a fixed annulus, where $\phi$
is bounded, and contributes only an $L^p$ error there.

At infinity, with $r=|x|$, we have
$\phi^p\asymp r^{-(n-p)}$ and $G=O(r^{-p-2})$.
Thus the polar-coordinate densities of $G\phi^p\,dx$ and
$\phi^p\,dx$ are bounded by $Cr^{-3}$ and comparable to
$r^{p-1}$, respectively. Choose a radial cut-off $\eta$ which
vanishes for $r\le R$ and equals one for $r\ge2R$.
Along each ray, $f=\eta w$ satisfies $f(R)=0$, and H\"older gives
\[
|f(r)|^p\le\bigl(\log(r/R)\bigr)^{p-1}
\int_R^r t^{p-1}|f'(t)|^p\,dt.
\]
Tonelli's theorem then yields
\begin{align*}
&\int_R^\infty r^{-3}|f(r)|^p\,dr\\
&\quad\le\int_R^\infty t^{p-1}|f'(t)|^p
\left[\int_t^\infty r^{-3}\bigl(\log(r/R)\bigr)^{p-1}\,dr\right]dt\\
&\quad\le\frac{\Gamma(p)}{2^pR^2}
\int_R^\infty t^{p-1}|f'(t)|^p\,dt.
\end{align*}
Indeed, for $t\ge R$, the substitution $y=\log(r/R)$ gives
\[
\int_t^\infty r^{-3}\bigl(\log(r/R)\bigr)^{p-1}\,dr
\le R^{-2}\int_0^\infty e^{-2y}y^{p-1}\,dy
=\frac{\Gamma(p)}{2^pR^2}.
\]
Integrating over the sphere and using
$f'=\eta\,\partial_rw+\eta'w$ and $|\partial_rw|\le|\nabla w|$
therefore controls the tail by $\int\phi^p|\nabla w|^p$ and
an $L^p$ error on the fixed annulus $R<|x|<2R$.

Place the remaining region and annuli in a connected Lipschitz
domain $U\Subset\Omega_\mu$, where $\phi\asymp1$ and $G$ is bounded.
Choose a ball $B\Subset U$ away from $Z(v)$ and the lines joining
active poles, so that $\inf_B R_h>0$. Hence
\[
\int_B|w|^p\,dx\le C\int_B R_h\phi^p|w|^p\,dx.
\]
Together with $\phi\asymp1$ on $U$, anchored Poincar\'e gives
\begin{align*}
\int_U|w|^p\,dx
&\le C\left(\int_U|\nabla w|^p\,dx+\int_B|w|^p\,dx\right)\\
&\le C\left(\int_U\phi^p|\nabla w|^p\,dx
+\int_B R_h\phi^p|w|^p\,dx\right).
\end{align*}
This controls both the remaining weighted integral and all
fixed-annulus errors. Combining the three regions proves
\eqref{eq:interaction-anchored}, hence $\kappa_*>0$ by
\eqref{eq:kappa-remainder-control}.
Density extends the estimate to $D^{1,p}$: finite sets have zero
$p$-capacity, and $G,W_h$ are form-bounded. A fixed test gives
$\kappa_*<\infty$.

To prove the upper bounds, write
$M=\sum_\alpha\mu_\alpha(v\cdot z_\alpha)^2$.
Jensen's inequality gives $M\ge|v|^4$, and hence
\[
0\le R_h
=2(p-2)h^{p-1}|v|^{p-4}(|v|^2S-M)
\le2(p-2)h^{p-1}G.
\]
The last inequality is strict on a set of positive measure.
Indeed, near an active pole with weight $\mu_\alpha$,
$M-|v|^4\sim\mu_\alpha^3(1-\mu_\alpha)r^{-4}>0$.
Moreover,
\[
G\phi^p\,dx\lesssim
r^{(n-p)(1-\mu_\alpha)-1}\,dr\,d\theta
\]
near that pole, whereas its radial upper bound at infinity is
$Cr^{-3}\,dr\,d\theta$. Thus $\int G\phi^p$ and
$\int R_h\phi^p$ are finite and positive, and
$\kappa_{\mathrm{gs}}<2(p-2)h^{p-1}$.

Let $(\chi_\varepsilon)$ be the smooth null cutoffs constructed in
Section~\ref{sec:groundstate}, with $0\le\chi_\varepsilon\le1$ and
$\chi_\varepsilon\to1$ locally on $\Omega$ and almost everywhere.
Set $u_\varepsilon=\phi_h\chi_\varepsilon\in C_c^\infty(\Omega)$.
The ground-state representation for the full potential
$V_h=W_h+R_h$ gives
\[
Q_h[u_\varepsilon]
=\int_{\mathbb R^n}
C_p(\nabla u_\varepsilon,\phi_h\nabla\chi_\varepsilon)\,dx
+\int_{\mathbb R^n}R_h\phi_h^p\chi_\varepsilon^p\,dx.
\]
The first integral tends to zero by the null-cutoff estimate in
Section~\ref{sec:groundstate}. For the two potential integrals,
dominated convergence applies since $R_h\phi_h^p,G\phi_h^p\in L^1$
and $0\le\chi_\varepsilon\le1$. Moreover, as proved above,
\[
0<\int_{\mathbb R^n}G\phi_h^p\,dx<\infty.
\]
Thus the denominator below is positive for small $\varepsilon$, and
\[
\frac{Q_h[u_\varepsilon]}
{\int_{\mathbb R^n}G|u_\varepsilon|^p\,dx}
\xrightarrow{\varepsilon\downarrow0}
\frac{\int_{\mathbb R^n}R_h\phi_h^p\,dx}
{\int_{\mathbb R^n}G\phi_h^p\,dx}
=\kappa_{\mathrm{gs}}.
\]
The variational definition \eqref{eq:auxiliary-kappa} therefore gives
$\kappa_*\le\kappa_{\mathrm{gs}}$.

To make this inequality strict, note that $R_h/G=O(r^2)$ near
every active pole. This follows because
$|v|^2S-M=\sum_\alpha\mu_\alpha|v\wedge z_\alpha|^2=O(r^{-2})$,
whereas $|v|^2(S-|v|^2)\asymp r^{-4}$.
Choose a nonzero $\psi\in C_c^\infty(\Omega)$, $\psi\ge0$,
supported in a small annulus where
$R_h<\kappa_{\mathrm{gs}}G$.
Set $K=\operatorname{supp}\psi$ and
$u_{\varepsilon,t}=\phi\chi_\varepsilon(1+t\psi)$, with $|t|$
small enough that $1+t\psi>0$. For sufficiently small
$\varepsilon$, $\chi_\varepsilon=1$ in a neighborhood of $K$.
The full ground-state remainder for $u_{\varepsilon,t}$ therefore
decomposes as $\mathcal R_{\varepsilon,t}=J(t)+E_\varepsilon$, where
\[
\begin{aligned}
J(t)&:=\int_K
C_p\bigl(\nabla[\phi(1+t\psi)],t\phi\nabla\psi\bigr)\,dx,\\
E_\varepsilon&:=\int_{\mathbb R^n\setminus K}
C_p\bigl(\nabla(\phi\chi_\varepsilon),
\phi\nabla\chi_\varepsilon\bigr)\,dx.
\end{aligned}
\]
Since $p>2$ and the coefficients are bounded on $K$, the Bregman
remainder estimate gives $J(t)=O(t^2)$ as $t\to0$.
The cutoff remainder $E_\varepsilon$, which is independent of $t$,
tends to zero by the null-cutoff estimate in
Section~\ref{sec:groundstate}. Consequently,
\[
\mathcal R_{\varepsilon,t}
=J(t)+E_\varepsilon
=O(t^2)+o_{\varepsilon\downarrow0}(1).
\]
For each fixed sufficiently small $t$, dominated convergence
applies to the two weighted potential integrals, so the auxiliary
quotients of $u_{\varepsilon,t}$ converge as $\varepsilon\downarrow0$ to
\[
F(t)=\frac{J(t)+\int R_h\phi^p(1+t\psi)^p\,dx}
{\int G\phi^p(1+t\psi)^p\,dx}.
\]
Therefore
$\kappa_*\le F(t)$, $F(0)=\kappa_{\mathrm{gs}}$, and
\[
F'(0)=\frac{p\int(R_h-\kappa_{\mathrm{gs}}G)\phi^p\psi\,dx}
{\int G\phi^p\,dx}<0.
\]
Taking $t>0$ sufficiently small proves the strict bound in
\eqref{eq:kappa-gs-bounds}.
\end{proof}

\begin{proof}[Proof of Theorem~\ref{thm:superquadratic-phases}]
For $h\le s<s_{\max}$, let
\begin{equation}\label{eq:superquad-b}
b(s):=H\left(\frac{n-p}{n-p-(p-1)s}-p\right).
\end{equation}
The decomposition \eqref{eq:superquad-two-kernels} gives the
exact identity
\begin{equation}\label{eq:superquad-exact-identity}
\int|\nabla u|^p\,dx-\Lambda_\infty(s)\int W_s|u|^p\,dx
=Q_h[u]-b(s)\int G|u|^p\,dx.
\end{equation}
The function $b$ increases strictly from zero to infinity,
and $b(s_c)=\kappa_*$. By the definition of $\kappa_*$ and
the concentration upper bound at infinity,
\begin{equation}\label{eq:infinity-branch-characterization}
\begin{aligned}
C_*(p,s,\mu)&=\Lambda_\infty(s)&& (h\le s\le s_c),\\
C_*(p,s,\mu)&<\Lambda_\infty(s)&& (s_c<s<s_{\max}).
\end{aligned}
\end{equation}
The strict inequality follows by testing an auxiliary quotient
below $b(s)$. To compare with the pole thresholds, direct algebra
gives $g(s)=g(sq)$ exactly at $s=s_\infty$ and
$h<s_\infty<\min\{s_{\max},h/q\}$: use
$1-pq^{p-1}+(p-1)q^p>0$ and $q^p-pq+p-1>0$ for $0<q<1$.
Since $g$ increases on $[0,h]$, for $h\le s\le s_\infty$,
\[
g(s)\ge g(sq)\ge g(s\mu_\alpha)
\quad\text{for every }\alpha,
\qquad \Lambda(p,s,\mu)=\Lambda_\infty(s).
\]
Finally, near an active pole,
\[
G(x)=\mu_\alpha^{p-1}(1-\mu_\alpha)
|x-a_\alpha|^{-p}(1+O(|x-a_\alpha|)).
\]
The concentrating Hardy tests used in
Lemma~\ref{lem:concentration-upper}, placed at a pole of weight
$q$, therefore give
\begin{equation}\label{eq:kappa-pole-bound}
\kappa_*\le\frac{H-g(hq)}{q^{p-1}(1-q)}
=H\frac{1-pq^{p-1}+(p-1)q^p}{q^{p-1}(1-q)}
=b(s_\infty).
\end{equation}
Together with $\kappa_*>0$, this proves
\eqref{eq:transition-parameter-order} and the value asserted
in (i). Strict inequality over one follows from $g(s)<H$ for
$s>h$.

For $h\le s<s_c$, \eqref{eq:superquad-exact-identity} gives
\begin{equation}\label{eq:superquad-optimal-remainder}
\int|\nabla u|^p\,dx-C_*(p,s,\mu)\int W_s|u|^p\,dx
\ge(\kappa_*-b(s))\int G|u|^p\,dx.
\end{equation}
This proves subcriticality and nonattainment, and gives
$\int G|u_k|^p\to0$ for normalized minimizing sequences.
Moreover, $W_s\le C_R G$ on each fixed ball because
\[
\frac{W_s}{G}=(n-p)s^{p-1}
+g(s)\frac{|v|^2}{S-|v|^2},
\]
and the last quotient is locally bounded: its limit at an
active pole is $\mu_\alpha/(1-\mu_\alpha)$, and its denominator
is strictly positive elsewhere, including at zeros of $v$.
This proves \eqref{eq:superquad-mass-escape} and (ii).

At $s=s_c$, the equality $C_*=\Lambda$ already proved in (i)
and Theorem~\ref{prop:compactness-threshold} exclude finite-energy
attainment. The exact identity
\eqref{eq:superquad-exact-identity} shows that attainment of
$C_*(p,s_c,\mu)$ is equivalent to attainment of $\kappa_*$.
Thus the auxiliary infimum is also not attained, completing (i).

If $s_c<s_\infty$, \eqref{eq:infinity-branch-characterization}
gives $C_*<\Lambda$ on $(s_c,s_\infty]$. To extend the gap,
\eqref{eq:superquad-two-kernels} and Hardy at $h$ give
\[
\int|v|^p|u|^p\,dx\le\frac1H\int|\nabla u|^p\,dx,
\qquad
\int G|u|^p\,dx\le\frac1{pH}\int|\nabla u|^p\,dx.
\]
Taking suprema over functions of unit energy yields
\begin{equation}\label{eq:superquad-s-continuity}
\left|\frac1{C_*(p,t,\mu)}-\frac1{C_*(p,s,\mu)}\right|
\le\frac{|g(t)-g(s)|}{H}
+\frac{n-p}{pH}|t^{p-1}-s^{p-1}|.
\end{equation}
Thus $C_*$ is continuous on the full interval $[h,s_{\max}]$.
So is $\Lambda$, since
\[
\Lambda(p,s,\mu)=
\frac{H}{\max\{g(s),\max_\alpha g(s\mu_\alpha)\}},
\]
with a positive denominator. Choose $\delta>0$ sufficiently
small that the gap persists and
$s_\infty+\delta<\min\{s_{\max},h/q\}$.
Theorem~\ref{tmHSW3} gives $C_*>1$ on this interval, and
Theorem~\ref{prop:compactness-threshold} gives all the
compactness and criticality assertions in (iii).

By \eqref{eq:kappa-gs-bounds} and \eqref{eq:kappa-pole-bound},
condition \eqref{eq:explicit-two-phase-condition} implies
$\kappa_*<b(s_\infty)$, hence $s_c<s_\infty$.
As $p\downarrow2$ with $n$ and $\mu$ fixed, its left-hand
side tends to zero and its right-hand side tends to
$(1-q)/q>0$. This proves the final assertion.
\end{proof}

At a separated endpoint, a compactly anchored comparison proves
criticality despite nonattainment.

\begin{proposition}[Criticality at a separated endpoint]
\label{prop:separated-endpoint-criticality}
Retain the assumptions and notation of
Theorem~\ref{thm:superquadratic-phases}, and set
\[
K_{\rm pole}:=\frac{H-g(hq)}{q^{p-1}(1-q)},
\qquad q=\max_\alpha\mu_\alpha.
\]
If $\kappa_*<K_{\rm pole}$, equivalently $s_c<s_\infty$, then the form
\[
Q_c[u]:=\int|\nabla u|^p\,dx-\int Z_c|u|^p\,dx,
\qquad
Z_c:=W_h+\kappa_*G=\Lambda_\infty(s_c)W_{s_c},
\]
is critical. Its positive ground state $\Phi_c$ on
$\Omega_\mu=\mathbb R^n\setminus\{a_\alpha:\mu_\alpha>0\}$ is unique
up to multiplication by a positive constant, and
\begin{equation}\label{eq:endpoint-groundstate-integrability}
0<\int G\Phi_c^p\,dx<\infty,
\qquad
\int W_{s_c}\Phi_c^p\,dx=\infty,
\qquad
\Phi_c\notin D^{1,p}(\mathbb R^n).
\end{equation}
Thus the endpoint is positive-critical relative to $G$ and null-critical
relative to $W_{s_c}$, where these terms refer to the finiteness or
divergence of the corresponding weighted ground-state integral.
\end{proposition}

\begin{proof}
The equivalence of the two strict inequalities follows from
\eqref{eq:superquad-b} and \eqref{eq:kappa-pole-bound}.
At every active pole, Proposition~\ref{prop:polecompare} gives
\[
Z_c(x)=A_\alpha r^{-p}+O(r^{1-p}),\qquad
A_\alpha=\frac{H g(s_c\mu_\alpha)}{g(s_c)}<H,
\qquad r=|x-a_\alpha|.
\]
Indeed, $s_c<s_\infty$ implies
$g(s_c)>g(s_cq)\ge g(s_c\mu_\alpha)$.
Also
\[
G(x)=d_\alpha r^{-p}(1+O(r)),\qquad
d_\alpha=\mu_\alpha^{p-1}(1-\mu_\alpha)>0.
\]
Translate the weighted barycenter to the origin. At infinity,
\[
Z_c(x)=H|x|^{-p}+O(|x|^{-p-2}),
\qquad G(x)=O(|x|^{-p-2}).
\]

Choose a positive smooth function $\psi$ on $\Omega_\mu$ that equals
$|x-a_\alpha|^{-h}$ near each active pole and
$r^{-h}(1-r^{-1})$, $r=|x|$, for all sufficiently large $r$.
Join these prescriptions by positive smooth interpolation.
Near a pole,
\[
\frac{-\Delta_p\psi}{\psi^{p-1}}-Z_c
=(H-A_\alpha)r^{-p}+O(r^{1-p}),
\]
which dominates a fixed positive multiple of $G$.
At infinity, direct radial differentiation gives
\[
\frac{-\Delta_p\psi}{\psi^{p-1}}
=Hr^{-p}+(p-1)h^{p-2}r^{-p-1}+O(r^{-p-2}),
\]
so the same domination holds there. There are consequently
$\delta>0$, a compact set $K\Subset\Omega_\mu$, and $M\ge0$ such that
\[
\frac{-\Delta_p\psi}{\psi^{p-1}}-Z_c
\ge\delta G-M\mathbf1_K
\qquad\text{on }\Omega_\mu.
\]
The compact error is bounded because $\inf_K\psi>0$ and,
as $p>2$, $\Delta_p\psi$ remains bounded even at its gradient zeros.
Picone's inequality yields
\begin{equation}\label{eq:endpoint-anchored-comparison}
\delta\int G|u|^p\,dx
\le Q_c[u]+M\int_K|u|^p\,dx
\qquad(u\in C_c^\infty(\Omega_\mu)).
\end{equation}

Suppose that $Q_c$ were subcritical on $\Omega_\mu$. The
ground-state alternative \cite[Theorem~1.6]{PT} then supplies a
strictly positive continuous function $\omega$ such that
\[
Q_c[u]\ge\int\omega|u|^p\,dx.
\]
Its hypotheses hold since $Z_c$ is locally bounded on
$\Omega_\mu$ and $Q_c\ge0$. With $a_K=\min_K\omega>0$,
\eqref{eq:endpoint-anchored-comparison} gives
\[
Q_c[u]\ge\frac{\delta}{1+M/a_K}\int G|u|^p\,dx.
\]
Density extends this improvement to $D^{1,p}$, contradicting
the definition of $\kappa_*$; here the poles have zero capacity
and $W_h,G$ are form-bounded. Thus $Q_c$ is critical on
$\Omega_\mu$, and also on $\mathbb R^n$, since any whole-space
improvement would restrict to the punctured test class.

The same ground-state alternative gives a positive ground state
$\Phi_c$, unique up to scale, and a null sequence converging to it
locally in $L^p(\Omega_\mu)$. Apply
\eqref{eq:endpoint-anchored-comparison} to this sequence and then use
Fatou's lemma. Its right-hand side remains bounded, so
$\int G\Phi_c^p<\infty$; the integral is positive because
$G>0$ a.e. and $\Phi_c>0$.
Finally, Lemma~\ref{lem:critical-ends}(ii) gives
$\Phi_c(x)\ge c|x|^{-h}$ at infinity. Since
$W_{s_c}(x)\sim g(s_c)|x|^{-p}$ and $g(s_c)>0$,
\[
\int_{|x|>R}W_{s_c}\Phi_c^p\,dx
\ge c\int_R^\infty\frac{dr}{r}=\infty.
\]
The same lower bound excludes $\Phi_c\in D^{1,p}$.
\end{proof}

\begin{remark}[The coincident-threshold case]
\label{rem:coincident-endpoint-criticality}
Proposition~\ref{prop:separated-endpoint-criticality} does not settle
criticality when $\kappa_*=K_{\rm pole}$, or equivalently
$s_c=s_\infty$. At a maximal-weight pole the coefficient of $Z_c$
then equals $H$, and the anchored estimate
\eqref{eq:endpoint-anchored-comparison} cannot hold with $\delta>0$.
Indeed, the annular Hardy near-minimizers can be concentrated at that
pole and normalized by $\int G|u_k|^p=1$, while $Q_c[u_k]\to0$ and
their supports eventually avoid every fixed compact subset of
$\Omega_\mu$. If the coincident endpoint is critical, its ground
state instead satisfies $\int G\Phi_c^p=\infty$: apply
Lemma~\ref{lem:critical-ends}(i) at a maximal-weight pole and use
$G\asymp|x-a_\alpha|^{-p}$. Thus positive-criticality relative to
$G$ in \eqref{eq:endpoint-groundstate-integrability} genuinely uses
the strict separation assumption.
\end{remark}

\begin{remark}[An admissible coefficient region]
\label{rem:coefficient-region}
The auxiliary threshold also has a two-parameter description.
For
\[
\mathcal Q_{A,B}[u]:=\int|\nabla u|^p\,dx
-A\int|v|^p|u|^p\,dx-B\int G|u|^p\,dx,
\]
the set
\[
\mathcal K:=\{(A,B)\in\mathbb R^2:
\mathcal Q_{A,B}[u]\ge0\ \text{for all }u\in D^{1,p}\}
\]
is closed and convex, being an intersection of closed
half-planes. Concentration at infinity and at an active pole
gives, respectively,
\[
A\le H,\qquad
\mu_\alpha^p A+\mu_\alpha^{p-1}(1-\mu_\alpha)B\le H.
\]
Its vertical section is determined exactly:
\[
\mathcal K\cap\{A=H\}
=\{(H,B):B\le pH+\kappa_*\}.
\]
Thus $s_c$ is the upper endpoint of the admissible segment
traced by $(H,pH+b(s))$; the whole segment already lies on
the boundary of $\mathcal K$, since $A\le H$.

When a pole of maximal weight $q$ determines $\Lambda$, put
\[
A_q(s)=\frac{Hg(s)}{g(sq)},\qquad
B_q(s)=\frac{H(n-p)s^{p-1}}{g(sq)}.
\]
These coefficients lie on the local Hardy constraint line
\[
q^pA+q^{p-1}(1-q)B=H,
\]
and $C_* = H/g(sq)$ exactly when $(A_q(s),B_q(s))\in\mathcal K$.
Here $A_q$ decreases strictly and $B_q$ increases strictly
with $s$. Convexity therefore makes the equality set an
interval, possibly empty or a single point, on any interval
where this pole determines the threshold. Proposition~\ref{prop:pole-concentration-layer} below identifies
a nontrivial interval in this intersection when the maximal weight
is sufficiently large.
\end{remark}

\subsection{Strict separation and multipolar stability}
\label{subsec:strict-separation}

A critical single-pole test proves $s_c<s_\infty$ for all bipolar
configurations and several multipolar regimes.

\subsubsection{A single-pole test for strict separation}

Put $P:=|v|^p$ and $\omega_{n-1}:=|\mathbb S^{n-1}|$.
For an active pole $a=a_\alpha$ of weight $q=\mu_\alpha\in(0,1)$,
define
\[
c_q(p):=\frac{1-q^p}{q^{p-1}(1-q)},
\qquad
K(q):=H\bigl(c_q(p)-p\bigr).
\]
When $p$ is fixed, we write $c_q=c_q(p)$.
Thus
\[
Z_q:=W_h+K(q)G=H(P+c_qG).
\]
The identity $q^p+c_qq^{p-1}(1-q)=1$ makes the coefficient of
$Z_q$ at the selected pole equal to the single-pole Hardy
constant $H$.

\begin{lemma}[A single-pole integral criterion]\label{lem:pole-integral-test}
Let $2<p<n$ and let $a=a_\alpha$ be any active pole.  The integral
\begin{equation}\label{eq:pole-integral-test}
I_a:=
\int_{\mathbb R^n}
\left(P+c_qG-|x-a|^{-p}\right)|x-a|^{p-n}\,dx
\end{equation}
is absolutely convergent.  If $I_a\ge0$, then
\[
\kappa_*<K(q).
\]
In particular, if $q=\max_\beta\mu_\beta$ and $I_a\ge0$, then
$s_c<s_\infty$.
\end{lemma}

\begin{proof}
Put $r=|x-a|$. The normalization of $c_q$ gives
$P+c_qG=r^{-p}+O(r^{1-p})$ at $a$, and the far-field expansion
gives $P+c_qG=r^{-p}+O(r^{-p-1})$ at infinity. Thus the
integrand is respectively $O(r^{1-n})$ and $O(r^{-n-1})$.
At another pole it is $O(|x-a_\beta|^{-p})$, while zeros of $v$
cause no singularity for $p>2$. This proves absolute convergence.

Let $\psi(x)=|x-a|^{-h}$ and let $\chi_j$ be smooth radial
logarithmic cutoffs satisfying $0\le\chi_j\le1$,
$\chi_j=1$ for $e^{-j}\le r\le e^j$,
$\chi_j=0$ outside $e^{-2j}<r<e^{2j}$, and
$|\nabla\chi_j|\le C/(jr)$.
For the classical Hardy form
\[
\mathcal H_a[u]:=
\int|\nabla u|^p\,dx-H\int r^{-p}|u|^p\,dx,
\]
the radial ground-state estimate yields
\[
0\le\mathcal H_a[\psi\chi_j]
\le C\bigl(j^{1-p}+j^{-1}\bigr)\longrightarrow0.
\]
Indeed, its two cutoff contributions are bounded by
$\int\psi^p|\nabla\chi_j|^p$ and
$\int\psi^2|\nabla\psi|^{p-2}
\chi_j^{p-2}|\nabla\chi_j|^2$, respectively.
The corresponding radial measures are multiples of $dr/r$.
Define
\[
\mathcal E_q[u]:=\int|\nabla u|^p\,dx-\int Z_q|u|^p\,dx
=Q_h[u]-K(q)\int G|u|^p\,dx.
\]
Since $\psi^p=r^{p-n}$, we have the exact decomposition
\[
\begin{aligned}
I_{a,j}&:=\int_{\mathbb R^n}
\left(P+c_qG-r^{-p}\right)r^{p-n}\chi_j^p\,dx,\\
\mathcal E_q[\psi\chi_j]
&=\mathcal H_a[\psi\chi_j]-HI_{a,j}.
\end{aligned}
\]
The Hardy remainder tends to zero by the logarithmic null-cutoff
estimate above. On the other hand, the absolute integrability
proved in \eqref{eq:pole-integral-test}, together with
$0\le\chi_j\le1$ and $\chi_j\to1$ almost everywhere, gives
$I_{a,j}\to I_a$ by dominated convergence. Combining these two
limits yields
\begin{equation}\label{eq:pole-test-limit}
\mathcal E_q[\psi\chi_j]\xrightarrow{j\to\infty}-HI_a.
\end{equation}
If $I_a>0$, then $H>0$ implies that, for all sufficiently large $j$,
\[
\mathcal E_q[\psi\chi_j]<-\frac{HI_a}{2}<0.
\]
The functions $u_j:=\psi\chi_j$ are smooth and vanish near $a$
and outside a large ball, so
$u_j\in C_c^\infty(\mathbb R^n)\subset D^{1,p}(\mathbb R^n)$.
They are therefore admissible negative-energy tests.

It remains to treat $I_a=0$. By \eqref{eq:pole-test-limit},
$\mathcal E_q[\psi\chi_j]\to-HI_a=0$.
There is another active pole $a_\beta\ne a$, near which $Z_q$
has a strictly positive $|x-a_\beta|^{-p}$ leading term,
whereas $Hr^{-p}$ remains bounded. Choose a nonempty open set
$U_\beta\Subset\Omega_\mu$ near $a_\beta$ on which $Z_q>Hr^{-p}$,
and $0\le\eta\in C_c^\infty(U_\beta)$, $\eta\not\equiv0$.
For $|t|$ small enough that $1+t\eta>0$, define the
multiplicatively perturbed Hardy ground state and the test function by
\[
\psi_t(x):=\psi(x)\bigl[1+t\eta(x)\bigr],
\qquad
u_{j,t}(x):=\chi_j(x)\psi_t(x).
\]
For all sufficiently large $j$, $\chi_j=1$ near
$\operatorname{supp}\eta$, so
$\mathcal E_q[u_{j,t}]-\mathcal E_q[\psi\chi_j]=D(t)$, where
\[
\begin{aligned}
D(t)&:=\int_{U_\beta}
\bigl(|\nabla\psi_t|^p-|\nabla\psi|^p\bigr)\,dx\\
&\quad-\int_{U_\beta}Z_q\bigl(\psi_t^p-\psi^p\bigr)\,dx.
\end{aligned}
\]
Both integrands are compactly supported in $U_\beta$.
Thus $D(t)$ is independent of $j$, and
$E(t):=\lim_{j\to\infty}\mathcal E_q[u_{j,t}]=D(t)$, with $E(0)=0$.
Since $\left.\partial_t\psi_t\right|_{t=0}=\psi\eta$,
differentiation and the identity $-\Delta_p\psi=Hr^{-p}\psi^{p-1}$,
tested against the compactly supported function $\psi\eta$, give
\[
\begin{aligned}
E'(0)
&=p\int_{U_\beta}|\nabla\psi|^{p-2}\nabla\psi
\cdot\nabla(\psi\eta)\,dx
-p\int_{U_\beta}Z_q\psi^p\eta\,dx\\
&=p\int_{U_\beta}(Hr^{-p}-Z_q)\psi^p\eta\,dx<0.
\end{aligned}
\]
Hence some fixed $t_0>0$ satisfies $E(t_0)<0$, and
$\mathcal E_q[u_{j,t_0}]<0$ for all sufficiently large $j$.
For such an admissible test $u=u_{j,t_0}$,
$0<\int G|u|^p\,dx<\infty$, since $G>0$ a.e. and $G$ is
form-bounded. Therefore
\[
\kappa_*\le\frac{Q_h[u]}{\int G|u|^p\,dx}
=K(q)+\frac{\mathcal E_q[u]}{\int G|u|^p\,dx}<K(q).
\]
If $q=\max_\alpha\mu_\alpha$, the definition of $s_\infty$ gives
\[
n-p-(p-1)s_\infty
=\frac{(n-p)q^{p-1}(1-q)}{1-q^p}.
\]
Consequently, by \eqref{eq:superquad-b},
\[
b(s_\infty)
=H\left(\frac{1-q^p}{q^{p-1}(1-q)}-p\right)=K(q).
\]
Thus $b(s_c)=\kappa_*<K(q)=b(s_\infty)$.
Since $b$ is strictly increasing on $(h,s_{\max})$, we conclude
that $s_c<s_\infty$.
\end{proof}

For completeness, the integral criterion admits a spherical representation.
Fix an active pole $a$ and write $x=a+r\theta$, with $r>0$ and
$\theta\in\mathbb S^{n-1}$. Define
\[
N_a(r,\theta):=r\,v(a+r\theta),
\qquad S_a(r,\theta):=r^2S(a+r\theta),
\]
so that $Z_q(a+r\theta)=Hr^{-p}F_a(r,\theta)$, where
\[
F_a(r,\theta):=
|N_a|^{p-2}
\bigl(c_qS_a-(c_q-1)|N_a|^2\bigr).
\]
Indeed, this follows from $Z_q=H(P+c_qG)$ and
$G=|v|^{p-2}(S-|v|^2)$. Since $\psi(x)=r^{-h}$ and $ph=n-p$,
\[
r^{-p}\psi(x)^p\,dx
=r^{-p-ph+n-1}\,dr\,d\theta
=\frac{dr}{r}\,d\theta.
\]
Consequently, \eqref{eq:pole-integral-test} becomes
\begin{equation}\label{eq:pole-test-spherical-form}
I_a=\int_0^\infty\int_{\mathbb S^{n-1}}
\bigl(F_a(r,\theta)-1\bigr)\,d\theta\,\frac{dr}{r}.
\end{equation}
The uniform endpoint expansions $F_a=1+O(r)$ as $r\downarrow0$
and $F_a=1+O(r^{-1})$ as $r\to\infty$ ensure convergence at both
ends. At the other poles, local integrability follows from $p<n$,
as shown in Lemma~\ref{lem:pole-integral-test}. Thus the double
integral is absolutely convergent, although the inner spherical
integral need not be finite at the finitely many radii passing
through another pole.

\begin{proposition}[Strictness for two poles]
\label{prop:bipolar-integral-positive}
Let $n\ge3$, $2<p<n$, $0<q<1$, and let $\mathbf{e}$ be a unit vector.
Define
\[
F_q(x):=
\frac{|x-q\mathbf{e}|^{p-2}}{|x-\mathbf{e}|^p}
\bigl(|x|^2-2q\,x\cdot\mathbf{e}+q^{2-p}\bigr).
\]
Then
\begin{equation}\label{eq:bipolar-positive-sphere-average}
\frac1{\omega_{n-1}}\int_{\mathbb S^{n-1}}\log F_q(r\theta)\,d\theta>0,
\qquad
\frac1{\omega_{n-1}}\int_{\mathbb S^{n-1}}F_q(r\theta)\,d\theta>1
\end{equation}
for every $r>0$, with the second average allowed to be infinite.
In particular, the absolutely convergent integral
\begin{equation}\label{eq:bipolar-positive-pole-integral}
I(q):=\int_{\mathbb R^n}|x|^{-n}(F_q(x)-1)\,dx
\end{equation}
is strictly positive. Consequently, for every bipolar
configuration and every $2<p<n$, taking
$q=\max\{\mu_1,\mu_2\}$ gives
\begin{equation}\label{eq:bipolar-strict-kappa}
\kappa_*<K(q),
\qquad s_c<s_\infty.
\end{equation}
\end{proposition}

\begin{proof}
Fix $r>0$ and introduce the logarithmic spherical mean
\[
\mathcal L_r(u):=
\frac1{\omega_{n-1}}\int_{\mathbb S^{n-1}}\log|r\theta-e^u\mathbf{e}|^2\,d\theta,
\qquad u\in\mathbb R.
\]
We first show that $\mathcal L_r$ is convex. Put
$t=\theta\cdot\mathbf{e}$ and $c=\cosh(u-\log r)$. For
$u\ne\log r$, we have $c>1$. Differentiation under the integral
sign and integration by parts against the density proportional
to $(1-t^2)^{(n-3)/2}$ give
\begin{align*}
\mathcal L_r''(u)
&=\frac1{\omega_{n-1}}\int_{\mathbb S^{n-1}}\frac{1-ct}{(c-t)^2}\,d\theta\\
&=(n-2)\frac1{\omega_{n-1}}\int_{\mathbb S^{n-1}}\frac{t}{c-t}\,d\theta>0.
\end{align*}
The second identity follows by integrating the derivative of
$\frac{(1-t^2)^{(n-1)/2}}{c-t}$ over $(-1,1)$; the boundary terms
vanish since $c>1$. Strict positivity follows by pairing $t$
with $-t$. Moreover, $\mathcal L_r$ and $\mathcal L_r'$ extend
continuously across $u=\log r$, with $\mathcal L_r'(\log r)=1$.
Consequently, $\mathcal L_r$ is convex on all of $\mathbb R$.

Let $a=(p-2)/2>0$. The spherical mean of
$\log(C-bt)$ decreases strictly with $b\ge0$ whenever
$C>b$, by the same reflection pairing. Since $q^{-a}>q$,
it follows, also at the integrable limiting endpoint, that
\[
\frac1{\omega_{n-1}}\int_{\mathbb S^{n-1}}
\log(r^2-2qr t+q^{-2a})\,d\theta
>\mathcal L_r(-a\log q).
\]
Consequently,
\begin{align*}
\frac1{\omega_{n-1}}\int_{\mathbb S^{n-1}}\log F_q(r\theta)\,d\theta
&>a\mathcal L_r(\log q)+\mathcal L_r(-a\log q)
 -(a+1)\mathcal L_r(0)\\
&\ge0,
\end{align*}
where the last step is convexity and
$a\log q+(-a\log q)=0$. Jensen's inequality for the
exponential proves the second assertion in
\eqref{eq:bipolar-positive-sphere-average}. Logarithmic
singularities at $r=q$ or $r=1$ are integrable on the sphere;
if the mean of $F_q$ is infinite, the same assertion follows
by truncation.

To justify integration in $r$, observe that
\[
F_q(x)-1=O(|x|)\quad(x\to0),\qquad
F_q(x)-1=O(|x|^{-1})\quad(|x|\to\infty).
\]
Near $\mathbf{e}$ the only unbounded term is $O(|x-\mathbf{e}|^{-p})$,
which is locally integrable because $p<n$; near $q\mathbf{e}$,
$F_q$ is bounded. Thus \eqref{eq:bipolar-positive-pole-integral}
is absolutely convergent. Polar coordinates and
\eqref{eq:bipolar-positive-sphere-average} show that it is
strictly positive.

Finally, translate, rotate and dilate a bipolar configuration
so that its poles are $0$ and $\mathbf{e}$, with weight $q$ at $0$.
These transformations preserve the sign of the pole test
integral. The bipolar vector identities yield
\[
|v|^2=\frac{|x-q\mathbf{e}|^2}{|x|^2|x-\mathbf{e}|^2},\qquad
S-|v|^2=\frac{q(1-q)}{|x|^2|x-\mathbf{e}|^2}.
\]
Writing $\beta_q=(1-q^p)/[q^{p-1}(1-q)]$, we therefore have
\[
|x|^p\bigl(P+\beta_qG\bigr)=F_q(x).
\]
This identifies $I(q)$ with the integral in
Lemma~\ref{lem:pole-integral-test}. When $q$ is the larger
weight, that lemma proves the first inequality in
\eqref{eq:bipolar-strict-kappa}; the identity
\eqref{eq:kappa-pole-bound} then gives $s_c<s_\infty$.
\end{proof}

\subsubsection{Strict separation in multipolar parameter regimes}

\begin{proposition}[Multipolar strict-separation regimes]
\label{prop:multipolar-strictness}
\label{prop:strictness-near-upper-exponent}
\label{prop:strictness-near-vertices-uniform}
\label{prop:strictness-near-bipolar-configurations}
The inequality $s_c<s_\infty$ holds in each of the following regimes.
\begin{enumerate}[(i)]
\item Fix $n$, distinct poles, and $\mu\in\Psi$.
Then strict separation holds for $p<n$ sufficiently close to $n$.
More precisely, at a maximal-weight pole $a=a_\alpha$, put
$q=\mu_\alpha$ and
$A_\beta(p)=\mu_\beta^p+c_q(p)\mu_\beta^{p-1}(1-\mu_\beta)$. Then
\begin{equation}\label{eq:pole-test-upper-exponent-limit}
\lim_{p\uparrow n}(n-p)I_a(p)
=\omega_{n-1}
\sum_{\substack{\beta\ne\alpha\\ \mu_\beta>0}}A_\beta(n)>0.
\end{equation}

\item Fix $n\ge3$, distinct poles, and an index $\alpha$.
There exists $\lambda_0\in(0,1/2)$, depending only on $n$ and
these poles, such that strict separation holds whenever
\[
\mu_\alpha=1-\lambda,\qquad
\mu_\beta=\lambda\nu_\beta\quad(\beta\ne\alpha),\qquad
0<\lambda<\lambda_0,
\]
uniformly for $2<p<n$ and all $\nu_\beta\ge0$ with
$\sum_{\beta\ne\alpha}\nu_\beta=1$.

\item Fix $2<p<n$, $\mu\in\Psi$, a maximal-weight index $\alpha$
with $q=\mu_\alpha$, and distinct points $a,b$.
There exists $\varepsilon_0>0$ such that strict separation holds if
$a_\alpha=a$, $|a_\beta-b|<\varepsilon_0$ for $\beta\ne\alpha$,
and the active poles remain distinct. The perturbation radius
$\varepsilon_0$ can be chosen uniformly for $p$ in any fixed
compact subinterval of $(2,n)$.
Figure~\ref{fig:pole-clustering} illustrates this geometry.
\end{enumerate}
\end{proposition}

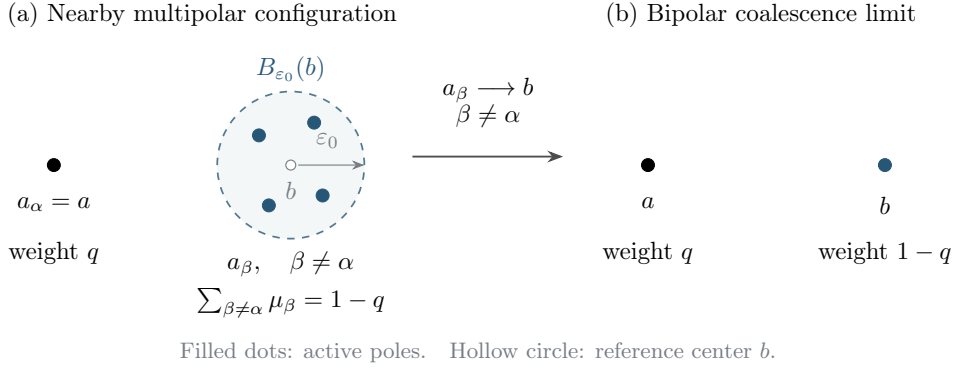
\begin{figure}[!htbp]
\centering
\definecolor{poleblue}{RGB}{39,87,119}
\definecolor{guidegray}{RGB}{119,125,132}
\resizebox{\linewidth}{!}{%
\begin{tikzpicture}[
  x=1cm,y=1cm,
  font=\small,
  >=Stealth,
  pole/.style={circle,fill=poleblue,draw=poleblue,inner sep=0pt,
    minimum size=5.2pt},
  fixedpole/.style={circle,fill=black,draw=black,inner sep=0pt,
    minimum size=5.2pt},
  guide/.style={draw=guidegray,line width=.45pt},
  heading/.style={font=\small,align=center}
]
\path[use as bounding box] (-.15,-2.80) rectangle (13.25,2.45);

\node[heading] at (2.65,2.13) {(a) Nearby multipolar configuration};
\node[heading] at (10.55,2.13) {(b) Bipolar coalescence limit};

\coordinate (a-left) at (.55,0);
\coordinate (b-center) at (3.90,0);
\draw[draw=poleblue!80,dashed,line width=.7pt,fill=poleblue!5]
  (b-center) circle[radius=1.04];
\node[text=poleblue] at (3.90,1.36) {$B_{\varepsilon_0}(b)$};

\node[fixedpole] at (a-left) {};
\node[below=9pt] at (a-left) {$a_\alpha=a$};
\node[below=27pt] at (a-left) {weight $q$};

\draw[guide,fill=white] (b-center) circle[radius=2.0pt];
\node[below=3pt,text=guidegray] at (b-center) {$b$};
\draw[guide,->,shorten <=3pt] (b-center) -- ++(1.04,0)
  node[midway,above=3pt,text=guidegray] {$\varepsilon_0$};

\node[pole] at (3.45,.42) {};
\node[pole] at (4.23,.60) {};
\node[pole] at (4.35,-.43) {};
\node[pole] at (3.59,-.57) {};
\node at (3.90,-1.40) {$a_\beta,\quad\beta\ne\alpha$};
\node at (3.90,-1.95) {$\sum_{\beta\ne\alpha}\mu_\beta=1-q$};

\draw[->,draw=black!70,line width=.75pt]
  (5.62,.12) -- (7.75,.12);
\node[align=center,above=8pt] at (6.69,.12)
  {$a_\beta\longrightarrow b$\\[-1pt]$\beta\ne\alpha$};

\coordinate (a-right) at (8.95,0);
\coordinate (b-right) at (12.30,0);
\node[fixedpole] at (a-right) {};
\node[pole] at (b-right) {};
\node[below=9pt] at (a-right) {$a$};
\node[below=27pt] at (a-right) {weight $q$};
\node[below=9pt] at (b-right) {$b$};
\node[below=27pt] at (b-right) {weight $1-q$};

\node[text=guidegray,font=\footnotesize] at (6.55,-2.65)
  {Filled dots: active poles.\quad Hollow circle: reference center $b$.};
\end{tikzpicture}
}
\caption{Schematic bipolar clustering in
Proposition~\ref{prop:multipolar-strictness}(iii).
The pole $a_\alpha=a$, of maximal weight $q=\mu_\alpha$, remains
fixed. The other active poles lie in $B_{\varepsilon_0}(b)$;
their coalescence at $b$ produces the bipolar limit with weights
$q$ and $1-q$. The hollow marker denotes the reference center
$b$, which need not be a pole before coalescence.}
\label{fig:pole-clustering}
\end{figure}

\begin{proof}
\emph{(i) Upper exponent.}
Fix $p_0\in(2,n)$ and consider $p\in[p_0,n)$.
Use disjoint fixed small balls around the active poles.
Writing $r=|x-a|$, cancellation of the leading singularity at $a$
and the far-field expansion give
\[
\bigl(P(x)+c_q(p)G(x)-r^{-p}\bigr)r^{p-n}
=\begin{cases}
O(r^{1-n}),&r\downarrow0,\\
O(r^{-n-1}),&r\to\infty,
\end{cases}
\]
with constants uniform for $p\in[p_0,n)$.
Near another active pole $a_\beta$, write
$\rho=|x-a_\beta|$ and $L_\beta=|a_\beta-a|$.  The integrand
in \eqref{eq:pole-integral-test} has the expansion
\[
A_\beta(p)L_\beta^{p-n}\rho^{-p}
+O(\rho^{1-p})+O(1),
\]
with constants uniform for $p\in[p_0,n)$.
The remainders have uniformly bounded integrals.
For a fixed small radius $\varepsilon$,
\[
(n-p)\int_{B_\varepsilon(a_\beta)}
A_\beta(p)L_\beta^{p-n}\rho^{-p}\,dx
=\omega_{n-1}A_\beta(p)
L_\beta^{p-n}\varepsilon^{n-p}
\xrightarrow{p\uparrow n}\omega_{n-1}A_\beta(n).
\]
The integral over the remaining compact set is uniformly bounded.
This proves \eqref{eq:pole-test-upper-exponent-limit}.
At least one summand is positive, so $I_a(p)>0$ near $n$.
Lemma~\ref{lem:pole-integral-test} completes the proof.

\emph{(ii) Uniform vertex neighborhood.}
Translate $a_\alpha$ to the origin and write
\[
z=\frac{x}{|x|^2},\qquad
v_\nu=\sum_{\beta\ne\alpha}\nu_\beta z_\beta,\qquad
S_\nu=\sum_{\beta\ne\alpha}\nu_\beta|z_\beta|^2.
\]
Thus
\[
v_\lambda=(1-\lambda)z+\lambda v_\nu,
\qquad S_\lambda=(1-\lambda)|z|^2+\lambda S_\nu.
\]
In this proof only, let $J_{\lambda,p,\nu}(x)$ denote the
integrand defining $I_0$ in \eqref{eq:pole-integral-test}, with
$q=1-\lambda$ and these weights.
Extend $c_{1-\lambda}(p)$ smoothly to $\lambda=0$ by
$c_1(p)=p$.  Direct differentiation at $\lambda=0$ yields
\begin{equation}\label{eq:vertex-first-variation-density}
\left.\partial_\lambda J_{\lambda,p,\nu}(x)\right|_{\lambda=0}
=D_{p,\nu}(x):=
p|x|^{2-n}\bigl(S_\nu-z\cdot v_\nu\bigr).
\end{equation}
Moreover,
\[
\operatorname{div}\bigl(|x|^{2-n}v_\nu\bigr)
=(n-2)|x|^{2-n}(S_\nu-z\cdot v_\nu).
\]
To justify integration of this identity, note that
\[
D_{p,\nu}(x)=
\begin{cases}
O(|x|^{1-n}),&x\to0,\\
O(|x-a_\beta|^{-2}),&x\to a_\beta,\quad\beta\ne\alpha,\\
O(|x|^{-n-1}),&|x|\to\infty.
\end{cases}
\]
Since $n>2$, these bounds imply $D_{p,\nu}\in L^1(\mathbb R^n)$.
Apply the divergence theorem to $B_R(0)$ with disjoint
$\varepsilon$-balls around the origin and the other poles removed.
The inner boundary fluxes of $|x|^{2-n}v_\nu$ are
$O(\varepsilon)$ at the origin and $O(\varepsilon^{n-2})$
at the other poles, and therefore vanish as $\varepsilon\downarrow0$.
Moreover, $\sum_{\beta\ne\alpha}\nu_\beta=1$ gives
$v_\nu(x)=x/|x|^2+O(|x|^{-2})$ at infinity, so
\[
\int_{\partial B_R(0)}
|x|^{2-n}v_\nu(x)\cdot\frac{x}{R}\,d\sigma
=\omega_{n-1}+O(R^{-1}).
\]
Letting first $\varepsilon\downarrow0$ and then $R\to\infty$ yields
\begin{equation}\label{eq:vertex-first-variation-integral}
\int_{\mathbb R^n}D_{p,\nu}\,dx
=\frac{p\omega_{n-1}}{n-2}.
\end{equation}
Thus the nonzero integral comes entirely from the flux at infinity
and is independent of the remaining pole positions and weight distribution.

To obtain a lower bound uniform for $2<p<n$, we apply a Taylor
estimate only outside fixed neighborhoods of the other poles
and use a one-sided estimate inside them.
Let
\[
E_\varepsilon
:=
\bigcup_{\beta\ne\alpha}B_\varepsilon(a_\beta),
\]
where $\varepsilon$ is small enough that these balls are disjoint
and their closures avoid the origin.
On $\mathbb R^n\setminus(E_\varepsilon\cup\{0\})$, we have
$|v_\nu|\le M_\varepsilon|z|$, uniformly over the closed simplex
of $\nu$. For
$0\le\lambda\le\lambda_\varepsilon:=[2(1+M_\varepsilon)]^{-1}$,
\[
|v_\lambda|
\ge (1-\lambda)|z|-\lambda|v_\nu|
\ge\bigl[1-\lambda(1+M_\varepsilon)\bigr]|z|
\ge\tfrac12|z|.
\]
Taylor expansion then gives, uniformly for $p\in[2,n]$ and
$\nu$ in that simplex,
\begin{equation}\label{eq:vertex-uniform-exterior-expansion}
\int_{\mathbb R^n\setminus E_\varepsilon}
\left|J_{\lambda,p,\nu}-\lambda D_{p,\nu}\right|\,dx
\le C_\varepsilon\lambda^2.
\end{equation}
For completeness, fix a compact set
$K\subset\mathbb R^n\setminus(E_\varepsilon\cup\{0\})$.
The preceding lower bound keeps $|v_\lambda|$ uniformly away
from zero on $K$. Since $v_\lambda$ and $S_\lambda$ depend
affinely on $\lambda$, and the extension $c_1(p)=p$ makes
$c_{1-\lambda}(p)$ jointly smooth in $(\lambda,p)$ on
$[0,\lambda_\varepsilon]\times[2,n]$, direct differentiation yields
\[
\sup_{\substack{x\in K,\ p\in[2,n],\ 0\le\lambda\le\lambda_\varepsilon\\
\nu_\beta\ge0,\ \sum_{\beta\ne\alpha}\nu_\beta=1}}
\bigl|\partial_\lambda^2J_{\lambda,p,\nu}(x)\bigr|
\le C_{K,\varepsilon}.
\]
This controls compact sets away from the origin; the remaining
regions require the following weighted estimates.
Put $r=|x|$ and $\theta=x/r$. The scaled fields satisfy
\[
(rv_\lambda,r^2S_\lambda)=
\begin{cases}
((1-\lambda)\theta,1-\lambda)+O(r),&r\downarrow0,\\
(\theta,1)+O(r^{-1}),&r\to\infty,
\end{cases}
\]
uniformly together with their first two $\lambda$-derivatives.
After extracting the factor $r^{-n}$ from $J_{\lambda,p,\nu}$,
the limiting expression at the origin vanishes by
\[
(1-\lambda)^p+
c_{1-\lambda}(p)(1-\lambda)^{p-1}\lambda=1
\]
and the limiting expression at infinity vanishes by total-mass
normalization. Both cancellations hold identically in $\lambda$.
Since $|rv_\lambda|\ge1/2$, smooth dependence of the rescaled
integrand therefore gives, for fixed sufficiently small $r_0$
and sufficiently large $R$,
\[
\bigl|\partial_\lambda^2J_{\lambda,p,\nu}(x)\bigr|
\le C
\begin{cases}
r^{1-n},&0<r<r_0,\\
r^{-n-1},&r>R.
\end{cases}
\]
Together with the compact-set bound, these integrable estimates
justify \eqref{eq:vertex-uniform-exterior-expansion} for each
fixed $\varepsilon$, uniformly in $p\in[2,n]$, $\nu$, and
$0\le\lambda\le\lambda_\varepsilon$.
No uniform Taylor expansion of the full integral $I_0$ is asserted:
the excluded neighborhoods may contribute unbounded positive
mass as $p\uparrow n$.

The omitted first variation satisfies
\[
\int_{E_\varepsilon}|D_{p,\nu}|\,dx
\le C\varepsilon^{n-2},
\]
uniformly for $2\le p\le n$ and $\nu$.
Inside $E_\varepsilon$, we need only a lower bound.
Convexity of $|\cdot|^p$, together with
$G_\lambda\ge0$ and $c_{1-\lambda}(p)>0$, gives the lower bound
\[
J_{\lambda,p,\nu}(x)
\ge
p\lambda |x|^{2-n}z\cdot(v_\nu-z).
\]
Because the right-hand side has at worst an inverse-distance
singularity at the other poles,
\begin{equation}\label{eq:vertex-uniform-interior-bound}
\int_{E_\varepsilon}J_{\lambda,p,\nu}\,dx
\ge-C\lambda\varepsilon^{n-1}.
\end{equation}
All constants in the last two estimates are independent of
$p\in(2,n)$ and $\nu$.

Combining \eqref{eq:vertex-first-variation-integral}--
\eqref{eq:vertex-uniform-interior-bound}, we obtain
\[
\frac{I_0(p,\lambda,\nu)}{\lambda}
\ge
\frac{p\omega_{n-1}}{n-2}
-C\varepsilon^{n-2}
-C\varepsilon^{n-1}
-C_\varepsilon\lambda.
\]
First let $\lambda\downarrow0$ and then
$\varepsilon\downarrow0$.  This proves
\[
\liminf_{\lambda\downarrow0}
\inf_{\substack{2<p<n\\
\nu_\beta\ge0,\ \sum_{\beta\ne\alpha}\nu_\beta=1}}
\frac{I_0(p,\lambda,\nu)}{\lambda}
\ge\frac{2\omega_{n-1}}{n-2}>0.
\]
Hence $I_0>0$ in one vertex neighborhood, uniformly in $p$
and $\nu$.  Taking $\lambda_0<1/2$ makes $1-\lambda$ the
maximal weight, so Lemma~\ref{lem:pole-integral-test} applies.

\emph{(iii) Bipolar clustering.}
As the other poles converge to $b$, the fields converge locally
uniformly away from $\{a,b\}$ to the bipolar fields with
weights $(q,1-q)$ at $(a,b)$.
The selected-pole and far-field remainder estimates in
Lemma~\ref{lem:pole-integral-test} are uniform during this
coalescence.  It remains to control a small ball around $b$.
By H\"older's inequality,
\[
|v|^{p-2}S
\le\sum_\beta\mu_\beta|z_\beta|^p,
\]
and, since $c_q>1$,
\[
0\le P+c_qG
=c_q|v|^{p-2}S-(c_q-1)|v|^p
\le c_q\sum_\beta\mu_\beta|z_\beta|^p.
\]
For all poles in the cluster contained in $B_{\delta/2}(b)$,
the integral of the last expression over $B_\delta(b)$
is bounded by $C\delta^{n-p}$, up to the bounded contribution
of the pole at $a$.  The factor $|x-a|^{p-n}$ is bounded on
this ball.  The resulting uniform integrability, combined with
local convergence away from the two limiting points, gives
\[
I_a\longrightarrow I_{\mathrm{bip}}(n,p,q).
\]
Translation and dilation do not change this integral.
Proposition~\ref{prop:bipolar-integral-positive} gives
$I_{\mathrm{bip}}(n,p,q)>0$ for every $q\in(0,1)$;
no ordering of the two limiting weights is needed.
Lemma~\ref{lem:pole-integral-test} proves the assertion.

If $p$ ranges in a compact interval $J\Subset(2,n)$,
the same estimates are uniform, with a positive exponent
$n-\max J$ in the small-ball bound.
The bipolar integral is continuous and strictly positive on
$J$, so its minimum there is positive.
This proves the final uniformity statement.
\end{proof}

\subsection{An open pole-concentration layer}
\label{subsec:pole-layer}

Cancellation along a pole's Hardy constraint line gives another
open non-attainment region when its weight is sufficiently large.

\begin{proposition}[A pole-concentration layer]
\label{prop:pole-concentration-layer}
Fix $2<p<n$, a genuinely multipolar weight vector and distinct
active poles. Suppose that
\[
\begin{gathered}
q:=\max_\alpha\mu_\alpha\ge\frac{p-1}{p},
\qquad s_0:=\frac hq\le s_{\max},\\
D_q:=G-\frac{1-q}{q}P,\qquad P=|v|^p.
\end{gathered}
\]
There exist $\varepsilon,c>0$ such that, for
\[
s\in J:=(s_0-\varepsilon,s_0+\varepsilon)
\cap[h,s_{\max}],
\]
one has
\begin{equation}\label{eq:pole-layer-remainder}
\int|\nabla u|^p\,dx-\frac{H}{g(sq)}\int W_s|u|^p\,dx
\ge c\int|D_q||u|^p\,dx
\qquad(u\in D^{1,p}(\mathbb R^n)).
\end{equation}
Consequently,
\[
C_*(p,s,\mu)=\Lambda(p,s,\mu)=\frac{H}{g(sq)}
\qquad(s\in J),
\]
the optimum is not attained, and the optimally rescaled operator
is subcritical. If $q>(p-1)/p$, the interval contains parameters
on both sides of $s_0$. If $q=(p-1)/p$, then $s_0=s_{\max}$
and the assertion is one-sided within the admissible range.
\end{proposition}

\begin{proof}
The largest weight is unique, since $q>1/2$. Translate its pole
to the origin and set $\phi=\phi_{s_0}$. For
$w\in C_c^\infty(\Omega)$, we first prove
\begin{equation}\label{eq:pole-layer-anchor}
\int|D_q|\phi^p|w|^p\,dx
\le C\left(\int\phi^p|\nabla w|^p\,dx
+\int R_{s_0}\phi^p|w|^p\,dx\right).
\end{equation}
At the origin, the leading singularities in $D_q$ cancel:
\[
S-\frac{|v|^2}{q}=O(r^{-1}),\qquad
|D_q|=O(r^{1-p}),\qquad \phi^p\asymp r^{-(n-p)}.
\]
Thus $|D_q|\phi^p\,dx\le C\,dr\,d\theta$, whereas
the radial gradient term has weight $r^{p-1}$. If $f(R)=0$,
H\"older's inequality gives
\[
|f(r)|^p\le\left(\log\frac Rr\right)^{p-1}
\int_r^R t^{p-1}|f'(t)|^p\,dt.
\]
Integrating and changing the order of integration yields
\begin{equation}\label{eq:critical-pole-lower-order-anchor}
\int_0^R|f(r)|^p\,dr
\le R\Gamma(p)\int_0^Rr^{p-1}|f'(r)|^p\,dr.
\end{equation}
Apply this estimate to radial slices of $f=\chi w$, where
$\chi$ is one near the origin and vanishes before radius $R$.
The derivative of $\chi$ contributes only on a fixed annulus.

At every other active pole, with $\gamma_\alpha=s_0\mu_\alpha<h$,
\[
|D_q|\le C r_\alpha^{-p},\qquad
\phi^p\asymp r_\alpha^{-p\gamma_\alpha}.
\]
The weighted Hardy inequality applies because
$n-p-p\gamma_\alpha>0$. At infinity,
\[
|D_q|\le Cr^{-p},\qquad \phi^p\asymp r^{-ps_0},
\qquad n-p-ps_0<0.
\]
Both regions use the weighted Hardy inequality
\[
\int_0^\infty r^{b-1}|f|^p\,dr
\le\left(\frac p{|b|}\right)^p
\int_0^\infty r^{b+p-1}|f'|^p\,dr,
\qquad b\ne0,
\]
for functions compactly supported away from the endpoints,
by integration by parts and H\"older. Local and exterior
cutoffs leave only fixed-annulus errors.

Choose a connected bounded Lipschitz set $U\Subset\Omega$
containing these annuli and a ball $B\Subset U$ on which
$R_{s_0}\ge c_B>0$. Such a ball exists because $p>2$ and at
least two poles are active. On $U$, the weight $\phi$ is bounded
above and below, and anchored Poincar\'e controls all remaining
terms by the right-hand side of \eqref{eq:pole-layer-anchor}.
This proves that estimate.

For $u=\phi w$, the ground-state representation and the
superquadratic remainder estimate now imply
\begin{equation}\label{eq:pole-layer-base-improvement}
\int|\nabla u|^p\,dx-\int W_{s_0}|u|^p\,dx
\ge c_0\int|D_q||u|^p\,dx.
\end{equation}
Initially this holds on $C_c^\infty(\Omega)$. It extends to
$D^{1,p}$ by density: the poles have zero $p$-capacity, and
$|D_q|\le G+(1-q)P/q$ is form-bounded by the Hardy inequality
for $W_h=HP+pHG$.

Write
\[
Z_s:=\frac H{g(sq)}W_s=A_q(s)P+B_q(s)G,
\qquad
B_q(s)=\frac{H(n-p)}
{q^{p-1}[n-p-(p-1)qs]}.
\]
The identity
$q^pA_q(s)+q^{p-1}(1-q)B_q(s)=H$ and $g(s_0q)=H$
give
\begin{equation}\label{eq:pole-line-perturbation}
Z_s-W_{s_0}=[B_q(s)-B_q(s_0)]D_q.
\end{equation}
The denominator defining $B_q$ stays positive on the admissible
range. Hence continuity and \eqref{eq:pole-layer-base-improvement}
prove \eqref{eq:pole-layer-remainder}, after reducing
$\varepsilon$ so that $|B_q(s)-B_q(s_0)|<c_0/2$.

For any other weight $t\le q$, we have $t+q\le1$ and
$q(1-q)-t(1-t)=(q-t)(1-q-t)\ge0$. Since $p>2$, this gives
$q^{p-1}(1-q)\ge t^{p-1}(1-t)$.
The quantity $[g(sq)-g(st)]/s^{p-1}$ is affine decreasing
in $s$ and remains nonnegative at $s=s_{\max}$. Thus the
largest weight determines the smallest pole threshold throughout
the admissible range. Since $s_\infty<s_0$, we may further
shrink $J$ so that
$\Lambda=H/g(sq)$ there. The concentration upper bound at the
maximal pole and \eqref{eq:pole-layer-remainder} give the
claimed value of $C_*$.

Finally, $|D_q|>0$ a.e.: its possible zeros are contained in
the zero sets of the nontrivial real-analytic functions $v$
and $S-|v|^2/q$, the latter being negative at infinity.
Thus \eqref{eq:pole-layer-remainder} excludes a nonzero
finite-energy extremizer. Its nontrivial nonnegative remainder
also proves subcriticality.
\end{proof}

\paragraph{Consequences and remaining questions.}
The superquadratic picture separates three mechanisms:
an infinity-concentration interval inside the full-potential
attainment region, a non-attained critical endpoint followed
by a compact phase when $s_c<s_\infty$, and a subcritical
pole-concentration layer for large maximal weight.
Theorem~\ref{prop:compactness-threshold} characterizes attainment
by the strict gap for $p\ge2$.

Universal strict separation, criticality at a possible coincident
endpoint, and all boundaries of pole-dominated phases remain open.
In particular, no general sign assertion for $I_a$ or its spherical
integrand is used; a negative test integral alone would not imply
$s_c=s_\infty$. For $1<p<2$, the general closed coefficient formula
and classification beyond the coexistence paths also remain open.

\raggedbottom

\end{document}